\documentclass[a4paper,10pt]{amsart}

\usepackage{amssymb,amsthm,amsmath,amscd,amsfonts,bbm,mathrsfs}

\usepackage[cmtip,all]{xy}

\theoremstyle{plain}
\newtheorem{Lemma}{Lemma}
\newtheorem{Thm}[Lemma]{Theorem}
\newtheorem{Prop}[Lemma]{Proposition}
\newtheorem{Cor}[Lemma]{Corollary}
\theoremstyle{definition}
\newtheorem{Defn}[Lemma]{Definition}
\newtheorem{Not}[Lemma]{Notation}
\newtheorem{Example}[Lemma]{Example}
\newtheorem{Constr}[Lemma]{Construction}
\theoremstyle{remark}
\newtheorem{Remark}[Lemma]{Remark}

\numberwithin{equation}{section}
\numberwithin{enumi}{section}

\newcommand{\BBB}{\mathscr{B}}
\newcommand{\CCC}{\mathcal{C}}
\newcommand{\DDD}{\mathscr{D}}
\newcommand{\EEE}{\mathscr{E}}
\newcommand{\FFF}{\mathcal{F}}
\newcommand{\GGG}{\mathscr{G}}
\newcommand{\JJJ}{\mathcal{J}}
\newcommand{\III}{\mathcal{I}}
\newcommand{\MMM}{\mathcal{M}}
\newcommand{\NNN}{\mathcal{N}}
\newcommand{\OOO}{\mathcal{O}}
\newcommand{\PPP}{\mathscr{P}}
\newcommand{\SSS}{\mathscr{S}}
\newcommand{\TTT}{\mathscr{T}}
\newcommand{\WWW}{\mathscr{W}}

\newcommand{\Fa}{\mathfrak{a}}
\newcommand{\Fb}{\mathfrak{b}}
\newcommand{\Fc}{\mathfrak{c}}
\newcommand{\Fm}{\mathfrak{m}}
\newcommand{\Fp}{\mathfrak{p}}
\newcommand{\FS}{\mathfrak{S}}

\newcommand{\CC}{{\mathbb{C}}}
\newcommand{\DD}{{\mathbb{D}}}
\newcommand{\FF}{{\mathbb{F}}}
\newcommand{\GG}{{\mathbb{G}}}
\newcommand{\II}{{\mathbb{I}}}
\newcommand{\JJ}{{\mathbb{J}}}
\newcommand{\KK}{{\mathbb{K}}}
\newcommand{\NN}{{\mathbb{N}}}
\newcommand{\QQ}{{\mathbb{Q}}}
\newcommand{\WW}{{\mathbb{W}}}
\newcommand{\ZZ}{{\mathbb{Z}}}

\newcommand{\ff}{{\mathbbm{f}}}
\newcommand{\vv}{{\mathbbm{v}}}
\newcommand{\uu}{{\mathbbm{u}}}

\DeclareMathOperator{\Art}{Art}
\DeclareMathOperator{\Bil}{Bil}
\DeclareMathOperator{\BT}{BT}

\DeclareMathOperator{\Cris}{Cris}

\DeclareMathOperator{\Def}{Def}
\DeclareMathOperator{\End}{End}

\DeclareMathOperator{\Hom}{Hom}

\DeclareMathOperator{\Ker}{Ker}
\DeclareMathOperator{\Lie}{Lie}
\DeclareMathOperator{\Log}{Log}
\DeclareMathOperator{\Nil}{Nil}

\DeclareMathOperator{\uExt}{\underline{Ext}}
\DeclareMathOperator{\Rad}{Rad}
\DeclareMathOperator{\Spec}{Spec}
\DeclareMathOperator{\Tor}{Tor}
\DeclareMathOperator{\Sym}{Sym}
\DeclareMathOperator{\V}{V}

\DeclareMathOperator{\adm}{adm}
\DeclareMathOperator{\divv}{div}
\DeclareMathOperator{\fin}{fin}

\DeclareMathOperator{\red}{red}

\newcommand{\ArtC}{({\Art})}

\newcommand{\pfinC}{(p\text{-}{\fin})}
\newcommand{\pdivC}{(p\text{-}{\divv})}

\newcommand{\Kideal}{U}

\newcommand{\Stumm}[1]{}

\begin{document}

\title[Dieudonn\'e displays and crystalline Dieudonn\'e theory]
{Relations between Dieudonn\'e displays
and crystalline Dieudonn\'e theory}
\author{Eike Lau}
%\date{\today}
\address{Institut f\"{u}r Mathematik,
Universit\"{a}t Paderborn, D-33098 Paderborn}
\email{elau@math.upb.de}
% \subjclass[2000]{Primary 14L05; Secondary 14F30}

\begin{abstract}
We discuss the relation between crystalline Dieu\-donn\'e
theory and Dieudonn\'e displays of $p$-divisible groups. 
The theory of Dieu\-donn\'e displays
is extended to the prime 2 without restriction, which
implies that the classification of finite locally free group
schemes by Breuil-Kisin modules holds for the prime 2 as well.
\end{abstract}

\maketitle

\section*{Introduction}

\renewcommand{\theLemma}{\Alph{Lemma}}

Formal $p$-divisible groups $G$ over a $p$-adically complete ring $R$
are classified by Zink's nilpotent displays \cite{Zink-Disp, Lau-Disp}.
These are projective modules over 
the ring of Witt vectors $W(R)$ equipped with a filtration 
and with certain Frobenius-linear operators.
A central point of the theory is a description of the Dieudonn\'e crystal
of $G$ in terms of the nilpotent display associated to $G$.

Arbitrary $p$-divisible groups over $R$ can
be classified by displays only when $R$ is a perfect ring.
In certain cases, there is the following refinement.

Assume that $R$ is a local Artin ring with perfect residue
field $k$ of characteristic $p$ and with maximal ideal $\NNN_R$. 
Then $W(R)$ has a unique subring $\WW(R)$,
called here the Zink ring of $R$, which is stable under
the Frobenius and which sits in an exact sequence
$$
0\to\hat W(\NNN_R)\to\WW(R)\to W(k)\to 0
$$
where $\hat W$ means Witt vectors
with only finitely many non-zero components.
Let us call $R$ \emph{odd} if $p>2$ or if $p$ annihilates $R$.
The Verschiebung homomorphism $v$ of $W(R)$,
which appears in the definition of displays, stabilises
the subring $\WW(R)$ if and only if $R$ is odd.
In this case, Zink \cite{Zink-DDisp} defines Dieudonn\'e displays over $R$
as displays with $\WW(R)$ in place of $W(R)$ and shows that 
they classify all $p$-divisible groups over $R$.

The restriction for $p=2$ can be avoided with a small trick:
The ring $\WW(R)$ is always stable under the modified
Verschiebung $\vv(x)=v(u_0x)$, where $u_0\in W(R)$ is
the unit defined by the relation $v(u_0)=p-[p]$. This allows
to define Dieudonn\'e displays without assuming that $R$ is odd.
It turns out that the Zink ring
and Dieudonn\'e displays can be defined for the following 
class of rings $R$, which we call \emph{admissible}: 
The order of nilpotence of nilpotent
elements of $R$ is bounded, and $R_{\red}$
is a perfect ring of characteristic $p$.

\begin{Thm}
\label{Th-A}
For each admissible ring $R$ there is a functor
\[
\Phi_R:(p\text{-divisible groups over }R)\to(\text{Dieudonn\'e displays over }R),
\]
which is an equivalence of exact categories.
\end{Thm}

The equivalence easily extends to projective limits of admissible rings,
which includes complete local rings with perfect residue field.
If $R$ is perfect, the theorem says that
$p$-divisible groups over $R$ are equivalent to Dieudonn\'e
modules. This is a result of Gabber, which is used in the proof.
We repeat that for Artin rings (which is certainly the case of interest for
most applications%
\footnote{In subsequent work, Dieudonn\'e displays over a larger
class of base rings will be used to study the image of the crystalline Dieudonn\'e functor 
over l.c.i.\ schemes.}%
), Theorem \ref{Th-A}  is known when $R$ is odd;
in this case $\Phi_R$ is the inverse of the functor $BT$ of \cite{Zink-DDisp} and \cite{Lau-Dual}.
But the present construction of  the functor $\Phi_R$ based on the 
crystalline Dieudonn\'e module is new, and also gives the following
second result.

Let $\DD(G)$ denote the covariant Dieudonn\'e crystal of a $p$-divisible
group $G$.
Following \cite{Zink-Windows}, to a Dieudonn\'e display $\PPP$ over an
admissible ring $R$ one can associate a crystal in locally free
modules $\DD(\PPP)$. 

\begin{Thm}
\label{Th-B}
For a $p$-divisible group $G$ over an admissible ring $R$ with associated 
Dieudonn\'e display $\PPP=\Phi_R(G)$ there is a natural isomorphism
\[
\DD(G)\cong\DD(\PPP).
\]
\end{Thm}

This compatibility was not known before and can be useful in applications; 
see for example \cite{Viehmann-Wedhorn}.
Our proofs of Theorems \ref{Th-A} and \ref{Th-B} are closely related.
The main point is to construct the functor $\Phi_R$ and variants of it.
Let $\II_R$ be the kernel of the natural homomorphism $\WW(R)\to R$.

First, if $R$ is an odd admissible ring, 
the ideal $\II_R$ carries natural divided powers.
Thus the crystalline Dieudonn\'e module
of a $p$-divisible group over $R$ can be evaluated at $\WW(R)$,
which gives a filtered $F$-$V$-module over $\WW(R)$.
We show that this construction can be extended to a functor $\Phi_R$ 
as in Theorem \ref{Th-A}. This is not evident because
a filtered $F$-$V$-module does not in general determine 
a Dieudonn\'e display. 
But the construction of $\Phi_R$ can be reduced to the
case where $R$ is a universal deformation ring;
then the Dieudonn\'e display is determined uniquely 
because $p$ is not a zero divisor in $\WW(R)$.

Next, for a divided power extension of admissible rings $S\to R$ one can
define Dieudonn\'e displays relative to $S\to R$, called triples in the
work of Zink. They are modules over an extension $\WW(S/R)$ of $\WW(S)$.
If $R$ is odd and the divided powers are compatible with the canonical 
divided powers of $p$, then the evaluation of the crystalline Dieudonn\'e module
at the divided power extension $\WW(S/R)\to R$ can be extended to a functor
\[
\Phi_{S/R}:(\text{$p$-divisible groups over $R$})
\to(\text{Dieudonn\'e displays for $S/R$}).
\]
Again this is not evident; the proof comes down to the fact
that $p$ is not a zero divisor in the Zink ring of the divided power envelope of the
diagonal of the universal deformation space of a $p$-divisible group.
Once the functors $\Phi_{S/R}$ are known to exist, Theorems
\ref{Th-A} and \ref{Th-B} for odd admissible rings are straightforward consequences.

Let now $R$ be an admissible ring which is not odd, so $p=2$.
In this case the preceding constructions do not apply directly because the ideal $\II_R$
does not in general carry divided powers. This changes 
when $\WW(R)$ is replaced by the slightly larger
$v$-stabilised Zink ring $\WW^+(R)=\WW(R)[v(1)]$.
With an obvious definition of $v$-stabilised Dieudonn\'e displays we get a functor
\[
\Phi_R^+:(2\text{-div.\ groups over }R)\to(v\text{-stabilised Dieudonn\'e displays over }R),
\]
which is however not an equivalence.
In order to construct a functor $\Phi_R$ as in Theorem \ref{Th-A}
we have to descend from $\WW ^+(R)$ to $\WW(R)$. 
This can be reduced to the minimal case where $2\NNN_R=0$.
Then the ideal $\II_R$ carries exceptional divided powers,
which allows to evaluate the crystalline Dieudonn\'e module at $\WW(R)$.
In order to get the functor $\Phi_R$ we need some lift towards
characteristic zero, which is provided by the fact that the exceptional 
divided powers exist on $\II_R/(v([4]))$ as soon as $4\NNN_R=0$.
Once $\Phi_R$ is known to exist in general, Theorems \ref{Th-A}
and \ref{Th-B} follow again quite formally.

\subsubsection*{Breuil-Kisin modules}

Let now $R$ be a complete regular local ring with
perfect residue field $k$ of characteristic $p$. 
Theorem \ref{Th-A} implies that the classification of
$p$-divisible groups over $R$ by Breuil windows 
derived in \cite{Vasiu-Zink} and \cite{Lau-Frames} 
for odd $p$ holds for $p=2$ as well.
Let us recall what this means:
We write $R=\FS/E\FS$ where $\FS$ 
is a power series ring over $W(k)$ and where $E$ 
has constant term $p$; we also have to choose an
appropriate Frobenius lift $\sigma$ on $\FS$.
A Breuil window is a free $\FS$-module $Q$ equipped with
an $\FS$-linear map $\phi:Q\to Q^{(\sigma)}$ whose
cokernel is annihilated by $E$; this is equivalent to the
notion of a Breuil-Kisin module. As usual one also gets
a classification of finite locally free $p$-group schemes over $R$.

In the case of discrete valuation rings this completes the proof 
of a conjecture of Breuil \cite{Breuil}, which was proved by 
Kisin in \cite{Kisin-crys} if $p$ is odd, and in \cite{Kisin-2adic}
for connected $p$-divisible groups if $p=2$.x
Shortly after the first version of this article was posted, 
independent proofs of Breuil's conjecture 
by W.~Kim \cite{Kim} and T.~Liu \cite{Liu} appeared online.

Assume that $R$ has characteristic zero, and let
$S$ be the $p$-adic completion of the divided power envelope of the ideal
$E\FS\subset\FS$. As a consequence of Theorem \ref{Th-B} we show that
for a $p$-divisible group over $R$ the value of
its crystalline Dieudonn\'e module at $S$ coincides with the base change
of its Breuil window under $\sigma:\FS\to S$.

\subsubsection*{The functor $BT$}

The original proof of Theorem \ref{Th-A} for odd local Artin rings
in \cite{Zink-DDisp} depends on the construction
of a functor $BT$ from Dieudonn\'e displays to $p$-divisible groups,
which is a combination of the functor $BT$ from nilpotent displays
to formal $p$-divisible groups and a calculation of extensions.
A modified construction of this functor was given in \cite{Lau-Dual}.
Once the definition of Dieudonn\'e displays for non-odd local Artin
rings is available, all these arguments can be carried over almost literally
to give an alternative proof of Theorem \ref{Th-A} in that case.
In the present approach this construction serves only as an
explicit description of the inverse of the functor $\Phi_R$; 
this is used in \cite{Lau-Displayed}.

\medskip

All rings are commutative with a unit unless the contrary is stated.
For a $p$-divisible group $G$ we denote by $\DD(G)$
the \emph{covariant} Dieudonn\'e crystal.

\medskip

The author thanks Xavier Caruso, Tyler Lawson, and
Thomas Zink for interesting and helpful conversations,
and the anonymous referee for a very careful reading 
of the manuscript.

\setcounter{tocdepth}{1}
\tableofcontents

%---------------------------------------------------------------

\numberwithin{Lemma}{section}
\section{The Zink ring}

In this section we study the Zink ring $\WW(R)$,
which is introduced in \cite{Zink-DDisp} under
the notation $\hat W(R)$, and variants of $\WW(R)$
in the presence of divided powers, following \cite{Zink-Windows}.
The definitions are stated in more generality, allowing 
arbitrary perfect rings instead of perfect fields.
The modified Verschiebung $\vv$ for $p=2$ is new.

\subsection{Preliminaries}

We fix a prime $p$.
A commutative ring without unit
$N$ is called bounded nilpotent if there is
a number $n$ such that $x^n=0$ 
for every $x\in N$. 
We will consider the following type of base rings.

\begin{Defn}
\label{Def-admissible}
A ring $R$ is called \emph{admissible} if its nilradical
$\NNN_R$ is bounded nilpotent and if $R_{\red}=R/\NNN_R$
is a perfect ring of characteristic $p$. 
\end{Defn}

Local Artin rings with perfect residue field are admissible.
The ring $\OOO_{\CC_p}/p$ is not admissible.
We will also consider projective limits of admissible rings:

\begin{Defn}
\label{Def-top-admissible}
An \emph{admissible topological ring} is a complete and separated 
topological ring $R$ with linear topology such that the
ideal $\NNN_R$ of topologically nilpotent elements is open,
the ring $R_{\red}=R/\NNN_R$ is perfect of characteristic $p$,
and for each open ideal $N$ of $R$ contained in $\NNN_R$
the quotient $\NNN_R/N$ is bounded nilpotent.
Thus $R$ is the projective limit of the admissible rings $R/N$.
\end{Defn}

Examples of admissible topological rings
are complete local rings with perfect
residue field. Admissible topological rings in which
$\NNN_R$ is not topologically nilpotent arise from
divided power envelopes; see Lemma \ref{Le-pd-alg}.

\begin{Not}
For a commutative, not necessarily unitary ring $A$ let $W(A)$ be
the ring of $p$-typical Witt vectors of $A$.
We write $f$ and $v$ for the Frobenius and 
Verschiebung of $W(A)$. Let $I_A=v(W(A))$, let 
$w_i:W(A)\to A$ be given by the $i$-th Witt polynomial,
and let $\hat W(A)$ be the group of all elements of 
$W(A)$ with nilpotent coefficients which are almost 
all zero.
\end{Not}

Let us recall two well-known facts.

\begin{Lemma}
\label{Le-lift-W}
Let $A$ be a perfect ring of characteristic $p$ 
and let $B$ be a ring with a bounded nilpotent ideal 
$J\subseteq B$.
Every ring homomorphism $A\to B/J$ lifts to a
unique ring homomorphism $W(A)\to B$.
\end{Lemma}

\begin{proof}
See \cite[Chap.~IV, Prop.~4.3]{Grothendieck-Montreal},
where the ideal $J$ is assumed nilpotent, 
but the proof applies here as well. 
\end{proof}

\begin{Lemma}[{\cite[Lemma 2.2]{Zink-Windows}}]
\label{Le-W(N)}
Let $N$ be a non-unitary ring which is bounded 
nilpotent and annihilated by a power of $p$. 
Then $W(N)$ is bounded 
nilpotent and annihilated by a power of $p$. 
\qed
\end{Lemma}

\subsection{The Zink ring}

Let $R$ be an admissible ring. 
By Lemma \ref{Le-lift-W} the exact sequence
$$
0\to W(\NNN_R)\to W(R)\to W(R_{\red})\to 0
$$
has a unique ring homomorphism section $s:W(R_{\red})\to W(R)$, 
which is $f$-equivariant by its uniqueness. Let
$$
\WW(R)=sW(R_{\red})\oplus\hat W(\NNN_R).
$$
Since $\hat W(\NNN_R)$ is an $f$-stable ideal of $W(R)$,
the group $\WW(R)$ is an $f$-stable subring of $W(R)$,
which we call the Zink ring of $R$. 

\begin{Lemma}
\label{Le-WW-v}
The ring\/ $\WW(R)$ is stable under the Verschiebung
homomorphism $v:W(R)\to W(R)$ if and only if
$p\ge 3$ or $pR=0$. In this case we have an
exact sequence
$$
0\to\WW(R)\xrightarrow v\WW(R)\xrightarrow{w_0} R\to 0.
$$
\end{Lemma}

\begin{proof}
See \cite[Lemma 2]{Zink-DDisp}.
For some $r\ge 0$ the ring $R_0=\ZZ/p^r\ZZ$ is a
subring of $R$, and we have $\WW(R_0)=W(R_0)\cap\WW(R)$. 
The calculation in loc.\ cit.\ shows that the element 
$v(1)\in W(R_0)$ lies in $\WW(R_0)$ if and only if 
$p\ge 3$ or $r=1$. For $a\in W(R_{\red})$ we have 
$v(s(f(a)))=v(f(s(a)))=v(1)s(a)$.
Since $\hat W(\NNN_R)$ is stable under $v$ and since
$f$ is surjective on $W(R_{\red})$, the first assertion 
of the lemma follows. The sequence is an extension of
$$
0\to W(R_{\red})\xrightarrow v W(R_{\red})\to R_{\red}\to 0
$$
and 
$$
0\to\hat W(\NNN_R)\xrightarrow v\hat W(\NNN_R)\to\NNN_R\to 0,
$$
which are both exact.
\end{proof}

With a slight modification the exception at the prime $2$
can be removed. The element $p-[p]$ of $W(\ZZ_p)$ lies 
in the image of $v$ because it maps to zero in $\ZZ_p$. 
Moreover $v^{-1}(p-[p])$ maps to $1$ in $W(\FF_p)$, so
this element is a unit in $W(\ZZ_p)$. We define
$$
u_0=
\begin{cases}
v^{-1}(2-[2]) & \text{if $p=2$,} \\
1 & \text{if $p\ge 3$.}
\end{cases}
$$
The image of $u_0$ in $W(R)$ is also denoted $u_0$.
For $x\in W(R)$ let
$$
\vv(x)=v(u_0x).
$$
One could also take $u_0=v^{-1}(p-[p])$ for all $p$, which 
would allow to state some results more uniformly, but for odd 
$p$ this would be overcomplicated.

\begin{Lemma}
\label{Le-WW-vv}
The ring\/ $\WW(R)$ is stable under\/ $\vv:W(R)\to W(R)$, 
and there is an exact sequence
$$
0\to\WW(R)\xrightarrow {\vv}\WW(R)\xrightarrow{w_0} R\to 0.
$$
\end{Lemma}

\begin{proof}
By Lemma \ref{Le-WW-v} we can assume that $p=2$.
For $a\in W(R_{\red})$ we have 
$\vv(s(f(a)))=v(u_0f(s(a)))=v(u_0)s(a)=(2-[2])s(a)$,
which lies in $\WW(R)$. Since $\hat W(\NNN_R)$ is stable
under $\vv$ and since $f$ is surjective on $W(R_{\red})$
it follows that $\WW(R)$ is stable under $\vv$.
The sequence is an extension of
$$
0\to W(R_{\red})\xrightarrow\vv W(R_{\red})\to R_{\red}\to 0
$$
and 
$$
0\to\hat W(\NNN_R)\xrightarrow\vv\hat W(\NNN_R)\to\NNN_R\to 0.
$$
They are exact because in both cases $\vv=v\circ u_0$ 
where $u_0$ acts bijectively. 
\end{proof}

\subsection{The enlarged Zink ring}
\label{Subse-enlarged-Zink}

Let us recall the logarithm of the Witt ring.
For a divided power extension of rings $(B\to R,\delta)$ 
with kernel $\Fb\subseteq B$,
the $\delta$-divided Witt polynomials define an
isomorphism of $W(B)$-modules
\begin{equation*}
%\label{Eq-Log}
\Log:W(\Fb)\cong\Fb^\NN
\end{equation*}
where $x\in W(B)$ acts on $\Fb^\NN$ by
$[b_0,b_1,\ldots]\mapsto[w_0(x)b_0,w_1(x)b_1,\ldots]$.
The Frobenius and Verschiebung of $W(\Fb)$ act on $\Fb^\NN$ 
by 
$$
f([b_0,b_1,\ldots])=[pb_1,pb_2,\ldots],
\qquad
v([b_0,b_1,\ldots])=[0,b_0,b_1,\ldots].
$$ 
Moreover $\Log$ induces an injective map 
$\hat W(\Fb)\to\Fb^{(\NN)}$,
which is bijective when the divided powers $\delta$
are nilpotent; see \cite[(149)]{Zink-Disp}  
and the subsequent discussion. In general let
$$
\tilde W(\Fb)=\Log^{-1}(\Fb^{(\NN)}).
$$
This is an $f$-stable and $v$-stable ideal of $W(B)$
containing $\hat W(\Fb)$. 

Assume now that $(B\to R,\delta)$ is a divided power
extension of admissible rings (it suffices to assume that
$R$ is admissible and that $p$ is nilpotent in $B$
because then $\Fb$ is bounded nilpotent due to the divided powers, 
so $B$ is admissible as well). Let
$$
\WW(B,\delta)=\WW(B)+\tilde W(\Fb).
$$
This is an $f$-stable subring of $W(B)$, which we call 
the enlarged Zink ring of $B$ with respect to the
divided power ideal $(\Fb,\delta)$. 
We also write $\WW(B/R)$ for $\WW(B,\delta)$.
If the divided powers $\delta$ are nilpotent then 
$\WW(B,\delta)=\WW(B)$. We have the following analogues of 
Lemmas \ref{Le-WW-vv} and \ref{Le-WW-v}.

\begin{Lemma}
\label{Le-WW-large-vv}
The ring\/ $\WW(B,\delta)$ is stable under\/ $\vv:W(R)\to W(R)$,
and there is an exact sequence
$$
0\to\WW(B,\delta)\xrightarrow\vv\WW(B,\delta)
\xrightarrow{w_0} B\to 0.
$$
\end{Lemma}

\begin{proof}
The ring $\WW(B,\delta)$ is stable under $\vv$ because
so are $\WW(B)$ and $\tilde W(\Fb)$; see Lemma \ref{Le-WW-vv}.
We have $\WW(B,\delta)/\tilde W(\Fb)=\WW(R)$.
Thus the exact sequence follows from the exactness of
$0\to\tilde W(\Fb)\xrightarrow\vv\tilde\WW(\Fb)\to\Fb\to 0$
together with the exact sequence of Lemma \ref{Le-WW-vv}.
\end{proof}

\begin{Lemma}
\label{Le-WW-large-v}
The ring\/ $\WW(B,\delta)$ is stable under $v:W(R)\to W(R)$
if $p\ge 3$, or if $p\in\Fb$ and the divided powers
$\delta$ on $\Fb$ induce the canonical divided powers 
on $pB$. In this case we have an exact sequence
$$
0\to\WW(B,\delta)\xrightarrow v\WW(B,\delta)
\xrightarrow{w_0}B\to 0.
$$
\end{Lemma}

\begin{proof}
If $p\ge 3$ then $\WW(B,\delta)$ is stable under $v$ 
because $\WW(B)$ and $\tilde W(\Fb)$ are stable under $v$;
see Lemma \ref{Le-WW-v}.
Assume that $p\in\Fb$ and that $\delta$ induces the canonical
divided powers on $pB$. Let $\xi=p-v(1)\in W(B)$.
This element lies in $W(pB)\subseteq W(\Fb)$ and
satisfies $\Log(\xi)=[p,0,0,\ldots]$. 
Thus $\xi\in\tilde W(\Fb)$, 
which implies that $v(1)\in\WW(B,\delta)$. 
Using this, the proof of Lemma \ref{Le-WW-v} 
shows that $\WW(B,\delta)$ is stable under $v$. 
The exact sequence follows as usual.
\end{proof}

\subsection{The $v$-stabilised Zink ring}
\label{Subse-v-stab-Zink}

Assume that $p=2$. For an admissible ring $R$
let $\gamma$ be the canonical divided powers
on the ideal $2R$.  
We denote the associated enlarged Zink ring by
$$
\WW^+(R)=\WW(R,\gamma)=\WW(R)+\tilde W(2R)\subseteq W(R).
$$
The kernel of the projection $\WW^+(R)\to W(R_{\red})$
will be denoted $\hat W^+(\NNN_R)$.
In view of the following lemma we call $\WW^+(R)$
the $v$-stabilised Zink ring.
 
\begin{Lemma}
\label{Le-WW+}
Let $p=2$. We have 
$$
\WW^+(R)=\WW(R)+\WW(R)v(1).
$$ 
The ring\/ $\WW^+(R)$ is equal to\/ $\WW(R)$ if and only if $2R=0$.
The\/ $\WW(R)$-module\/ $\WW^+(R)/\WW(R)$ is 
an $R_{\red}$-module generated by $v(1)$.
\end{Lemma}

\begin{proof}
By Lemma \ref{Le-WW-large-v} we have $v(1)\in\WW^+(R)$.
Clearly $2R=0$ implies that $\WW^+(R)=\WW(R)$.
In general we consider the filtration 
$$
W(2\NNN_R)\subseteq W(2R)\subseteq W(R)
$$
and the graded modules for the induced filtrations 
on $\WW(R)$ and on $\WW^+(R)$. First, the restriction of the
divided powers $\gamma$ to the ideal $2\NNN_R$ 
is nilpotent, which implies that
$$
\WW^+(R)\cap W(2\NNN_R)=\tilde W(2\NNN_R)=
\hat W(2\NNN_R)=\WW(R)\cap W(2\NNN_R).
$$
Next we have $\WW^+(R/2R)=\WW(R/2R)$, or equivalently
$$
\WW^+(R)/\WW^+(R)\cap W(2R)=\WW(R)/\WW(R)\cap W(2R).
$$
Let $\Fc=2R/2\NNN_R$. 
By the preceding remarks we have an isomorphism
$$
\WW^+(R)/\WW(R)\cong\tilde W(\Fc)/\hat W(\Fc).
$$
This is an $R/\NNN_R$-module. Assume that $2R\ne 0$, 
which implies that $\Fc\ne 0$.
For some ideal $\NNN_R\subseteq\Fb\subseteq R$,
multiplication by $2$ induces an isomorphism $R/\Fb\cong\Fc$.
Modulo $2$ the divided Witt polynomials are 
$\tilde w_i(x)\equiv\gamma_2(x_{i-1})+x_i$,
so the isomorphism $\Log:W(\Fc)\to\Fc^\NN$
takes the form
$$
\Log(2a_0,2a_1,\ldots)=2[a_0,a_0^2+a_1,a_1^2+a_2,a_2^2+a_3,\ldots]
$$
with $a_i\in R/\Fb$.
It follows that $\tilde W(\Fc)/\hat W(\Fc)$ can be
identified with the direct limit of the Frobenius homomorphism
$R/\Fb\to R/\Fb\to\ldots\,$, which is isomorphic to 
$R/\sqrt\Fb$. Under this identification, the element 
$\xi=2-v(1)$ of $\tilde W(\Fc)$ maps to $1$ in $R/\sqrt\Fb$
because we have $\Log(\xi)=[2,0,\ldots]$.
Hence $\WW^+(R)/\WW(R)$ is generated by $v(1)$, 
with annihilator $\sqrt\Fb$.
\end{proof}

Assume again that $p=2$.
Let $(B\to R,\delta)$ be a divided power
extension of admissible rings with kernel $\Fb\subseteq B$
such that $\delta$ is compatible with the canonical divided powers
$\gamma$ on $2B$. Let $\delta^+$
be the divided powers on $\Fb^+=\Fb+2B$ that extend $\delta$
and $\gamma$. In this case we write
$$
\WW^+(B,\delta)=\WW(B,\delta^+)=\WW(B)+\tilde W(\Fb^+).
$$
Clearly $\WW(B,\delta)\subseteq\WW^+(B,\delta)\supseteq\WW^+(B)$.
If the divided powers on $\Fb^+/2B$ induced by $\delta$ 
are nilpotent then $\WW^+(B,\delta)=\WW^+(B)$.

\subsection{Passing to the limit}

The preceding considerations carry over to the topological case
as follows. For an admissible topological ring $R$ let
$$
\WW(R)=\varprojlim_N\WW(R/N),
$$ 
the limit taken over all open ideals $N$ of $R$ with 
$N\subseteq\NNN_R$. Then Lemmas \ref{Le-WW-v} and 
\ref{Le-WW-vv} hold for admissible topological rings. 
The enlarged Zink ring can be defined for topological 
divided power extensions in the following sense.

\begin{Defn}
\label{Def-pd-top-admissible}
Let $B$ and $R$ be admissible topological rings.
A topological divided power extension is a surjective ring
homomorphism $B\to R$ whose kernel $\Fb$ is equipped with 
divided powers $\delta$ such that $\Fb$ is closed in $B$, 
the topology of $R$ is the quotient topology of $B/\Fb$, 
and the linear topology of $B$ is induced by open ideals 
$N$ for which $N\cap\Fb$ is stable under $\delta$.
Let $\delta/N$ be the divided powers
on $N/N\cap\Fb$ induced by $\delta$.
We say that $\delta$ is topologically compatible with the
canonical divided powers of $p$ 
if the topology of $B$ is induced by open ideals $N$
such that $\delta/N$ is defined and compatible with the 
canonical divided powers of $p$.
\end{Defn}

\begin{Remark}
The existence of divided powers on $\Fb$ implies that
$\Fb\subseteq\NNN_B$. If $B$ is a noetherian complete
local ring then every ideal $\Fb$ of $B$ is closed, moreover
if $\Fb$ is given, for each $n$ there is an open ideal $N\subseteq\Fm_B^n$ 
such that $\Fb\cap N$ is stable under arbitrary divided powers 
$\delta$ on $\Fb$. Indeed, by Artin-Rees there is an $l$ with 
$\Fm_B^n\Fb\supseteq\Fm_B^l\cap\Fb$; then 
take $N=\Fm_B^n\Fb+\Fm_B^l$, which implies that 
$\Fb\cap N=\Fm_B^n\Fb$.
\end{Remark}

For a topological divided power extension of admissible 
topological rings $(B\to R,\delta)$ with kernel 
$\Fb\subseteq B$ we define
$$
\WW(B,\delta)=\varprojlim_N\WW(B/N,\delta/N)
$$
where $N$ runs through the open ideals of $B$ contained
in $\NNN_B$ such that $N\cap\Fb$ is stable under $\delta$.
Lemmas \ref{Le-WW-large-vv} and \ref{Le-WW-large-v} hold
in the topological case.

Assume that $p=2$. Then for an admissible topological ring 
we put
$$
\WW^+(R)=\varprojlim_N\WW^+(R/N),
$$
the limit taken over all open ideals $N$ of $R$ contained
in $\NNN_R$. If $(B\to R,\delta)$ is a topological 
divided power extension of admissible topological rings
which is topologically compatible with the canonical divided
powers of $2$, we can define
$$
\WW^+(B,\delta)=\varprojlim_N\WW^+(B/N,\delta/N)
$$
where $N$ runs through the open ideals of $B$ contained in $\NNN_B$
such that $\delta/N$ is defined and compatible with the canonical divided
powers of $2$.

The following example of admissible topological rings is used in section \ref{Se-bt2disp}.

\begin{Lemma}
\label{Le-pd-alg}
Let $R$ be a ring which is $I$-adically complete for an ideal $I\subseteq R$
such that $K=R/I$ is a perfect ring of characteristic $p$. 
Assume that $I=J+pR$ for an ideal $J\subseteq R$
such that $R/J^n$ has no $p$-torsion for each $n$.
For a projective $R$-module $t$ of finite type we consider 
the complete symmetric algebra 
\[
R[[t]]=\prod_{n\ge 0}\Sym_R^n(t).
\]
Let $(\Fa\subseteq S,\delta)$ be the divided power envelope of the ideal 
$tR[[t]]\subseteq R[[t]]$ and let $\hat S$ be the $I$-adic completion of $S$. 
Then 
\begin{enumerate}
\renewcommand{\theenumi}{\roman{enumi}}
\item
\label{It-pd-alg-pd}
$\hat S\to R$ is naturally a topological divided power extension
of admissible topological rings which is topologically compatible 
with the canonical divided powers of $p$, and
\item
\label{It-pd-alg-tors}
$\hat S$ has no $p$-torsion.
\end{enumerate}
\end{Lemma}

\begin{proof}
Let $\bar R_n=R/(p^nR+J^n)$ and $\bar S_n=S\otimes_R R_n$. 
We have $S=R\oplus\Fa$ and thus $\bar S_n=\bar R_n\oplus\bar\Fa_n$ 
with $\bar\Fa_n=\Fa\otimes_R\bar R_n$,
moreover the ideal $\bar\Fa_n$ carries divided powers $\delta_n$ induced by $\delta$; 
see \cite[Ch.~I, 1.7.1]{Berthelot:CohCristalline}.
In particular $\bar S_n$ is admissible. 
Since $\hat S\to R$ is the projective limit over $n$ of $\bar S_n\to\bar R_n$, to prove
\eqref{It-pd-alg-pd} it suffices to show that $\delta_n$ is compatible 
with the canonical divided powers of $p$.
Now $\Spec R\to\Spec R[[t]]$ is a regular immersion by Lemma \ref{Le:R[[T]]-regular},
thus $S$ is flat over $R$ by Proposition \ref{Pr:pd-flat}.
Since $R$ has no $p$-torsion the same holds for $S$, so the divided
powers on $\Fa$ extend canonically to the ideal $\Fb=\Fa+pS$.
We have $S/\Fb=R/pR$. The assumptions imply that 
$\Tor_1^R(R/J^n,R/pR)$ is zero. Hence there is an exact sequence
\[
0\to\Fb/J^n\Fb\to S/J^nS\to R/(pR+J^n)\to 0,
\]
which in turn gives an exact sequence
\[
0\to(\Fb/J^n\Fb)/p^n(S/J^nS)\to S/(J^nS+p^nS)\to R/(pR+J^n)\to 0.
\]
In both sequences the kernels carry divided powers which extend
the canonical divided powers of $p$ since the ideals $J^n\Fb$ of $S$
and $p^n(S/J^nS)$ of $S/J^nS$ are stable under the given divided powers.
Thus the divided powers $\delta_n$ on $\bar\Fa_n$ are compatible with the
canonical divided powers of $p$, which proves \eqref{It-pd-alg-pd}.

Let $S_n=S/J^nS$ and let $\hat S_n$ be its $p$-adic completion.
Since $S$ is flat over $R$ and since $R/J^n$ has no $p$-torsion, $S_n$ and $\hat S_n$
have no $p$-torsion. Using that $\hat S=\varprojlim_n\hat S_n$ it follows that $\hat S$ 
has no $p$-torsion, which proves \eqref{It-pd-alg-tors}.
\end{proof}

\subsection{Completeness}

For an admissible ring $R$, the Zink ring $\WW(R)$ is $p$-adically complete.
Indeed, $W(R_{\red})$ is $p$-adically complete, and
$\hat W(\NNN_R)$ is annihilated by a power of $p$
because this holds for $W(\NNN_R)$ by Lemma \ref{Le-W(N)}. 
The following topological variant of this fact seems to be less obvious.

\begin{Prop}
\label{Pr-WW-p-adic}
Let $R$ be an $I$-adically complete ring such that the ideal $I$ is
finitely generated and $K=R/I$ is a perfect ring of 
characteristic $p$. Then the ring\/ $\WW(R)$ is $p$-adically complete.
If $p=2$, the ring $\WW^+(R)$ is $p$-adically complete too.
\end{Prop}

This is similar to \cite[Proposition 3]{Zink-Disp}, which says that
$W(R)$ is $p$-adically complete if this holds for $R$.

\begin{proof}
The ring $W(R)$ is $p$-adically separated because this
holds for each $W(R/I^n)$. Thus $\WW(R)$ is
$p$-adically separated too.
Let $S=W(K)[[t_1,\ldots, t_r]]$ and let $S\to R$ 
be a homomorphism which maps $t_1,\ldots, t_r$
to a set of generators of $I/I^2$. Then $S\to R$ is
surjective, and so is $\WW(S)\to\WW(R)$. Since
$\WW(R)$ is $p$-adically separated, in order to
show that $\WW(R)$ is $p$-adically complete 
we may assume that $R=S$. Consider the ideals
$J_n=p^nW(R)+W(I^n)$
of $W(R)$ and $\JJ_n=\WW(R)\cap J_n$ of $\WW(R)$.
Then 
$$
W(R)/J_n=W_n(K)\oplus W(I/I^n),
$$
$$
\WW(R)/\JJ_n=W_n(K)\oplus\hat W(I/I^n).
$$
It follows that $W(R)$ and $\WW(R)$ are complete and separated 
for the linear topologies generated by the ideals $J_n$ and
$\JJ_n$, respectively; moreover $\WW(R)$ 
is closed in $W(R)$.
The ring $W(R)$ is also complete and separated
for the linear topology generated by the ideals
$J_{n,m}'=\Ker(W(R)\to W_m(R/I^n))$. The $J$-topology is
finer than the $J'$-topology because $J_{2n}\subseteq J'_{n,n}$.

We claim that for each $r\ge 1$ the ideal $p^rW(R)$ of $W(R)$
is closed in the $J'$-topology. 
This is a variant of \cite[Lemma 6]{Zink-Disp} with essentially the same proof.
First, for $s\ge 1$ an element $x=(x_0,\ldots,x_m)$ of $W_{m+1}(R)$ 
satisfies $x_i\in I^s$ for all $i$ if and only if 
$w_i(x)\in I^{i+s}$ for all $i$; see the proof of \cite[Lemma 4]{Zink-Disp}.
Then the proof of \cite[Lemma 5]{Zink-Disp} shows that an element
$x\in W_m(R)$ is divisible by $p^r$ if and only
if for each $s$ the image $\bar x\in W_m(R/I^s)$ 
is divisible by $p^r$. Using this, the claim follows
from the proof of \cite[Lemma 6]{Zink-Disp}.

Thus $p^rW(R)$ is closed in the finer $J$-topology
as well. Assume that we have $p\WW(R)=pW(R)\cap\WW(R)$.
Then $p^r\WW(R)$ is closed in the $\JJ$-topology, which implies
that $\WW(R)$ is $p$-adically complete; see \cite[Lemma 7]{Zink-Disp}. 
Thus for $p\ge 3$ the proof is completed by 
Lemma \ref{Le-WW-p-adic} below. 
For $p=2$ the same reasoning shows that $\WW^+(R)$ is
$p$-adically complete. Now $\WW^+(R)/\WW(R)$ is isomorphic to $K$ as
abelian groups by the proof of Lemma \ref{Le-WW+}. 
We get exact sequences 
$$
0\to K\to \WW(R)/2^n\WW(R)\to\WW^+(R)/2^n\WW^+(R)\to K\to 0,
$$
where the transition maps from $n+1$ to $n$ are zero on
the left hand $K$ and the identity on the right
hand $K$. It follows that $\WW(R)$ is $p$-adically complete as well.
\end{proof}

\begin{Lemma}
\label{Le-WW-p-adic}
For a perfect ring $K$ of characteristic $p$
let $R=W(K)[[t_1,\ldots,t_r]]$ with the 
$(p,t_1,\ldots,t_r)$-adic topology. If $p\ge3$ then
$$
pW(R)\cap\WW(R)=p\WW(R).
$$
If $p=2$ then 
$$
2W(R)\cap\WW^+(R)=2\WW^+(R).
$$
\end{Lemma}

\begin{proof}
Assume $p=2$. Let $I$ be the kernel of $R\to K$
and let $\bar I=I/pR$.
The filtration $0\subseteq W(pR)\subseteq W(I)\subseteq W(R)$ induces
a filtration of $\WW(R)$ with successive quotients
$\tilde W(pR):=\varprojlim_n\tilde W(pR/I^npR)$ and
$\hat W(\bar I):=\varprojlim_n\hat W(\bar I/\bar I^n)$ 
and $W(K)$. To prove the lemma it suffices to show that
$$
pW(\bar I)\cap\hat W(\bar I)=p\hat W(\bar I)
$$
and
$$
pW(pR)\cap\tilde W(pR)=p\tilde W(pR).
$$
The first equality holds because multiplication by $p$ 
on $W(\bar I)$ is given by 
$(a_0,a_1,\ldots)\mapsto(0,a_0^p,a_1^p,\ldots)$,
and for $a\in\bar I$ with $a^p\in\bar I^{pn}$
we have $a\in\bar I^n$. The second equality holds
because the isomorphism $\Log:W(pR)\cong (pR)^\NN$
induces an isomorphism between $\tilde W(pR)$ and the 
group of all sequences in $(pR)^\NN$
that converge to zero $I$-adically.
The proof for $p\ge 3$ is similar.
\end{proof}

\subsection{Divided powers}
\label{Subse-pd}

In section \ref{Se-bt2disp} we will use that
the augmentation ideals of the Zink ring and its variants 
carry natural divided powers, with some exception when $p=2$;
see also \ref{Subse-strange-pd}.

Let us first recall the canonical divided powers on the
Witt ring. If $R$ is a $\ZZ_{(p)}$-algebra, then
$W(R)$ is a $\ZZ_{(p)}$-algebra as well, and the
ideal $I_R$ carries divided powers $\gamma$ which
are determined by $(p-1)!\gamma_p(v(x))=p^{p-2}v(x^p)$.
Assume that $(B\to R,\delta)$ is a divided power
extension of $\ZZ_{(p)}$-algebras with kernel $\Fb\subseteq B$.
Let $I_{B/R}$ be the kernel of $W(B)\to R$.
If $i:\Fb\to W(\Fb)$ is defined by $\Log(i(b))=[b,0,0,\ldots]$,
we have $I_{B/R}=I_B\oplus i(\Fb)$, and
the divided powers $\gamma$ on $I_B$ extend to divided
powers $\gamma'=\gamma\oplus\delta$ on $I_{B/R}$ such
that $\gamma'_n(i(b))=i(\delta_n(b))$ for $b\in\Fb$.
If $p\in\Fb$ and if $\delta$ extends the
canonical divided powers of $p$, then $\gamma\oplus\delta$ 
extends the canonical divided powers
of $p$, and $f$ preserves $\gamma\oplus\delta$.
This is clear when $B$ has no $p$-torsion;
the general case follows because $(B\to R)$ can be written 
as the quotient of a divided power extension $(B'\to R')$,
where $B'$ is the divided power algebra of a free module over a polynomial ring $R''$ over $\ZZ_{(p)}$, and $R'=R''/pR''$.

These facts extend to the Zink ring as follows.

\begin{Lemma}
\label{Le-pd-II}
Let\/ $\II\subseteq\WW$ be one of the following.
\begin{enumerate}
\renewcommand{\theenumi}{\roman{enumi}}
\item
$\II=\II_R$ and\/ $\WW=\WW(R)$ for an admissible ring $R$
with $p\ge 3$,
\item
$\II=\II^+_R$ and\/ $\WW=\WW^+(R)$ for an admissible ring
$R$ with $p=2$.
\end{enumerate}
Then the divided powers $\gamma$ on $I_R$ induce
divided powers on\/ $\II$.
\end{Lemma}

\begin{proof}
Since $\WW$ is a $\ZZ_{(p)}$-algebra it suffices to show
that $\II$ is stable under the map $\gamma_p:I_R\to I_R$, which
is true because $\II=v(\WW)$ by Lemmas \ref{Le-WW-v} and
\ref{Le-WW-large-v}.
\end{proof}

\begin{Lemma}
\label{Le-pd-II**}
Let $(B\to R,\delta)$ be a divided power extension
of admissible rings with kernel $\Fb\subseteq B$,
and let\/ $\II_{B/R}$ be the kernel of\/ $\WW(B,\delta)\to R$.
Assume that $p\ge 3$; or that $p=2$ and\/ $p\in\Fb$ and $\delta$ 
extends the canonical divided powers of\/ $p$. 
Then the divided powers $\gamma\oplus\delta$ on $I_{B/R}$ 
induce divided powers on\/ $\II_{B/R}$.
If $p\in\Fb$ and if\/ $\delta$ extends the canonical
divided powers of $p$, then the divided powers on\/ 
$\II_{B/R}$ induced by $\gamma\oplus\delta$ extend the canonical 
divided powers of $p$ and are preserved by $f$.
\end{Lemma}

\begin{proof}
Let $\II'_B$ be the kernel of\/ $\WW(B/R)\to B$. 
Then $\II_{B/R}=\II'_B\oplus i(\Fb)$, and we have
$\II'_B=v(\WW(B/R))$ by Lemma \ref{Le-WW-large-v}.
Thus $\II'_{B}$ is stable under $\gamma$, and 
$\II_{B/R}$ is stable under $\gamma\oplus\delta$.
The second assertion follows from the corresponding
fact for the Witt ring.
\end{proof}

%---------------------------------------------------------------

\section{Dieudonn\'e displays}

In this section, Dieudonn\'e displays and a number of 
variants related to divided power extensions are 
defined. We use the formalism of frames and windows 
introduced in \cite{Lau-Frames}.
First of all, let us recall a well-known fact.

\begin{Lemma}
\label{Le-W-f}
Let $A$ be a commutative, not necessarily unitary ring.
For $x\in W(A)$ we have $f(x)\equiv x^p$ modulo $pW(A)$.
Similarly, for $x\in\hat W(A)$ we have $f(x)\equiv x^p$ modulo 
$p\hat W(A)$.
\end{Lemma}

\begin{proof}
For $x\in W(R)$ write $x=[x_0]+v(y)$ with $x_0\in R$ 
and $y\in W(R)$. Then $f(x)\equiv[x_0^p]\equiv x^p$ modulo 
$pW(R)$ because $fv=p$ and $v(y)^p=p^{p-1}v(y^p)$. 
The same calculation applies with $\hat W$ in place of $W$.
\end{proof}

\subsection{Frames and windows}
\label{Subse-frames}

We recall the notion of frames and windows 
from \cite{Lau-Frames} with some additions.
A \emph{pre-frame} is a quintuple
$$
\FFF=(S,I,R,\sigma,\sigma_1)
$$
where $S$ and $R=S/I$ are rings,
where $\sigma:S\to S$ is a ring endomorphism with 
$\sigma(a)\equiv a^p$ modulo $pS$,
and where $\sigma_1:I\to S$ 
is a $\sigma$-linear map of $S$-modules
whose image generates $S$ as an $S$-module.
Then there is a unique element $\theta\in S$ with 
$\sigma(a)=\theta\sigma_1(a)$ for $a\in I$. 
The pre-frame $\FFF$ is called a \emph{frame} if 
$$
I+pS\subseteq\Rad(S).
$$
If in addition all projective $R$-modules of finite type
can be lifted to projective $S$ modules then $\FFF$ is
called a \emph{lifting frame}.

A homomorphism of pre-frames or frames $\alpha:\FFF\to\FFF'$
is a ring homomorphism $\alpha:S\to S'$ with 
$\alpha(I)\subseteq I'$ such that $\sigma'\alpha=\alpha\sigma$
and $\sigma_1'\alpha=u\cdot\alpha\sigma_1$
for a unit $u\in S'$, which is then determined by $\alpha$. 
It also follows that $\alpha(\theta)=u\theta'$. We say that $\alpha$ is a 
$u$-homomorphism of pre-frames or frames.
If $u=1$ then $\alpha$ is called \emph{strict}.

Let now $\FFF$ be a frame. An $\FFF$-\emph{window} is a quadruple
$$
\PPP=(P,Q,F,F_1)
$$
where $P$ is a projective $S$-module of finite type
with a submodule $Q$ such that there exists
a decomposition of $S$-modules
$P=L\oplus T$ with $Q=L\oplus IT$, called a \emph{normal
decomposition}, and where $F:P\to P$
and $F_1:Q\to P$ are $\sigma$-linear maps
of $S$-modules with 
$$
F_1(ax)=\sigma_1(a)F(x)
$$ 
for $a\in I$ and $x\in P$;
we also assume that $F_1(Q)$ generates $P$ as an $S$-module.
Then $F(x)=\theta F_1(x)$ for $x\in Q$. If $F$ is a lifting
frame, every pair $(P,Q)$ such that $P$ is a
projective $S$-module of finite type and $P/Q$ is a projective 
$R$-module admits a normal decomposition. In general,
for given $(P,Q)$ together with a normal 
decomposition $P=L\oplus T$, giving $\sigma$-linear
maps $(F,F_1)$ which make an $\FFF$-window $\PPP$
is equivalent to giving a $\sigma$-linear isomorphism
$$
\Psi:L\oplus T\to P
$$ 
defined by $F_1$ on $L$ and by $F$ on $T$.
The triple $(L,T,\Psi)$ is called a 
\emph{normal representation} of $\PPP$.

A frame homomorphism $\alpha:\FFF\to\FFF'$ 
induces a base change functor $\alpha_*$
from $\FFF$-windows to $\FFF'$-windows.
In terms of normal representations it is given by
$$
(L,T,\Psi)\mapsto(S'\otimes_SL,S'\otimes_ST,\Psi')
$$
with $\Psi'(s'\otimes l)=u\sigma'(s')\otimes\Psi(l)$
and $\Psi'(s'\otimes t)=\sigma'(s')\otimes\Psi(t)$.

A frame homomorphism $\alpha:\FFF\to\FFF'$ is called 
\emph{crystalline} if the functor
$\alpha_*$ is an equivalence of categories.
For reference we recall \cite[Theorem 3.2]{Lau-Frames}:

\begin{Thm}
\label{Th-frame-crys}
Let $\alpha:\FFF\to\FFF'$ be a homomorphism of frames
which induces an isomorphism $R\cong R'$ and a
surjection $S\to S'$ with kernel $\Fa$. 
We assume that there is a finite filtration of ideals
$\Fa=\Fa_0\supseteq\ldots\supseteq\Fa_n=0$ 
with $\sigma_1(\Fa_i)\subseteq\Fa_i$ and 
$\sigma(\Fa_i)\subseteq\Fa_{i+1}$,
that $\sigma_1$ is elementwise nilpotent on each
$\Fa_i/\Fa_{i+1}$,  
and that all projective $S'$-modules of finite type 
lift to projective $S$-modules of finite type.
Then $\alpha$ is crystalline.
\qed
\end{Thm}

Let us recall the operator $V^\sharp$ of a window.
For an $S$-module $M$ we write $M^{(\sigma)}=S\otimes_{\sigma,S}M$. 
A filtered $F$-$V$-module over $\FFF$ is a quadruple
$$
(P,Q,F^\sharp,V^\sharp)
$$ 
where $P$ is a projective $S$-module of finite type, 
$Q$ is a submodule of $P$ such that $P/Q$ is projective over 
$R$, and $F^\sharp:P^{(\sigma)}\to P$ and $V^\sharp:P\to P^{(\sigma)}$ 
are $S$-linear maps with $F^\sharp V^\sharp=\theta$ 
and $V^\sharp F^\sharp=\theta$. 

\begin{Lemma}
\label{Le-F-V-mod}
There is a natural functor from $\FFF$-windows 
to filtered $F$-$V$-modules over $\FFF$,
which is fully faithful if $\theta$ is not
a zero divisor in $S$.
\end{Lemma}

\begin{proof}
The functor is $(P,Q,F,F_1)\mapsto(P,Q,F^\sharp,V^\sharp)$
where $F^\sharp$ is the linearisation of $F$, and $V^\sharp$
is the unique $S$-linear map such that $V^\sharp(F_1(x))=1\otimes x$
for $x\in Q$. Clearly this determines $V^\sharp$ if it exists.
In terms of a normal representation $(L,T,\Psi)$ of $\PPP$,
thus $P=L\oplus T$, one can define 
$V^\sharp=(1\oplus\theta)(\Psi^\sharp)^{(-1)}$.
The required relation $F^\sharp V^\sharp=\theta$ on $P$
is equivalent to $F^\sharp V^\sharp F_1=\theta F_1$
on $Q$, which is clear since $\theta F_1=F$. 
The required relation $V^\sharp F^\sharp=\theta$ on $P^{(\sigma)}$ 
holds if and only if it holds after multiplication with
$\sigma_1(a)$ for all $a\in I$. For $x\in P$ we calculate
$\sigma_1(a)V^\sharp F^\sharp(1\otimes x)=V^\sharp F_1(ax)
=\sigma(a)\otimes x=\theta\sigma_1(a)(1\otimes x)$.

Assume that $\theta$ is not a zero divisor in $S$.
It suffices to show that the forgetful functors 
from windows to triples $(P,Q,F)$ and from
filtered $F$-$V$-modules to triples $(P,Q,F^\sharp)$ 
are fully faithful. In the first case this holds
because $\theta F_1=F$. In the 
second case, for an endomorphism $\alpha$ of $P$ with 
$\alpha F^\sharp=F^\sharp\alpha^{(\sigma)}$ we calculate
$V^\sharp\alpha\theta=V^\sharp\alpha F^\sharp V^\sharp=
V^\sharp F^\sharp\alpha^{(\sigma)}V^\sharp=\theta\alpha^{(\sigma)}V^\sharp$,
which implies that $V^\sharp\alpha=\alpha^{(\sigma)} V^\sharp$.
\end{proof}

Finally, let us recall the duality formalism. 
Let $\FFF$ denote the $\FFF$-window $(S,I,\sigma,\sigma_1)$. 
A bilinear form between $\FFF$-windows 
$$
\beta:\PPP\times\PPP'\to\FFF
$$ 
is an $S$-bilinear map $\beta:P\times P'\to S$ 
such that $\beta(Q\times Q')\subseteq I$
and $\beta(F_1x,F_1'x')=\sigma_1(\beta(x,x'))$ for $x\in Q$ 
and $x'\in Q'$. For each $\PPP$, the functor 
$\PPP'\mapsto\Bil(\PPP\times\PPP',\FFF)$ is represented
by an $\FFF$-window $\PPP^t$, called the dual of $\PPP$.
The tautological bilinear form $\PPP\times\PPP^t\to S$ is 
a perfect bilinear map $P\times P^t\to S$. 
There is a bijection between normal representations 
$P=L\oplus T$ and normal representation $P^t=L^t\oplus T^t$ 
determined by
$\left<L,L^t\right>=0=\left<T,T^t\right>$. The associated
operators $\Psi:P\to P$ and $\Psi^t:P^t\to P^t$ are
related by $\left<\Psi x,\Psi^tx'\right>=\sigma\left<x,x'\right>$.

There is also an obvious duality of filtered $F$-$V$-modules
over $\FFF$: The dual of $\MMM=(P,Q,F^\sharp,V^\sharp)$ 
is $\MMM^t=(P^*,Q',V^{\sharp*},F^{\sharp*})$
where $P^*=\Hom_S(P,S)$, and $Q'$ is the submodule 
of all $y$ in $P^*$ with $y(Q)\subseteq I$.
It is easy to see that the functor in Lemma \ref{Le-F-V-mod}
preserves duality.

\subsection{Frames associated to the Witt ring}

For an arbitrary ring $R$ let $f_1:I_R\to W(R)$
be the inverse of the Verschiebung $v$. Then
$$
\WWW_R=(W(R),I_R,R,f,f_1)
$$
is a pre-frame with $\theta=p$. If $R$ is $p$-adically complete,
$\WWW_R$ is a lifting frame because $W(R)$ is
$I_R$-adically complete by \cite[Proposition 3]{Zink-Disp}, and windows 
over $\WWW_R$ are displays over $R$. 

For a divided power 
extension of rings $(B\to R,\delta)$ with kernel $\Fb\in B$
one can define a pre-frame
$$ 
\WWW_{B/R}=(W(B),I_{B/R},R,f,\tilde f_1)
$$
with $I_{B/R}=I_B+W(\Fb)$ such that $\tilde f_1:I_{B/R}\to W(B)$ 
is the unique extension of $f_1$ whose restriction
to $W(\Fb)$ is given by
$[a_0,a_1,a_2,\ldots]\mapsto[a_1,a_2,\ldots]$
in logarithmic coordinates; see \ref{Subse-enlarged-Zink}.
The projection $\WWW_B\to\WWW_R$ factors into strict 
pre-frame homomorphisms $\WWW_B\to\WWW_{B/R}\to\WWW_R$. 

As a special case, assume that $R$ is a perfect ring 
of characteristic $p$.
Then $f$ is an automorphism of $W(R)$, and $I_R=pW(R)$.
Let us call a Dieudonn\'e module over $R$ a triple $(P,F,V)$ 
where $P$ is a projective $W(R)$-module of finite type 
equipped with an $f$-linear endomorphism $F$ and an 
$f^{-1}$-linear endomorphism $V$ such that $FV=p$,
or equivalently $VF=p$.

\begin{Lemma}
\label{Le-disp-perf}
Displays over a perfect ring $R$ are equivalent to
Dieudonn\'e modules over $R$.
\end{Lemma}

\begin{proof}
To a display $(P,Q,F,F_1)$ we associate the Dieudonn\'e
module $(P,F,V)$ where the linearisation of $V:P\to P$ is 
the operator $V^\sharp$ defined in Lemma \ref{Le-F-V-mod}.
Then $VF_1:Q\to P$ is the inclusion. 
Here $F_1$ is surjective since $f$ is bijective.
Thus $Q=V(P)$, and the functor is fully faithful;
see Lemma \ref{Le-F-V-mod}.
 It remains to show that for every Dieudonn\'e
module $(P,F,V)$ the $R$-module $M=P/V(P)$ is projective.
For $\Fp\in\Spec R$ let $\ell_M(\Fp)$ be the dimension
of the fibre of $M$ at $\Fp$. Let $N=P/F(P)$. Then
$\ell_M+\ell_N=\ell_{P/pP}$ as functions on $\Spec R$.
Since $M$ and $N$ are of finite type and since $P/pP$ is
projective, the functions $\ell_M$ and $\ell_N$ are 
upper semicontinuous, and $\ell_{P/pP}$ is locally constant. 
It follows that $\ell_M$ is locally constant, which implies 
that $M$ is projective because $R$ is reduced.
\end{proof}

\subsection{Dieudonn\'e frames}
\label{Subse-Dieud-frames}

For an admissible ring $R$ in the sense of Definition \ref{Def-admissible}
let $\II_R$ be the kernel of 
$w_0:\WW(R)\to R$, and let $\ff_1:\II_R\to\WW(R)$ be the inverse
of $\vv$, which is well-defined by Lemma \ref{Le-WW-vv}.
If $p$ is odd, then $\vv=v$ and $\ff_1=f_1$.

\begin{Lemma}
\label{Le-DDD-frame}
The quintuple 
$$
\DDD_R=(\WW_R,\II_R,R,f,\ff_1)
$$ 
is a lifting frame.
\end{Lemma}

We call $\DDD_R$ the Dieudonn\'e frame associated to $R$.

\begin{proof}
In order that $\DDD_R$ is a pre-frame we need that 
$f(a)\equiv a^p$ modulo $p\WW(R)$ for $a\in\WW(R)$, 
which follows from Lemma \ref{Le-W-f}
applied to $W(R_{\red})$ and to $\hat W(\NNN_R)$.
Since $\hat W(\NNN_R)$ is a nilideal by Lemma \ref{Le-W(N)} 
and since the quotient $\WW(R)/\hat W(\NNN_R)=W(R_{\red})$ 
is $p$-adically complete with $pW(R_{\red})=I_{R_{\red}}$,
the kernel of $\WW(R)\to R_{\red}$ lies in the radical
of $\WW(R)$,  and projective $R_{\red}$-modules of finite
type lift to projective $\WW(R)$-modules of finite type. 
It follows that $\DDD_R$ is a lifting frame.
\end{proof}

The inclusion $\WW(R)\to W(R)$ is a $u_0$-homomorphism
of frames $\DDD_R\to\WWW_R$.
Thus for $\DDD_R$ we have $\theta=p$ if $p$ is odd and
$\theta=2u_0=2-[4]$ if $p=2$.

\begin{Defn}
A Dieudonn\'e display over $R$ is a window over $\DDD_R$.
\end{Defn}

Thus a Dieudonn\'e display is a quadruple $\PPP=(P,Q,F,F_1)$ 
where $P$ is a projective $\WW(R)$-module of finite type with
a filtration $\II_RP\subseteq Q\subseteq P$ such that 
$P/Q$ is a projective $R$-module, $F:P\to P$ and 
$F_1:Q\to P$ are $f$-linear maps with 
$F_1(ax)=\ff_1(a)F(x)$ for $a\in\II_R$ and $x\in P$, 
and $F_1(Q)$ generates $P$. We write
$$
\Lie(\PPP)=P/Q.
$$
The \emph{height} of $\PPP$ is the rank of the 
$\WW(R)$-module $P$, and the \emph{dimension} of $\PPP$ 
is the rank of the $R$-module $\Lie(\PPP)$, both viewed 
as locally constant functions on $\Spec R$.
As in the case of general frames, 
we also denote by $\DDD_R$ the Dieudonn\'e display 
$(\WW(R),\II_R,f,\ff_1)$ over $R$.

\subsection{Relative Dieudonn\'e frames}

Let $(B\to R,\delta)$ be a divided power extension of
admissible rings with kernel $\Fb\subseteq B$.
Let $\WW(B/R)=\WW(B,\delta)$ as in \ref{Subse-enlarged-Zink} 
and let $\II_{B/R}$ be the kernel of the projection 
$\WW({B/R})\to R$, thus 
$$
\II_{B/R}=\II_B+\tilde W(\Fb).
$$

\begin{Lemma}
\label{Le-tilde-ff1}
There is a unique extension of\/ $\ff_1:\II_B\to\WW(B)$ to an 
$f$-linear map\/ $\tilde\ff_1:\II_{B/R}\to\WW(B/R)$
of\/ $\WW(B/R)$-modules such that the restriction of\/ $\tilde\ff_1$ 
to $\tilde W(\Fb)$ is given by
\begin{equation}
\label{Eq-tilde-ff1}
\tilde\ff_1([a_0,a_1,a_2,\ldots])=
[w_0(u_0^{-1})a_1,w_1(u_0^{-1})a_2,\ldots]
\end{equation}
in logarithmic coordinates. The quintuple
$$
\DDD_{B/R}=\DDD_{B/R,\delta}=(\WW(B/R),\II_{B/R},R,f,\tilde\ff_1)
$$
is a lifting frame.
\end{Lemma}

\begin{proof}
Clearly $\tilde\ff_1$ is determined by \eqref{Eq-tilde-ff1}.
Let $\II_B'$ be the kernel of\/ $\WW(B/R)\to B$. 
By Lemma \ref{Le-WW-large-vv}, the inverse of $\vv$ is
an $f$-linear map $\ff_1':\II_B'\to\WW(B/R)$
which extends $\ff_1$. In logarithmic coordinates,
the restriction of $\vv$ to $W(\Fb)$ is given by
$[a_0,a_1,\ldots]\mapsto[0,w_0(u_0)a_0,w_1(u_0)a_1,\ldots]$.
Thus $\ff_1'$ extends to the desired $\tilde\ff_1$.
As in the proof of Lemma \ref{Le-DDD-frame}, the
kernel of $\WW(B/R)\to R_{\red}$ lies in the radical
of $\WW(B/R)$, and projective $R_{\red}$-modules of finite type
lift to $\WW(B/R)$.
\end{proof}

We call $\DDD_{B/R}$ the relative Dieudonn\'e frame 
associated to the divided power extension $(B/R,\delta)$, 
and $\DDD_{B/R}$-windows
are called Dieudonn\'e displays for $B/R$.
There are natural strict frame homomorphisms
$$
\DDD_B\to\DDD_{B/R}\to\DDD_R.
$$
If the divided powers $\delta$ are nilpotent, then
$\WW(B)=\WW(B/R)$. 

\begin{Prop}
\label{Pr-DDD-crys}
The frame homomorphism $\DDD_{B/R}\to\DDD_R$ is crystalline.
\end{Prop}

\begin{proof}
This follows from Theorem \ref{Th-frame-crys}. 
Indeed, let $\Fa$ denote the kernel of the surjective homomorphism
$\WW(B/R)\to\WW(R)$, thus 
$\Fa=\tilde W(\Fb)\cong\Fb^{(\NN)}$. 
The endomorphism $\tilde\ff_1$ of $\Fa$ is elementwise 
nilpotent by \eqref{Eq-tilde-ff1}. The required filtration 
of $\Fa$ can be taken to be $\Fa_i=p^i\Fa$; this is a
finite filtration by Lemma \ref{Le-W(N)}.
We have $\tilde\ff_1(\Fa_i)=\Fa_i$ by \eqref{Eq-tilde-ff1}, 
and $f(\Fa_i)=\Fa_{i+1}$ because the endomorphism $f$ of $\Fa$
is given by $[a_0,a_1,\ldots]\mapsto [pa_1,pa_2,\ldots]$
in logarithmic coordinates.
\end{proof}

\subsection{$v$-stabilised Dieudonn\'e frames}
\label{Subse-v-stab-D}

Assume that $p=2$.
The preceding constructions can be repeated with
$\WW^+$ and $v$ in place of\/ $\WW$ and $\vv$.
More precisely, for an admissible ring $R$ let
$\II_R^+$ be the kernel of $\WW^+(R)\to R$ and
let $f_1:\II_R^+\to\WW^+(R)$ be the inverse of $v$,
which is well-defined by Lemma \ref{Le-WW-large-v}.
The $v$-stabilised Dieudonn\'e frame associated to
$R$ is defined as
$$
\DDD^+_R=(\WW^+(R),\II^+_R,R,f,f_1).
$$
This is a lifting frame by the proof of Lemma \ref{Le-DDD-frame}. 
The inclusion $\WW(R)\to\WW^+(R)$ is a $u_0$-homomorphism of
frames $\DDD_R\to\DDD^+_R$, which is invertible if and only
if $2R=0$. Windows over $\DDD^+_R$ are called $v$-stabilised
Dieudonn\'e displays over $R$.

Assume again that $p=2$, and let $(B\to R,\delta)$ be 
a divided power extension
of admissible rings with kernel $\Fb\subseteq B$ which is
compatible with the canonical divided powers of $2$. 
Let $\II^+_{B/R}$ be the kernel of the natural map
$\WW^+_{B/R}\to R$, thus
$$
\II^+_{B/R}=\II^+_B+\tilde W(\Fb).
$$
There is a unique extension of $f_1:\II^+_B\to\WW^+(B)$ to an 
$f$-linear map of $\WW^+(B/R)$-modules 
$\tilde f_1:\II^+_{B/R}\to\WW^+(B/R)$ such that its restriction 
to $\tilde W(\Fb)$ is given by $[a_0,a_1,a_2,\ldots]\mapsto[a_1,a_2,\ldots]$
in logarithmic coordinates, and the quintuple 
$$
\DDD^+_{B/R}=(\WW^+(B/R),\II^+_{B/R},R,f,\tilde f_1)
$$
is a lifting frame.
This follows from the proof of Lemma \ref{Le-tilde-ff1}.
We have a $u_0$-homomorphism of frames $\DDD_{B/R}\to\DDD^+_{B/R}$,
which is invertible if and only if $2R=0$, and strict frame
homomorphisms 
$$
\DDD^+_B\to\DDD^+_{B/R}\to\DDD^+_R.
$$
If the divided powers induced by $\delta$ on $(\Fb+2B)/2B$
are nilpotent, then $\WW^+(B)$ is equal to $\WW^+(B/R)$. 

\begin{Cor} 
\label{Co-DDD+-crys}
The frame homomorphism $\DDD^+_{B/R}\to\DDD^+_R$ is crystalline.
\end{Cor}

\begin{proof}
This follows from the proof of Proposition \ref{Pr-DDD-crys}.
\end{proof}

\subsection{The crystals associated to Dieudonn\'e displays}
\label{Subse-crystals}

Let $R$ be an admissible ring. 
We denote by $\Cris_{\adm}(R)$ the category of divided 
power extensions $(\Spec A\to\Spec B,\delta)$ where
$A$ is an $R$-algebra which is an admissible ring,
and where $p$ is nilpotent in $B$. Then the kernel of 
$B\to A$ is bounded nilpotent, so $B$ is an admissible
ring as well. 

Let $\PPP$ be a Dieudonn\'e display over $R$. 
For $(\Spec A\to\Spec B,\delta)$ in $\Cris_{\adm}(R)$ we 
denote the base change of $\PPP$ to $A$ by $\PPP\!_A$ and 
the unique Dieudonn\'e display for $B/A$ which 
lifts $\PPP\!_A$ by 
$$
\PPP_{B/A}=(P_{B/A},Q_{B/A},F,F_1);
$$ 
see Proposition \ref{Pr-DDD-crys}. A homomorphism 
of divided power extensions of admissible rings 
$\alpha:(B\to A,\delta)\to(B'\to A',\delta')$ induces a 
frame homomorphism $\DDD_\alpha:\DDD_{B/A}\to\DDD_{B'/A'}$,
and we have a natural isomorphism 
$$
(\DDD_{\alpha})_*(\PPP_{B/A})\cong\PPP_{B'/A'}.
$$
In more sophisticated terms this can be expressed as follows:
The frames $\DDD_{B/A}$ form a presheaf of frames
$\DDD_{**}$ on $\Cris_{\adm}(R)$, and 
 Proposition \ref{Pr-DDD-crys} implies that
 the category of
Dieudonn\'e displays over $R$ is equivalent to
the category of crystals in $\DDD_{**}$-windows
on $\Cris_{\adm}(R)$. Then $\PPP_{B/A}$ is the value in 
$(\Spec A\to\Spec B,\delta)$ of the crystal associated to $\PPP$.

For a Dieudonn\'e display $\PPP=(P,Q,F,F_1)$ over $R$,  
we define the Witt crystal $\KK(\PPP)$ on $\Cris_{\adm}(R)$ by
$$
\KK(\PPP)_{B/A}=P_{B/A}.
$$
This is a projective $\WW({B/A})$-module of finite type.
The Dieudonn\'e crystal $\DD(\PPP)$ on $\Cris_{\adm}(R)$ 
is defined by
$$
\DD(\PPP)_{B/A}=P_{B/A}\otimes_{\WW(B/A)}B.
$$
This is a projective $B$-module of finite type.
The Hodge filtration of $\PPP$ is the submodule
$$
Q/\II_RP\subseteq P/\II_RP=\DD(\PPP)_{R/R}.
$$

\begin{Cor}
\label{Co-deform-DDD}
Let $(B\to R,\delta)$ be a nilpotent divided power 
extension of admissible rings. The category of Dieudonn\'e
displays over $B$ is equivalent to the category of
Dieudonn\'e displays $\PPP$ over $R$ together with a lift
of the Hodge filtration of $\PPP$ to a direct summand of\/
$\DD(\PPP)_{B/R}$.
\end{Cor}

\begin{proof}
If the divided powers are nilpotent, then $\WW(B/R)=\WW(B)$,
and lifts of windows under the frame homomorphism 
$\DDD_{B}\to\DDD_{B/R}$ are in bijection with lifts of the 
Hodge filtration.
\end{proof}

The preceding definitions have a $v$-stabilised variant. 
Let $\Cris_{\adm}(R/\ZZ_p)$
be the full subcategory of $\Cris_{\adm}(R)$ where the divided
powers are compatible with the canonical divided powers
of $p$. 
Assume now that $p=2$ and 
let $\PPP^+$ be a $v$-stabilised Dieudonn\'e display
over $R$, i.e.\ a window over $\DDD^+_R$. For
$(\Spec A\to\Spec B,\delta)$ in $\Cris_{\adm}(R/\ZZ_2)$
we denote by $\PPP^+_A$ the base change of
$\PPP^+$ to $\DDD^+_A$ and by 
$$
\PPP^+_{B/A}=(P^+_{B/A},Q^+_{B/A},F,F_1)
$$ 
the unique lift of $\PPP^+_A$ to a $\DDD^+_{B/A}$-window, 
which exists by Corollary \ref{Co-DDD+-crys}.
The $v$-stabilised Witt crystal $\KK^+(\PPP^+)$ 
and the $v$-stabilised Dieudonn\'e crystal $\DD^+(\PPP^+)$
on $\Cris_{\adm}(R/\ZZ_2)$ are defined by 
$\KK^+(\PPP^+)_{B/A}=P^+_{B/A}$ and
$$
\DD^+(\PPP^+)_{B/A}=P^+_{B/A}\otimes_{\WW^+(B/A)}B.
$$

\begin{Cor}
\label{Co-deform-DDD+}
Assume that $p=2$. Let $(B\to R,\delta)$ be a divided power extension
of admissible rings which is compatible with the
canonical divided powers of $2$ such that the divided powers
induced by $\delta$ on the kernel of $B/2B\to R/2R$ are
nilpotent. Then the category of $v$-stabilised Dieudonn\'e
displays over $B$ is equivalent to the category of 
$v$-stabilised Dieudonn\'e displays $\PPP^+\!$ over $R$
together with a lift of the Hodge filtration of $\PPP^+\!$
to a direct summand of\/ $\DD^+(\PPP^+)_{B/R}$.
\end{Cor}

\begin{proof}
This is analogous to Corollary \ref{Co-deform-DDD},
using that $\WW^+(B/R)=\WW^+(B)$ under the given 
assumptions on $\delta$; see the end of \ref{Subse-v-stab-Zink}.
\end{proof}

\begin{Lemma}
\label{Le-crys-DDD+}
Let $\PPP$ be a Dieudonn\'e display over an admissible ring
$R$ with $p=2$ and let $\PPP^+$ be its base change to $\DDD_R^+$.
Then\/ $\DD(\PPP^+)$ is naturally isomorphic to the 
restriction of\/ $\DD(\PPP)$ to $\Cris_{\adm}(R/\ZZ_2)$.
\end{Lemma}

\begin{proof}
For each
$(\Spec A\to\Spec B,\delta)$ in $\Cris_{\adm}(R/\ZZ_2)$,
the $\DDD^+_{B/A}$-window $\PPP^+_{B/A}$ is the base change of $\PPP_{B/A}$ by
the frame homomorphism $\DDD_{B/A}\to\DDD^+_{B/A}$ by its uniqueness.
The lemma follows easily.
\end{proof}

\begin{Remark}
\label{Re-crys-DDD+}
Lemma \ref{Le-crys-DDD+} does not imply that the infinitesimal deformations 
of $\PPP$ and of $\PPP^+$ coincide: Let $B$ be an admissible ring
with $4B=0$ and $2B\ne 0$ and let $R=B/2B$. The ideal
$2B$ carries the canonical divided powers
$\gamma$ and the trivial divided powers $\delta$.
Corollary \ref{Co-deform-DDD} applies to $(B\to R,\delta)$
but not to $(B\to R,\gamma)$, while Corollary \ref{Co-deform-DDD+}
and Lemma \ref{Le-crys-DDD+} apply to $(B\to R,\gamma)$ but not to $(B\to R,\delta)$.
\end{Remark}

\subsection{Passing to the limit}

The preceding considerations extend easily 
to the case of admissible topological rings
with a countable base of topology.
Let us begin with a standard lemma.
For a ring $A$ let $\V(A)$ be the category of projective $A$-modules of finite type.

\begin{Lemma}
\label{Le:V-limit}
Let $A=\varprojlim_{n\in\NN}A_n$ be an inverse limit of rings such that
the transition maps $\pi_n:A_{n}\to A_{n-1}$ are surjective with $\Ker(\pi_n)\subseteq\Rad(A_n)$.
Then the natural functor $\rho:\V(A)\to\varprojlim_n \V(A_n)$ is an equivalence.
\end{Lemma}

\begin{proof}
Since for $P\in\V(A)$ we have $P=\varprojlim_n(P\otimes_AA_n)$
the functor $\rho$ is fully faithful.
For a system of $P_n\in\V(A_n)$ with isomorphisms $P_n\otimes_AA_{n-1}\cong P_{n-1}$
we have to show that the $A$-module $P=\varprojlim_n P_n$ lies in $\V(A)$.
Choose a surjective homomorphism $q_1:A_1^r\to P_1$ and lift this to a compatible
system of homomorphisms $q_n:A_n^r\to P_n$. All $q_n$ are surjective by Nakayama's Lemma.
Let $S_n$ be the set of linear sections of $q_n$. 
Since $S_n$ carries a simply transitive action of
$\Hom(P_n,\Ker(q_n))$, the reduction maps $S_n\to S_{n-1}$ are surjective.
%The reduction maps $S_n\to S_{n-1}$ are easily seen to be surjective. 
Thus the limit map $q:A^r\to P$ has a section, and we have $P\in\V(A)$.
This proves that $\rho$ is an equivalence. 
\end{proof}

For a ring $A$ let $\BT(A)$ be the category of $p$-divisible groups over $A$.

\begin{Lemma}
\label{Le:BT-limit}
For an inverse limit $A=\varprojlim_nA_n$ as in Lemma \ref{Le:V-limit}, 
the natural functor $\upsilon:\BT(A)\to\varprojlim_n\BT(A_n)$ is an equivalence.
\end{Lemma}

\begin{proof}
Cf.\ \cite[Ch.~II, Le.~4.16]{Messing-Crys}.
The functor $\rho$ of Lemma \ref{Le:V-limit}  preserves tensor products,  
and a complex $P\to P'\to P''\to 0$
in $\V(A)$ is exact if and only if its reduction to $A_1$ is exact. 
As in \cite[Ch.~II, Le.~4.16]{Messing-Crys} it follows that $\upsilon$ is an equivalence.
\end{proof}

For an admissible topological ring $R$ let $\DDD_R=\varprojlim_N\DDD_{R/N}$
where $N$ runs through the open ideals of $R$ contained in $\NNN_R$. 
As before, $\DDD_R$-windows are called Dieudonn\'e displays over $R$.

\begin{Lemma}
\label{Le:Disp-limit}
If $R$ is an admissible topological ring with a countable basis of topology, 
then Dieudonn\'e displays (or $p$-divisible groups) over $R$ are equivalent to 
compatible systems of Dieudonn\'e displays (or $p$-divisible groups) over ${R/N}$
for each open ideal $N$ contained in $\NNN_R$.
\end{Lemma}

\begin{proof}
One can write $R=\varprojlim_{n\in\NN}R_n$ for a surjective system of admissible
rings $R_n$ with $R_{\red}=(R_n)_{\red}$ for each $n$. 
Then the case of $p$-divisible groups follows from Lemma \ref{Le:BT-limit},
and the case of Dieudonn\'e displays follows from Lemma \ref{Le:V-limit}
applied to $R$ and to $\WW(R)=\varprojlim_n\WW(R_n)$; 
here the successive kernels are nil-ideals due to Lemma \ref{Le-W(N)}. 
Cf.\  \cite[Lemma 2.1]{Lau-Frames}.
\end{proof}

%---------------------------------------------------------------

\section{From $p$-divisible groups to Dieudonn\'e displays}
\label{Se-bt2disp}

In this section we define a functor from $p$-divisible
groups over odd admissible rings to Dieudonn\'e displays.
In the non-odd case there is a $v$-stabilised version of this functor,
which will serve as a first step towards the true functor
in the next section. We begin with some preparations.

\subsection{Finiteness over admissible rings}
\label{Subse-finiteness}

We show that the categories of $p$-divisible groups or Dieudonn\'e displays over an admissible ring $R$ are the direct limit of the corresponding categories over the finitely generated $W(R_{\red})$-subalgebras of $R$, with fully faithful transition maps.

\begin{Prop}
\label{Pr-finiteness-disp}
Every Dieudonn\'e display over an admissible ring $R$
is defined over a finitely generated $W(R_{\red})$-subalgebra
of $R$. For an injective homomorphism of admissible rings
$R\to S$ such that $R_{\red}\to S_{\red}$ is bijective,
the base change of Dieudonn\'e displays from $R$ to $S$
is fully faithful.
\end{Prop}

\begin{proof}
For a ring $A$ let $\V(A)$ be the category of projective $A$-modules of finite type. Since the ring $\WW(R)$ is the filtered union of\/ $\WW(R')$ where $R'$ runs through the finitely generated $W(R_{\red})$-subalgebras of $R$, the category $\V(\WW(R))$ is equivalent to the direct limit over $R'$ of $\V(\WW(R'))$. Since a Dieudonn\'e display over $R$ can be given by $L,T\in\V(\WW(R))$ together with an $f$-linear automorphism $\Psi$ of $L\oplus T$, the first assertion of the proposition follows. Similarly, every homomorphism of Dieudonn\'e displays over $R$ is defined over some finitely generated $R'$. Thus for the second assertion we may assume that $\NNN_S^r=0$. Let $\bar S=S/\NNN_S^{r-1}$ and $\bar R=R/R\cap\NNN_S^{r-1}$. Let $R''\subseteq S$ be the inverse image of $\bar R\subseteq\bar S$. By induction on $r$, the base change of Dieudonn\'e displays from $\bar R$ to $\bar S$ is fully faithful. It follows that the base change from $R''$ to $S$ is fully faithful as well. By Corollary \ref{Co-deform-DDD}, using trivial divided powers, Dieudonn\'e displays over $R$ or over $R''$ are equivalent to Dieudonn\'e displays over $\bar R$ together with a lift of the Hodge filtration to $R$ or to $R''$, respectively. Since $R\to R''$ is injective, it follows that the base change of Dieudonn\'e displays from $R$ to $R''$ is fully faithful.
\end{proof}

For the case of $p$-divisible groups we first recall some standard facts.

\begin{Lemma}
\label{Le-ker-trunc}
Let $B\to A$ be a surjective ring homomorphism 
with kernel $I$ such that $pI=0$ and $x^p=0$ for all $x\in I$. 
For an affine flat group scheme $H$ over $B$, the kernel
of $H(B)\to H(A)$ is annihilated by $p$.
\end{Lemma}

\begin{proof}
Let $B_0=B/pB$ and $H_0=H\otimes_BB_0$. 
The abelian group $B_0\oplus I$ 
becomes a ring with multiplication
$(a\oplus i)(a'\oplus i')=aa'\oplus(ai'+a'i+ii')$,
and one can identify
$B\times_{A}B$ with $B\times_{B_0}(B_0\oplus I)$.
Since the evaluation of affine schemes commutes with fibered 
products of rings, we obtain an isomorphism of
abelian groups
$$
\Ker(H(B)\to H(A))\cong\Ker(H_0(B_0\oplus I)\to H_0(B_0)).
$$
The right hand side lies
in the kernel of the Frobenius $F_{H_0}$ of $H_0$, which lies 
in $H_0[p]$ since $V_{H_0}\circ F_{H_0}=p$ by \cite[ VII$_\text{A}$ 4.3]{SGA3}. 
This proves the lemma.
\end{proof}

\begin{Lemma}
\label{Le-surj-trunc}
Let $B\to A$ be a surjective ring homomorphism with kernel $I$ such that $p$ is nilpotent in $B$ and $I$ is a nil-ideal. For a $p$-divisible group $G$ over $B$, the homomorphism $G(B)\to G(A)$ is surjective.
\end{Lemma}

\begin{proof}
For given $x\in G_n(A)$, since $G_n$ is finitely presented there is a finitely generated ideal
$I'\subseteq I$ such that $x$ lifts to an element $x'\in G_n(B/I')$. Now we can use that $G$ is formally smooth by \cite[Th.~3.3.13]{Messing-Crys}.
\end{proof}

\begin{Lemma}
\label{Le-red-BT}
Let $B\to A$ be a surjective ring homomorphism whose kernel is bounded nilpotent and such that $p$ is nilpotent in $B$. Then there is a number $r$ such that for two $p$-divisible groups $G$ and $H$ over $B$, the reduction homomorphism $\Hom(G,H)\to\Hom(G_A,H_A)$ is injective with kernel annihilated by $p^r$.
\end{Lemma}

\begin{proof}
This is an easy consequence of Lemmas \ref{Le-ker-trunc} and \ref{Le-surj-trunc};
cf.\ the proof of \cite[Lemma 1.1.3]{Katz:Serre-Tate}.
\end{proof}

\begin{Prop}
\label{Pr-finiteness-bt}
Every $p$-divisible group over an admissible ring $R$ is defined over a finitely generated $W(R_{\red})$-subalgebra of $R$. For an injective homomorphism of admissible rings $R\to S$ such that $R_{\red}\to S_{\red}$ is bijective, the base change of $p$-divisible groups from $R$ to $S$ is fully faithful.
\end{Prop}

\begin{proof}
For a $p$-divisible group $G$ over $R$ let $G_0=G\otimes_RR_{\red}$. Using Lemma \ref{Le-red-BT} we chose $r$ such that for two $p$-divisible groups $G$ and $H$ over $R$, the cokernel of $\Hom(G,H)\to\Hom(G_0,H_0)$ is annihilated by $p^r$. Now let $G$ be given, let $G''$ be a lift of $G_0$ to $W(R_{\red})$ and let $G'=G''\otimes_{W(R_{\red})}R$. There are homomorphisms $\varphi:G'\to G$ and $\psi:G\to G'$ which both lift the multiplication $p^r:G_0\to G_0$. Thus $\varphi\psi$ and $\psi\varphi$ are multiplication by $p^{2r}$. We obtain an isomorphism $G\cong G'/K_G$ where $K_G\subseteq G'$ is a finite locally free group scheme annihilated by $p^{2r}$; see Lemma \ref{Le:isogeny} below. In particular $K_G$ is finitely presented, and the first assertion of the proposition follows. To prove the second assertion, we consider two $p$-divisible groups $G$ and $H$ over $R$ and a homomorphism $\varphi_0:G_0\to H_0$ over $R_{\red}=S_{\red}$. There is a unique lift of $p^r\varphi_0$ to a homomorphism $\psi:G\to H$, and there is a lift of $\varphi_0$ to $R$ if $\psi$ vanishes on $G[p^r]$. Since $R\to S$ is injective, this holds if and only if the scalar extension $\psi_S$ vanishes on $G_S[p^r]$, which is equivalent to the existence of a lift of $\varphi_0$ to $S$.
\end{proof}

\begin{Lemma}
\label{Le:isogeny}
Let $\varphi:G\to H$ and $\psi:H\to G$ be homomorphisms of $p$-divisible groups over a scheme $S$ with $\varphi\psi=p^n$ and $\psi\varphi=p^n$. Then $\Ker(\varphi)$ and $\Ker(\psi)$ are finite locally free group schemes.
\end{Lemma}

\begin{proof}
Clearly $\Ker(\varphi)$ and $\Ker(\psi)$ are finite group schemes of finite presentation. Thus we may assume that $S=\Spec R$ for a local ring $R$ with residue field $k$. Let $\Ker(\psi)=\Spec A$ and $G_n=\Spec B$. Choose elements $a_1,\ldots,a_{p^r}\in A$ which map to a $k$-basis of $A_k$, so they generate $A$ as an $R$-module. We have a surjective homomorphism of fppf sheaves $\varphi:G_n\to\Ker(\psi)$. It follows that $B_k$ is a locally free $A_k$-module of some rank $p^s$, thus a free $A_k$-module since $A_k$ is finite. Choose $b_1,\ldots,b_{p^s}\in B$ which map to an $A_k$-basis of $B_k$. The elements $a_ib_j\in B$ map to a $k$-basis of $B_k$. Since $B$ is a free $R$-module they form an $R$-basis of $B$. It follows that $A$ is free over $R$ with basis $a_i$.
\end{proof}

\subsection{Deformation rings}

Let $\Lambda\to K$ be a surjective ring homomorphism
with finitely generated kernel $I\subseteq\Lambda$ 
such that $\Lambda$ is $I$-adically complete. The ring $K$ is not
assumed to be a field. Let
$\Nil_{\Lambda/K}$ be the category of $\Lambda$-algebras
$A$ together with a homomorphism of $\Lambda$-algebras
$A\to K$ with nilpotent kernel. 
We consider covariant functors 
$$
F:\Nil_{\Lambda/K}\to({\text{sets}})
$$
with the following properties (cf.\ \cite{Schlessinger}):

\begin{enumerate}
\setcounter{enumi}{\value{equation}}
\item \stepcounter{equation}
\label{Ax-def-1}
The set $F(K)$ has precisely one element.
\item \stepcounter{equation}
\label{Ax-def-2}
For a surjective homomorphism $A_1\to A$ in $\Nil_{\Lambda/K}$ 
the induced map $F(A_1)\to F(A)$ is surjective.
\item \stepcounter{equation}
\label{Ax-def-3}
For each pair of homomorphisms $A_1\to A\leftarrow A_2$ in 
$\Nil_{\Lambda/K}$ such that one of them is surjective 
the natural map
$F(A_1\times_AA_2)\to F(A_1)\times_{F(A)}F(A_2)$
is bijective. Then for each $K$-module $N$ the set
$F(K\oplus N)$ is naturally a $K$-module. In particular,
$t_F=F(K[\varepsilon])$ is a $K$-module, which is 
called the tangent space of $F$.
\item \stepcounter{equation}
\label{Ax-def-4}
For each $K$-module $N$ the natural homomorphism of $K$-modules 
$
t_F\otimes_KN\to F(K\oplus N)
$
is bijective. 
\item \stepcounter{equation}
\label{Ax-def-5}
The $K$-module $t_F$ is finitely presented.
\end{enumerate}

\noindent
The first three conditions imply that the functor $N\mapsto F(K\oplus N)$
preserves exact sequences of $K$-modules.
Thus \eqref{Ax-def-4} is automatic if $N$ is finitely presented.
Moreover \eqref{Ax-def-1}-\eqref{Ax-def-4} imply that the $K$-module
$t_F$ is flat, so \eqref{Ax-def-5} implies that $t_F$
is projective. 

\begin{Prop}
Assume that $F$ satisfies \eqref{Ax-def-1}-\eqref{Ax-def-5}. 
Then $F$ is pro-repre\-sented 
by a complete $\Lambda$-algebra $B$. 
Let $\tilde t$ be a projective $\Lambda$-module of finite type
which lifts $t_F$. Then $B$ is isomorphic to the 
complete symmetric algebra $\Lambda[[\tilde t^*]]$,
where ${}^*$ means dual.
This is a power series ring over $\Lambda$ if $t_F$
is a free $K$-module.
\end{Prop}

\begin{proof}
The $K$-module $t_F$ is projective as explained above. Thus $\tilde t$ exists.
Let $B=\Lambda[[\tilde t^*]]$ and let $\bar B=K\oplus t_F^*$.
We have an obvious projection $B\to\bar B$. 
Let $\bar\xi\in F(\bar B)=t_F\otimes t_F^*=\End(t_F)$
correspond to the identity of $t_F$ and let
$\xi\in F(B)$ be a lift of $\bar\xi$.
We claim that the induced homomorphism of functors
$\xi:B\to F$ is bijective. Note that the functor
$B$ satisfies \eqref{Ax-def-1}-\eqref{Ax-def-5}. By induction it suffices
to show that if $A\to\bar A$ is a surjection
in $\Nil_{\Lambda/K}$ whose kernel $N$
is a $K$-module of square zero and if
$B(\bar A)\to F(\bar A)$ is bijective, then
$B(A)\to F(A)$ is bijective as well. We have a natural
isomorphism $A\times_{\bar A}A\cong A\times_K(K\oplus N)$.
It follows that the fibres of $B(A)\to B(\bar A)$
and the fibres of $F(A)\to F(\bar A)$ are
principal homogeneous sets under the $K$-modules
$B(K\oplus N)$ and $F(K\oplus N)$, respectively. 
The homomorphism $t_B\to t_F$ induced by $\xi$
is bijective by construction, so $B(K\oplus N)\to F(K\oplus N)$ 
is bijective, and the proposition follows.
\end{proof}

\begin{Cor}
A homomorphism of functors which satisfy \eqref{Ax-def-1}-\eqref{Ax-def-5} is an
isomorphism if and only if it induces an isomorphism 
on the tangent spaces.
\qed
\end{Cor}

\begin{Remark}
\label{Re-deform-functorial}
Let $\Lambda'\to K'$ be another pair as above 
and let $g:\Lambda'\to\Lambda$ be a ring homomorphism
which induces a homomorphism $\bar g:K'\to K$. 
For given functors $F$ on $\Nil_{\Lambda/K}$ and
$F'$ on $\Nil_{\Lambda'/K'}$, a homomorphism
$h:F\to F'$ over $g$ is a compatible system of maps
$$
h(A):F(A)\to F'(A\times_KK')
$$
for $A$ in $\Nil_{\Lambda/K}$; here $A\times_KK'$ 
is naturally an object of $\Nil_{\Lambda'/K'}$.
If $F$ and $F'$ satisfy \eqref{Ax-def-1}-\eqref{Ax-def-5} and if $B$ and $B'$
are the complete algebras which pro-represent $F$
and $F'$, respectively, then $h$ corresponds to a 
homomorphism $B'\to B$ compatible with $g$ and $\bar g$.
If $h(A)$ is bijective for all $A$, the induced
homomorphism $B'\hat\otimes_{\Lambda'}\Lambda\to B$ is an
isomorphism.
\end{Remark}

\begin{Defn}
Assume that $p$ is nilpotent in $K=\Lambda/I$ as above.
For a $p$-divisible group $G$ over $K$ let
$$
\Def_G:\Nil_{\Lambda/K}\to(\text{sets})
$$
be the deformation functor of $G$. This means that
$\Def_G(A)$ is the set of isomorphism classes
of $p$-di\-visible groups $G'$ over $A$
together with an isomorphism $G'\otimes_AK\cong G$.
Let $t_G=\Lie(G^\vee)\otimes_K\Lie(G)$.
\end{Defn}

\begin{Prop}
\label{Pr-Def-G}
The functor\/ $\Def_G$ is pro-represented by a
complete $\Lambda$-algebra $B$. 
Explicitly, if $\tilde t$ is a projective 
$\Lambda$-module which lifts $t_G$, then 
$B$ is isomorphic to the complete symmetric
algebra $\Lambda[[\tilde t^*]]$.
\end{Prop}

We note that Lemma \ref{Le:BT-limit} gives a universal $p$-divisible group over $B$.

\begin{proof}
The functor $\Def_G$ satisfies \eqref{Ax-def-1}-\eqref{Ax-def-5} with tangent space 
$t_G$ because for a surjective homomorphism $A'\to A$ in 
$\Nil_{\Lambda/K}$ whose kernel $N$ is a $K$-module of 
square zero and for $H\in\Def_G(A)$, the set of lifts of $H$ 
to $A'$ is a principally homogeneous set under
the $K$-module $t_G\otimes_KN$ by \cite{Messing-Crys}.
\end{proof}

\begin{Remark}
\label{Re-Def-G-functorial}
Let $g:\Lambda'\to\Lambda$ over $\bar g:K'\to K$ 
be as in Remark \ref{Re-deform-functorial}
such that $p$ is nilpotent in $K'$. Let $G$ over $K$
be the base change of a $p$-divisible group $G'$
over $K'$. For $A$ in $\Nil_{\Lambda/K}$
we have a natural map
$$
\Def_{G'}(A\times_KK')\to\Def_G(A).
$$
This map is bijective, and its inverse is a homomorphism
$\Def_G\to\Def_{G'}$ over $g$ in the sense of 
\ref{Re-deform-functorial}.
If $B$ and $B'$ pro-represent $\Def_G$ and 
$\Def_{G'}$, respectively, we get an isomorphism 
$B'\hat\otimes_{\Lambda'}\Lambda\cong B$.
\end{Remark}

\begin{Defn}
Assume that $K=\Lambda/I$ is an admissible ring.
For a Dieudonn\'e display $\PPP$ over $K$ we denote by
$$
\Def_{\PPP}:\Nil_{\Lambda/K}\to(\text{sets})
$$ 
the deformation functor of $\PPP$.
Let $t_\PPP=\Hom(Q/I_KP,P/Q)$. 
\end{Defn}

We are mainly interested in the case where $K$ is perfect
and $\Lambda=W(K)$. Then Dieudonn\'e displays over $K$
are displays because $\WW(K)=W(K)$.

\begin{Prop}
\label{Pr-Def-PPP}
The functor\/ $\Def_\PPP$ is pro-represented by a complete
$\Lambda$-algebra $B$. Explicitly, if $\tilde t$ is
a projective $\Lambda$-module which lifts $t_\PPP$, then
$B$ is isomorphic to the complete symmetric algebra
$\Lambda[[\tilde t^*]]$.
\end{Prop}

We note that Lemma \ref{Le:Disp-limit} gives a universal Dieudonn\'e display over $B$.

\begin{proof}
The functor $\Def_\PPP$ satisfies \eqref{Ax-def-1}-\eqref{Ax-def-5} with tangent space
$t_\PPP$ because for a surjective homomorphism $A'\to A$ in $\Nil_{\Lambda/K}$ whose kernel $N$ is a $K$-module of square zero and for $\PPP'\in\Def_\PPP(A)$, the set of lifts of $\PPP'$ to $A'$ is a principally homogeneous set under the $K$-module $t_\PPP\otimes_KN$ by Corollary \ref{Co-deform-DDD}.
\end{proof}

\begin{Remark}
\label{Re-Def-PPP-funct}
Let $g:\Lambda'\to\Lambda$ over $\bar g:K'\to K$ be as in
Remark \ref{Re-deform-functorial} such that $K$ and $K'$
are admissible rings. Assume that $\PPP$ is the base change
of a Dieudonn\'e display $\PPP'$ over $K'$.
If $B$ and $B'$ represent $\Def_\PPP$ and $\Def_{\PPP'}$,
respectively, then $B'\hat\otimes_{\Lambda'}\Lambda\cong B$. 
This is analogous to Remark \ref{Re-Def-G-functorial}.
\end{Remark}

\subsection{Crystals and frames}
\label{Subse-cryst-frame}

Let $\FFF=(S,I,R,\sigma,\sigma_1)$ be a frame as in
\ref{Subse-frames} such that 
$S$ and $R$ are $p$-adically complete, $S$ has no $p$-torsion, 
$I$ carries divided powers, and $\sigma=p\sigma_1$ on $I$. 
Thus $(S,\sigma)$ is a frame for each $R/p^nR$ in the
sense of \cite{Zink-Windows}.
By a well-known construction, the crystalline
Dieudonn\'e functor allows to associate to a $p$-divisible
group over $R$ an $\FFF$-window; this is explained in
the proof of \cite[Theorem 1.6]{Zink-Windows} 
for the Dieudonn\'e crystal of a nilpotent display, 
and in \cite{Kisin-crys, Kisin-2adic} for $p$-divisible groups.

The construction goes as follows. 
First, one can define a filtered $F$-$V$-module;
here it is not necessary to assume that $S$ has no $p$-torsion.

\begin{Constr}
\label{Const-fil-mod}
Let $\FFF=(S,I,R,\sigma,\sigma_1)$ be a frame such that $S$ and $R$ are $p$-adically complete,
$I$ is equipped with divided powers $\delta$ which are compatible 
with the canonical divided powers of $p$, 
and $\sigma=p\sigma_1$ on $I$.
Let $\delta'$ be the divided powers
on $I'=I+pS$ which extend $\delta$ and the canonical divided
powers of $p$. We assume that $\sigma$ preserves $\delta'$,
which is automatic if $S$ has no $p$-torsion. 
Then one can define a functor
\begin{gather*}
\Phi^o:(p\text{-divisible groups over $R$})\to
(\text{filtered $F$-$V$-modules over $\FFF$}) \\
G\mapsto(P,Q,F^\sharp,V^\sharp)
\end{gather*}
as follows. Let $R_0=R/pR$ and let $\sigma_0$ be its Frobenius endomorphism. 
For a given $p$-divisible group $G$ over $R$ put
$%\[
P=\DD(G)_{S/R}=\DD(G_0)_{S/R_0},
$ %\]
where $\DD(G)$ is the \emph{covariant}%
\footnote{\label{Ft-DD-covariant}
This differs from the notation of \cite{BBM}, 
where $\DD(G)$ is contravariant. 
One can switch between the covariant and contravariant crystals 
by passing to the dual of $G$ or of $\DD(G)$, which amounts to
the same by the crystalline duality theorem \cite[5.3]{BBM}.
}
Dieudonn\'e crystal of $G$, and let $Q$ be
the kernel of the natural map $P\to\Lie(G)$. 
Since $\sigma$ preserves $\delta'$, there is a natural isomorphism
\[
P^{(\sigma)}\cong\DD(\sigma_0^*G_0)_{S/R_0}.
\]
Thus we can define
$V^\sharp:P\to P^{(\sigma)}$ to be induced by the Frobenius 
$F:G_0\to\sigma_0^*G_0$ and $F^\sharp:P^{(\sigma)}\to P$
to be induced by the Verschiebung $V:\sigma_0^*G_0\to G_0$.
\end{Constr}

In the second step one associates $F_1$.

\begin{Prop}
\label{Pr-fil-mod-win}
Let $\FFF$ be a frame as in the beginning of \ref{Subse-cryst-frame}.
For a $p$-divisible group $G$ over $R$ let\/ 
$\Phi^o(G)=(P,Q,F^\sharp,V^\sharp)$ be the filtered $F$-$V$-module 
over $\FFF$ given by Construction \ref{Const-fil-mod}. 
There is a unique $F_1:Q\to P$ such that 
$(P,Q,F,F_1)$ is an $\FFF$-window, and it gives back $V^\sharp$
by the functor of Lemma \ref{Le-F-V-mod}.
\end{Prop}

\begin{proof}
We have functors $(P,Q,F,F_1)\mapsto
(P,Q,F^\sharp,V^\sharp)\mapsto(P,Q,F^\sharp)$, 
which are fully faithful; see Lemma \ref{Le-F-V-mod}.
Thus we have to show that $F(Q)$ lies in $pP$ so that 
$F_1=p^{-1}F$ is well-defined, that $F_1(Q)$ generates $P$,
and that the pair $(P,Q)$ admits a normal decomposition.
Since $R$ and $S$ are $p$-adically complete and since the kernel of 
$S/pS\to R/pR$ is a nil-ideal due to its divided powers,
all projective $R$-modules of finite type lift to $S$.
Thus a normal decomposition exists.
The existence of $F_1$ and the surjectivity of its
linearisation are proved in \cite[Le.~A.2]{Kisin-crys}
if $S$ is local with perfect residue field, but the
proof can be easily adapted to the general case.
To prove surjectivity, for each maximal ideal of $S$,
which necessarily comes from a maximal ideal $\Fm$ of $R$, 
we choose an embedding of $R/\Fm$ into a perfect field $k$. 
There is a ring homomorphism $\alpha:S\to W(k)$ which lifts 
$R\to k$ such that $f\alpha=\alpha\sigma$; it can be 
constructed as $S\to W(S)\to W(k)$. 
Then $\alpha$ is a homomorphism of frames $\FFF\to\WWW_k$,
and the assertion is reduced to the case of $\WWW_k$,
which is classical. 
\end{proof}

\begin{Remark}
The surjectivity of $F_1$ in the proof of Proposition
\ref{Pr-fil-mod-win} can 
also be deduced from the crystalline duality theorem.
Let $P=L\oplus T$ be a normal decomposition and let 
$\Psi:P\to P$ be given by $F_1$ on $L$ and by $F$ on $T$. 
We have to show that the linearisation 
$\Psi^\sharp:P^{(\sigma)}\to P$ is an isomorphism.
Let $(P',Q',F',F_1')$ be the quadruple associated to the 
Cartier dual $G^\vee$. The duality theorem gives
a perfect pairing $P\times P'\to S$ such that
$\left<F(x),F'(x')\right>=p\sigma{\left<x,x'\right>}$.
It follows that
$\left<F(x),F_1'(x')\right>=\sigma{\left<x,x'\right>}$
and
$\left<F_1(x),F'(x')\right>=\sigma{\left<x,x'\right>}$
whenever this makes sense.
The unique decomposition $P'=L'\oplus T'$ with 
$\left<L,L'\right>=0=\left<T,T'\right>$ is a normal
decomposition of $P'$, and the dual of the associated
$\Psi^{\prime\sharp}$ is an inverse of $\Psi^\sharp$.
\end{Remark}

\subsection{The Dieudonn\'e display associated to 
a $p$-divisible group}
\label{Subse-Ddisp-assoc-to}

For an admissible ring $R$ with $p\ge 3$
we consider the Dieudonn\'e frame $\DDD_R$ 
defined in Lemma \ref{Le-DDD-frame}.
The ring $\WW(R)$ is $p$-adically complete by the remark 
preceding Proposition \ref{Pr-WW-p-adic}.
By Lemma \ref{Le-pd-II**} the ideal $\II_R$ carries 
natural divided powers compatible with the canonical 
divided powers of $p$, and the induced divided powers on
the kernel of $\WW(R)\to R/pR$ are preserved by the Frobenius.
Thus Construction \ref{Const-fil-mod} gives a functor
\[
\Phi^o_R:(\text{$p$-divisible groups over $R$})
\to(\text{filtered $F$-$V$-modules over $\DDD_R$})
\]
which is compatible with base change in $R$.

\begin{Thm}
\label{Th-bt2disp}
For each admissible ring $R$ with $p\ge 3$
there is a unique functor
$$
\Phi_R:(\text{$p$-divisible groups over $R$})
\to(\text{Dieudonn\'e displays over $R$})
$$
which is compatible with base change in $R$ such that 
the filtered $F$-$V$-module over $\DDD_R$ associated to
$\Phi_R(G)$ is equal to $\Phi^o_R(G)$. In particular
there is a natural isomorphism\/ $\Lie(G)\cong\Lie(\Phi_R(G))$.
\end{Thm}

\begin{proof}
Clearly $\Phi^o_R(G)=(P,Q,F^\sharp,V^\sharp)$ 
is functorial in $R$ and $G$.
We have to show that there is a unique operator 
$F_1:Q\to P$ which is functorial in $R$ and $G$ such that 
$\Phi_R(G)=(P,Q,F,F_1)$ is a Dieudonn\'e display over $R$.

Let $K=R_{\red}$ and $\Lambda=W(K)$. Let $\bar G=G\otimes_RK$ 
and let $B$ be the complete $\Lambda$-algebra which 
pro-represents the functor $\Def_{\bar G}$ on 
$\Nil_{\Lambda/K}$; see Proposition \ref{Pr-Def-G}.
Let $\GGG$ be the universal deformation of $G$ over $B$.
If $I$ denotes the kernel of $B\to K$, we can define
$$
\Phi^o_B(\GGG)=\varprojlim_n\Phi^o_{B/I^n}(\GGG\otimes_BB/I^n).
$$ 
On the other hand, the ring $\WW(B)$ is $p$-adically complete 
by Proposition \ref{Pr-WW-p-adic}. 
Therefore we can also define $\Phi^o_B(\GGG)$ 
be a direct application of Construction \ref{Const-fil-mod},
and this agrees with the limit definition.
The ring $\WW(B)$ has no $p$-torsion because 
$B$ has no $p$-torsion. Thus by Proposition 
\ref{Pr-fil-mod-win} there is a unique operator 
$F_1$ which makes $\Phi_B^o(\GGG)$ into a 
Dieudonn\'e display $\Phi_B(\GGG)$ over $B$. 

By Proposition \ref{Pr-finiteness-bt} there is a unique 
homomorphism $B\to R$ of augmented algebras 
such that $G=\GGG\otimes_BR$ as deformations of $\bar G$.
Necessarily we define $\Phi_R(G)$ as the base change 
of $\Phi_B(\GGG)$ under $B\to R$. It remains to
show that $\Phi_R(G)$ is functorial in $R$ and $G$. 

Assume that $G$ is the base change of a $p$-divisible
group $G'$ over $R'$ under a homomorphism of admissible
rings $R'\to R$. Let $K'$, $\Lambda'$, $\bar G'$, $B'$, $\GGG'$
have the obvious meaning. We have a natural homomorphism
of $W(K')$-algebras $B'\to B$ together with an
isomorphism $\GGG'\otimes_{B'}B\cong\GGG$;
see Remark \ref{Re-Def-G-functorial}. By the uniqueness
of $F_1$ over $B$ we see that $\Phi_B(\GGG)$
coincides with the base change of $\Phi_{B'}(\GGG')$.
It follows that $\Phi_R(G)$ is the base change of
$\Phi_{R'}(G')$.

Assume that $u:G\to G_1$ is a homomorphism of $p$-divisible
groups over $R$. Let $\bar G_1$, $B_1$, $\GGG_1$ 
have the obvious meaning. We have to show that
$\Phi^o_R(u)$ commutes with $F_1$. We may assume that $u$ is an 
isomorphism because otherwise one can pass to the automorphism
$\left(\begin{smallmatrix}1&0\\u&1\end{smallmatrix}\right)$
of $G\oplus G_1$. This reasoning uses that the natural isomorphism
$\Phi^o_R(G\oplus G_1)=\Phi^o_R(G)\oplus\Phi^o_R(G_1)$
preserves the operators $F_1$ defined on the three modules,
which follows from the uniqueness of $F_1$ over the
ring which pro-represents $\Def_{\bar G}\times\Def_{\bar G_1}$.
An isomorphism $u:G\to G_1$ induces an isomorphism
$\bar u:\bar G\cong\bar G_1$, which gives an isomorphism
$B\cong B_1$ together with an isomorphism
$\tilde u:\GGG\otimes_BB_1\cong\GGG_1$
that lifts $\bar u$. By the uniqueness of $F_1$
over $B_1$ it follows that 
$\Phi_{B_1}^o(\tilde u)$ preserves $F_1$. Since
$u$ is the base change of $\tilde u$ by the
homomorphism $B_1\to R$ defined by $G_1$,
it follows that $\Phi_R^o(u)$ preserves $F_1$ as well.
\end{proof}

In order to analyse the action of the functors $\Phi_R$
on infinitesimal deformations we need the following 
extension of Theorem \ref{Th-bt2disp}.
Let $(R'\to R,\delta)$ be a divided power extension 
of admissible rings with $p\ge 3$ which is compatible 
with the canonical divided powers of $p$. 
Again, the ring $\WW(R'/R)$ is $p$-adically complete, and
$\II_{R'/R}$ carries natural divided powers compatible
with the canonical divided powers of $p$ such that
$f$ preserves their extension to the kernel of 
$\WW(R'/R)\to R/pR$. Thus
Construction \ref{Const-fil-mod} gives a functor
\[
\Phi_{R'/R}^o:
(p\text{-divisible groups over }R)
\to
(\text{filtered $F$-$V$-modules over }\DDD_{R'/R})
\]
which is compatible with base change in the triple $(R'\to R,\delta)$.

\begin{Thm}
\label{Th-bt2disp-pd}
Assume that $p\ge 3$.
For each divided power extension of admissible rings 
$(R'\to R,\delta)$ compatible with the 
canonical divided powers of $p$ there is a unique functor 
$$
\Phi_{R'/R}:(\text{$p$-divisible groups over $R$})
\to(\text{Dieudonn\'e displays for $R'/R$})
$$
which is compatible with base change in the triple 
$(R'\to R,\delta)$
such that the filtered $F$-$V$-module over $\DDD_{R'/R}$ 
associated to\/ $\Phi_{R'/R}(G)$ is equal to\/ $\Phi_{R'/R}^o(G)$.
\end{Thm}

\begin{proof}
For a given $p$-divisible group $G$ over $R$ we choose
a lift to a $p$-divisible group $G'$ over $R'$, which
exists by \cite[th\'eor\`eme 4.4]{Illusie}.
The Dieudonn\'e display $\Phi_{R'}(G')$ is well-defined
by Theorem \ref{Th-bt2disp}, and necessarily 
$\Phi_{R'/R}(G)$ is defined as the base change of 
$\Phi_{R'}(G')$ by the frame homomorphism 
$\DDD_{R'}\to\DDD_{R'/R}$. We have to show
that the operator $F_1$ on $\Phi^o_{R'/R}(G)$ 
defined in this way does not depend on the choice
of $G'$. If this is proved it follows easily that 
$\Phi_{R'/R}(G)$ is functorial in $G$ and in $(R'\to R,\delta)$; 
here instead of arbitrary homomorphisms of $p$-divisible 
groups it suffices to treat isomorphisms.

Let $K$, $\Lambda$, $\bar G$, $B$, $\GGG$ be
as in the proof of Theorem \ref{Th-bt2disp}.
We have an isomorphism $B\cong\Lambda[[t]]$ for 
a finitely generated projective $\Lambda$-module $t$.
Let $C=B\hat\otimes_\Lambda B$. The automorphism $\tau=\left(\begin{smallmatrix}1&1\\0&1\end{smallmatrix}\right)$
of $t\oplus t$ defines an isomorphism 
$$
C=\Lambda[[t\oplus t]]\xrightarrow\tau
\Lambda[[t\oplus t]]=B[[t_B]]
$$
under which the multiplication homomorphism 
$\mu:C\to B$ corresponds to the augmentation 
$B[[t_B]]\to B$ defined by $t_B\mapsto 0$.
Let $I$ be the kernel of $B\to K$, let $S$ be the
divided power envelope of the ideal $t_BB[[t_B]]\subseteq B[[t_B]]$,
and let $C'$ be the $I$-adic completion of $S$. 
By Lemma \ref{Le-pd-alg}, $\mu$ extends
to a divided power extension of admissible topological
rings $\mu':C'\to B$ which is topologically compatible
with the canonical divided powers of $p$.\footnote{
The construction of $C'$ seems to depend on choosing one
of the two natural maps $B\to C$, but actually it is independent
of the choice as the $I$-adic topologies defined on $S$ 
by these two maps coincide.
}

Assume that $G_1$ and $G_2$ are two lifts of $G$
to $p$-divisible groups over $R'$. 
Let $\GGG_1$ and $\GGG_2$ be the $p$-divisible
groups over $C$ which are the base change of $\GGG$ 
by the two natural homomorphisms $B\to C$.
By Proposition \ref{Pr-finiteness-bt} there are
well-defined homomorphisms $\bar\alpha:B\to R$
and $\alpha:C\to R'$ such that $G=\GGG\otimes_{B,\bar\alpha}R$
and $G_i=\GGG_i\otimes_{C,\alpha}R'$ as deformations of $\bar G$. 
We have a commutative diagram of rings
$$
\xymatrix@M+0.2em@C+1em{
C \ar[r] \ar@/^3ex/[rr]^\alpha \ar[dr]_\mu & 
C' \ar@{-->}[r]_{\alpha'} \ar[d]^{\mu'} &
R' \ar[d] \\
& B \ar[r]^{\bar\alpha} & R
}
$$
where $\alpha'$ is constructed as follows.
There is a unique homomorphism $\alpha'':S\to R'$ which extends $\alpha$
and which commutes with the divided powers on the kernel of $S\to B$ and of $R'\to R$. 
Each of the two homomorphisms $B\to C\to R'$ factors over $B/I^n$ for some $n$. 
Thus $\alpha''$ induces a homomorphism $S/I^nS\to R'$, which gives the required $\alpha'$.
We obtain the following commutative diagram of frames,
where $\iota$ is given by $C\to C'$, and $\iota'$ is 
given by the identity of $R'$.
$$
\xymatrix@M+0.2em{
\DDD_C \ar[r]^-{\iota} \ar[d]_{\alpha} &
\DDD_{C'/B} \ar[d]^{\alpha'} \\
\DDD_{R'} \ar[r]^-{\iota'} & \DDD_{R'/R}
}
$$
We have to show that the isomorphism of filtered
$F$-$V$-modules over $\DDD_{R'/R}$
\begin{equation}
\label{Eq-first-isom}
\iota'_*(\Phi^o_{R'}(G_1))
\cong\Phi^o_{R'/R}(G)\cong
\iota'_*(\Phi^o_{R'}(G_2))
\end{equation}
commutes with the operator $F_1$ defined
on the outer terms by the functor $\Phi_{R'}$.
The construction of $\Phi^o$ can be
extended to topological divided power extensions
of admissible topological rings by passing to
the projective limit.
Then \eqref{Eq-first-isom} arises by $\alpha'_*$ from
the natural isomorphism of filtered $F$-$V$-modules 
over $\DDD_{C'/B}$
\begin{equation}
\label{Eq-second-isom}
\iota_*(\Phi^o_{C}(\GGG_1))
\cong\Phi^o_{C'/B}(\GGG)\cong
\iota_*(\Phi^o_{C}(\GGG_2)).
\end{equation}
Since $\alpha_*$ preserves $F_1$ it suffices to show 
that \eqref{Eq-second-isom} commutes with the operators 
$F_1$ defined on the outer terms by the functor $\Phi_C$.
This follows from the relation $pF_1=F$ because 
$\WW(C'/B)$ has no $p$-torsion by Lemma \ref{Le-pd-alg}.
\end{proof}

\begin{Cor}
\label{Co-bt2disp-DD}
Assume that $p\ge 3$.
For a $p$-divisible group $G$ over an admissible ring $R$
with associated Dieudonn\'e display  $\PPP=\Phi_R(G)$ 
there is a natural isomorphism of crystals on $\Cris_{\adm}(R/\ZZ_p)$
$$
\DD(G)\cong\DD(\PPP)
$$
which is compatible with the natural isomorphism $\Lie(G)\cong\Lie(\PPP)$.
\end{Cor}

The category $\Cris_{\adm}$ and the crystal $\DD(\PPP)$
were defined in \ref{Subse-crystals}.

\begin{proof}
Let $(R'\to R,\gamma)$ be a divided power extension
of admissible rings with $p\ge 3$ compatible with the canonical
divided powers of $p$. 
The Dieudonn\'e display $\Phi_{R'/R}(G)$ 
given by Theorem \ref{Th-bt2disp-pd} is the unique 
lift of $\PPP$ under the crystalline frame homomorphism 
$\DDD_{R'/R}\to\DDD_R$. By the construction of the
underlying filtered $F$-$V$-module $\Phi_{R'/R}^o(G)$ 
and by the definition of the crystal $\KK(\PPP)$ in \ref{Subse-crystals} 
we obtain a natural isomorphism of $\WW(R'/R)$-modules
$$
\DD(G)_{\WW(R')/R}\cong\KK(\PPP)_{R'/R}.
$$ 
The tensor product with the projection $\WW(R'/R)\to R'$, 
which is a homomorphism of divided power extensions of $R$, 
gives a natural isomorphism of $R'$-modules 
$\DD(G)_{R'/R}\cong\DD(\PPP)_{R'/R}$ which is compatible with
the natural isomorphism $\Lie(G)\cong\Lie(\PPP)$.
\end{proof}

Now Theorem \ref{Th-B} for odd primes can be deduced quite formally:

\begin{Cor}
\label{Co-bt2disp-DD-gen}
Assume that $p\ge 3$.
For a $p$-divisible group $G$ over an admissible ring $R$
with associated Dieudonn\'e display $\PPP=\Phi_R(G)$ there is
a natural isomorphism of crystals on $\Cris_{\adm}(R)$
\[
\DD(G)\cong\DD(\PPP)
\]
which is compatible with the natural isomorphism $\Lie(G)\cong\Lie(\PPP)$.
\end{Cor}

Here the covariant Dieudonn\'e crystal $\DD(G)$ can be defined for 
divided power extensions that are not necessarily compatible
with the canonical divided powers of $p$ by \cite[Ch.II \S 9]{Mazur-Messing};
see also \cite[\S 1.4]{BBM}.

\begin{proof}
Let $\DD'(G)=\DD(\Phi_R(G))$.
Consider a divided power extension $R'\to R$ of admissible rings
which need not be compatible with the canonical divided powers of $p$.
We claim that for two lifts  $G_1$ and $G_2$ of $G$ to $R'$
the following diagram of natural isomorphisms commutes.
\begin{equation}
\label{Diag-gen-pd}
\xymatrix@M+0.2em{
\DD(G_2)_{R'/R'}  \ar[d]^\sim &
\DD(G)_{R'/R} \ar[l]_-\sim \ar[r]^-\sim &
\DD(G_1)_{R'/R'} \ar[d]^\sim  \\
\DD'(G_2)_{R'/R'} \ar[r]^-\sim &
\DD'(G)_{R'/R} &
\DD'(G_1)_{R'/R'} \ar[l]_-\sim  
}
\end{equation}

This gives a well-defined isomorphism 
$\alpha(G):\DD(G)_{R'/R}\cong\DD'(G)_{R'/R}$.
It is easy to see that $\alpha(G)$ is compatible with the natural isomorphism
$\Lie(G)\cong\Lie(\PPP)$, that $\alpha(G)$ is functorial in the 
divided power extension $R'\to R$
and that $\alpha(G\oplus H)=\alpha(G)\oplus\alpha(H)$.
In order to show that $\alpha$ is functorial in $G$ it suffices to consider isomorphisms.
So let $u:G\to H$ be an isomorphism of $p$-divisible groups over $R$.
We can choose lifts $G_1$ of $G$ and $H_1$ of $H$ to $R'$ such that
$u$ extends to $\tilde u:G_1\to H_1$. Then the following diagram shows
that $\alpha$ commutes with $u$.
\[
\xymatrix@M+0.2em{
\DD(G)_{R'/R} \ar[r]^-\sim \ar[d]_{\DD(u)} &
\DD(G_1)_{R'/R'} \ar[r]^-\sim \ar[d]_{\DD(\tilde u)} &
\DD'(G_1)_{R'/R'} \ar[r]^-\sim \ar[d]^{\DD'(\tilde u)} &
\DD'(G)_{R'/R} \ar[d]^{\DD'(u)} \\
\DD(H)_{R'/R} \ar[r]^-\sim &
\DD(H_1)_{R'/R'} \ar[r]^-\sim &
\DD'(H_1)_{R'/R'} \ar[r]^-\sim &
\DD'(H)_{R'/R}
}
\]

It remains to show that \eqref{Diag-gen-pd} commutes.
Let $K,\Lambda,\bar G,B$ be as in the proof of Theorem \ref{Th-bt2disp}.
Let $C=B\hat\otimes_\Lambda B$ and $C'$ be as in the proof of 
Theorem \ref{Th-bt2disp-pd} so that the multiplication homomorphism
$\mu:C\to B$ extends to a topological divided power extension  
$\mu':C'\to B$ of admissible topological rings without $p$-torsion
which is topologically compatible with the canonical divided powers of $p$.
We have homomorphisms $B\to R$ defined by $G$
and $C\to R'$ defined by $(G_1,G_2)$, which extend
to a homomorphism of divided power extensions from
$(C'\to B)$ to $(R'\to R)$. Thus \eqref{Diag-gen-pd} is
the base change of a similar diagram for $(C'\to B)$,
which commutes by Corollary \ref{Co-bt2disp-DD}.
\end{proof}

\subsection{A $v$-stabilised variant} 

Let $R$ be an admissible ring with $p=2$. 
The $v$-stabilised Zink ring $\WW^+(R)$ considered in 
\ref{Subse-v-stab-Zink} and in \ref{Subse-v-stab-D}
is $2$-adically complete, and its ideal $\II^+_R$
carries natural divided powers which are compatible
with the canonical divided powers of $2$. 
The proof of Theorem \ref{Th-bt2disp} with
$\WW^+$ in place of $\WW$ shows the following.

\begin{Prop}
\label{Pr-bt2disp+}
For each admissible ring $R$ with $p=2$ there is
a unique functor
$$
\Phi_R^+:(\text{$2$-divisible groups over $R$})
\to(\text{$\DDD^+_R$-windows})
$$
which is compatible with base change in $R$ such that
the filtered $F$-$V$-module over $\DDD_R^+$ associated to
$\Phi^+_R(G)$ is given by Construction \ref{Const-fil-mod}.
\qed
\end{Prop}

\begin{Cor}
\label{Co-bt2disp+}
For each admissible ring $R$ with $p=2$ and $2R=0$ there is a unique
functor
$$
\Phi_R:(\text{$2$-divisible groups over $R$})
\to(\text{Dieudonn\'e displays over $R$})
$$
which is compatible with base change in $R$ such that the filtered
$F$-$V$-module over $\DDD_R$ associated to $\Phi_R(G)$ is given by
Construction \ref{Const-fil-mod}.
\end{Cor}

\begin{proof}
Proposition \ref{Pr-bt2disp+} gives the functors $\Phi_R$
since $\DDD^+_R=\DDD_R$ when $2R=0$.
The uniqueness follows as in the proof of Theorem \ref{Th-bt2disp},
using $B/2B$ instead of $B$.
\end{proof}

Let $(R'\to R,\delta)$ be a divided power extension 
of admissible rings with $p=2$ which is compatible 
with the canonical divided powers of $2$. 
The ring $\WW^+(R'/R)$ is $2$-adically complete, and its ideal 
$\II^+_{R'/R}$ carries natural divided powers
compatible with the canonical divided
powers of $2$. The proof of
Theorem \ref{Th-bt2disp-pd} with
$\WW^+$ in place of $\WW$ gives the following.

\begin{Prop}
\label{Pr-bt2disp-pd+}
For each divided power extension of admissible rings
$(R'\to R,\delta)$ with $p=2$ such that $\delta$ is compatible
with the canonical divided powers of\/ $2$ there is
a unique functor
$$
\Phi^+_{R'/R}:(\text{$2$-divisible groups over $R$})
\to(\text{$\DDD^+_{R'/R}$-windows})
$$
which is functorial in the triple
$(R'\to R,\delta)$ such that the filtered $F$-$V$-module
over $\DDD^+_{R'/R}$ associated to $\Phi^+_{R'/R}(G)$
is given by Construction \ref{Const-fil-mod}. 
\qed
\end{Prop}

The proof of Corollary \ref{Co-bt2disp-DD} then shows the following.

\begin{Cor}
\label{Co-bt2disp-DD-2}
Assume that $p=2$. For a\/ $2$-divisible group $G$ over an
admissible ring $R$ with associated $v$-stabilised Dieudonn\'e
display $\PPP^+=\Phi^+_R(G)$ there is a natural isomorphism
of crystals on $\Cris_{\adm}(R/\ZZ_2)$
$$
\DD(G)\cong\DD^+(\PPP^+)
$$
which is compatible with the natural isomorphism 
$\Lie(G)\cong\Lie(\PPP^+)$.
\qed
\end{Cor}

There is no analogue of Corollary \ref{Co-bt2disp-DD-gen} for $\Phi_R^+$
because $\DD^+(\PPP^+)$ is only a crystal on $\Cris_{\adm}(R/\ZZ_2)$ 
and not on $\Cris_{\adm}(R)$; but see Corollary \ref{Co-DDG-DDP-2}.

%---------------------------------------------------------------

\section{From $2$-divisible groups to Dieudonn\'e displays}

In this section we construct a functor $\Phi_R$ from $p$-divisible groups
over an admissible ring $R$ with $p=2$ to Dieudonn\'e displays.
When $2R=0$ this has been done in the previous section,
and the extension to
all $R$ is unique, as will be shown in the end of this section.
The construction relies on the following definition of divided powers on 
the ideal $\II_R\subseteq\WW(R)$ when $4R=0$.

\subsection{Divided powers on Zink rings}
\label{Subse-strange-pd}

We note that for a $\ZZ_{(2)}$-algebra
$B$ and an ideal $\Fb\subseteq B$, divided powers on $\Fb$ 
are equivalent to a map $\gamma:\Fb\to\Fb$ such that
\begin{gather}
\renewcommand{\theenumi}{\roman{enumi}}
%\tag{i}
\label{Ax-pd-pr}
\gamma(xy)=x^2\gamma(y)\text{ for $x\in B$ and $y\in\Fb$,} \\
%\tag{ii}
\label{Ax-pd-sum}
\gamma(x+y)=\gamma(x)+xy+\gamma(y)\text{ for $x,y\in\Fb$.}
\end{gather}
Here \eqref{Ax-pd-pr} and \eqref{Ax-pd-sum} also give
$2\gamma(x)=x^2$ for $x\in\Fb$ since we can calculate
\[
4\gamma(x)=\gamma(2x)=\gamma(x+x)=2\gamma(x)+x^2.
\]

For an admissible ring $R$ with $p=2$, the canonical divided powers
on the ideal $I_R\subseteq W(R)$ defined by $\gamma(v(a))=v(a^2)$
induce divided powers on $\II_R\subseteq\WW(R)$ only if $2R=0$;
see \ref{Subse-pd}.
Using $\vv$ instead of $v$ we get a little further.

\begin{Prop}
\label{Pr-strange-pd}
For an admissible ring $R$ with $p=2$ we consider the map
\[
\gamma:\II_R\to\II_R,\qquad \gamma(\vv(a))=\vv(a^2).
\]

If $4R=0$, then $\gamma$ defines divided powers on $\II_R$ which
are compatible with the canonical divided powers of\/ $2$, and the 
corresponding extension of $\gamma$ to $\II_R+2\WW(R)$
is stable under the Frobenius $f$ of\/ $\WW(R)$.

If $8R=0$, let $\Kideal\subseteq\WW(R)$ be the set of all Witt vectors 
of the form $v([x])=(0,x,0,\ldots)$ with $x\in 4R$.
This is an ideal. Let $\tilde S=\WW(R)/\Kideal$. 
Then $\gamma$ induces divided powers on
the ideal\/ $\II_R/\Kideal$ of $\tilde S$, which can naturally be extended to
divided powers on $\II_R/\Kideal+2\tilde S$ that commute with the
endomorphism $\sigma$ on $\tilde S$ induced by $f$,
and the extended divided powers stabilise the ideal $2\tilde S$.
\end{Prop}

\begin{proof}
We will only consider the case $8R=0$ and show that the extended
divided powers satisfy $\gamma(2)=2-[4]$. Then the case $4R=0$ follows.

Since $4R$ is an ideal of square zero we have $\hat W(4R)=(4R)^{(\NN)}$
as $W(R)$-modules where $W(R)$ acts on the $i$-th component of the
right hand side by the $i$-th Witt polynomial $w_i$, and $f$ annihilates $\hat W(4R)$. 
Thus $\Kideal$ is an ideal of $\WW(R)$, and $f$ induces $\sigma:\tilde S\to\tilde S$.
Let us show that $\gamma$ factors over a map $\II_R/\Kideal\to\II_R$.
Indeed, for $a\in\WW(R)$ and $x\in 4R$ we have
\[
\gamma(\vv(a)+v([x]))=\gamma(\vv(a)+\vv([x]))=\vv((a+[x])^2)=\vv(a^2)=\gamma(\vv(a));
\]
here $v([x])=\vv([x])$ because $u_0$ maps to $1$ in $W(\FF_2)$ and thus $u_0[x]=[x]$.
Let us verify Axiom \eqref{Ax-pd-pr} for the map $\gamma:\II_R\to\II_R$.
For $a,b\in\WW(R)$ we have
\[
\gamma(a\vv(b))=\gamma(\vv(f(a)b))=\vv(f(a^2)b^2)
=a^2\vv(b^2)=a^2\gamma(\vv(b)).
\]
Consider now Axiom \eqref{Ax-pd-sum}. For $a,b\in\WW(R)$ we calculate
\[
\gamma(\vv(a)+\vv(b))
=\vv((a+b)^2)
=\gamma(\vv(a))+\vv(2ab)+\gamma(\vv(b))
\]
so that $\vv(2ab)$ has to be related with $\vv(a)\vv(b)$, which is
\[
\vv(a)\vv(b)=v(u_0a)v(u_0b)=v(2u_0^2ab)=\vv(2u_0ab).
\]
Since $2u_0=2-[4]$  we get 
\[
\vv(2ab)-\vv(a)\vv(b)=\vv([4]ab)=v([4a_0b_0])\in \Kideal.
\]
Thus \eqref{Ax-pd-sum} holds for $\gamma:\II_R/\Kideal\to\II_R/\Kideal$,
and $\gamma$ defines divided powers on this ideal.
We want to extend these to divided powers on the ideal
\[
\tilde I=\II_R/\Kideal+2\tilde S=\II_R/\Kideal+\hat W(2R)/\Kideal.
\]
Let 
\[
\Fb=\{(2a_0,4a_1,0,\ldots)\mid a_i\in R\}\subseteq\hat W(2R).
\]
This is an ideal of $\WW(R)$ with $\II_R\cap\Fb=\Kideal$, and we have
\[
\tilde I=\II_R/\Kideal\oplus\Fb/\Kideal.
\]
Thus the extension of $\gamma$ to $\tilde I$ corresponds to giving 
arbitrary divided powers on $\Fb/\Kideal\cong 2R$. 
We take $\gamma([2a])=[-2a^2]$ for $a\in R$.
Using $\vv(1)=2-[2]$ we obtain
\[
\gamma(2)=\gamma([2]+\vv(1))=[-2]+\vv(1)=[-2]+2-[2]=2-[4]
\]
in $\tilde S$ as announced.
Let us show that $\gamma\sigma=\sigma\gamma$ on $\tilde I$:
For $a\in\WW(R)$ we have
\begin{multline*}
\gamma\sigma(\vv(a))=
\gamma((2-[4])a)=
\gamma(2-[4])a^2=
\gamma(2)a^2\\
=(2-[4])a^2=
\sigma(\vv(a^2))=
\sigma\gamma(\vv(a)).
\end{multline*}
Finally we have $[4]=2[2]$ in $\Fb/\Kideal$, which implies that $\gamma(2)\in 2\tilde S$.
This finishes the proof of Proposition \ref{Pr-strange-pd}.
\end{proof}

\begin{Remark}
The proof shows that the extension of $\gamma$ is uniquely determined by
the condition that it commutes with $\sigma$. 
By choosing $\gamma([2a])=[2a^2]$ we get an extension
with $\gamma(2)=2$ but which does not commute with $\sigma$.
\end{Remark}

Let $R$ be an admissible ring with $4R=0$.
Proposition \ref{Pr-strange-pd} implies that its
Dieudonn\'e frame $\DDD_R$ satisfies the hypotheses
of Construction \ref{Const-fil-mod} so that we obtain
a functor
\[
\Phi_R^o:(2\text{-divisible groups over }R)
\to(\text{filtered $F$-$V$-modules over }\DDD_R).
\]
However, we cannot argue as in Theorem \ref{Th-bt2disp}
in order to get a $\DDD_R$-window because the
divided powers on $\II_R$ do not exist for universal deformation rings,
and thus Proposition \ref{Pr-fil-mod-win} cannot be applied directly.
The following modification will be sufficient for our purpose.

\subsection{A frame lift}
\label{Subse-extension}

Assume we are given a strict frame homomorphism
\[
\FFF'=(S',I',R',\sigma',\sigma_1')\xrightarrow\pi\FFF=(S,I,R,\sigma,\sigma_1)
\]
such that  $\pi:S'\to S$ and $I'\to I$ are surjective, and an
ideal $\Kideal\subseteq\Ker(\pi)$ which is stable under $\sigma'$.
Let
\[
\tilde S=S'/\Kideal,\qquad
J=\Ker(\pi)/\Kideal,\qquad
\tilde I=\Ker(\tilde S\to R),
\]
thus $S=\tilde S/J$ and $R=\tilde S/\tilde I$.
Let $\tilde\sigma:\tilde S\to\tilde S$ be the homomorphism induced by $\sigma'$,
let $\theta'\in S'$ be the element defined by the relation 
$\sigma'=\theta'\sigma'_1$ on $I'$, and let 
$\tilde\theta\in\tilde S$ and $\theta\in S$ be its images.
We assume that $\FFF$ satisfies the
conditions of Construction \ref{Const-fil-mod}, i.e.\
$S$ and $R$ are $p$-adically complete, 
$I$ carries divided powers compatible with the canonical
divided powers of $p$ and with $\sigma$, and $\theta=p$. 
Then Construction \ref{Const-fil-mod} gives a functor 
\[
\Phi^o:(p\text{-divisible groups over }R)
\to(\text{filtered $F$-$V$-modules over }\FFF).
\]
We also assume that the following conditions are satisfied.

\begin{enumerate}
\setcounter{enumi}{\value{equation}}
\item \stepcounter{equation}
\label{Ax-J}
We have $\tilde\sigma(J)=0$ and $J=\{x\in\tilde S\mid px=0\}$.
\item \stepcounter{equation}
\label{Ax-unit}
We have $\tilde\theta=p\tilde u$ for a unit $\tilde u\in\tilde S$.
\item \stepcounter{equation}
\label{Ax-pd}
The ideal $\tilde I+p\tilde S$ is equipped with divided powers
which lift the given divided powers on $I+pS$, 
which commute with $\tilde\sigma$,
and which stabilise the ideal $p\tilde S$.
\item \stepcounter{equation}
\label{Ax-Fa}
There is an ideal $\Fa\subseteq S$ with $\sigma(\Fa)\subseteq\Fa\subseteq\Rad S$ 
such that the ring $S/\Fa$ has no $p$-torsion.
\end{enumerate}
If $S$ has no $p$-torsion one can take $\FFF'=\FFF$ and all
axioms are clear. The following extends Proposition \ref{Pr-fil-mod-win}.
Note that the prime $p$ is arbitrary here.

\begin{Prop}
\label{Pr-win-frame-lift}
In this situation there is a well-defined functor
\[
\Phi:(p\text{-divisible groups over }R)\to(\FFF\text{-windows})
\]
such that for $\Phi^o(G)=(P,Q,F^\sharp,V^\sharp)$ 
the filtered $F$-$V$-module associated to $\Phi(G)$ 
is equal to $(P,Q,F^\sharp,uV^\sharp)$
where $u\in S$ is the image of $\tilde u\in\tilde S$.
\end{Prop}

\begin{proof}
Conditions \eqref{Ax-J} and \eqref{Ax-unit} imply that
multiplication by $\tilde\theta$ on $\tilde S$ 
induces an injective map $\tilde\theta:S\to\tilde S$
with image $\tilde\theta\tilde S=p\tilde S$.
Moreover $\tilde\sigma$ induces a homomorphism
$\tilde\sigma:S\to\tilde S$ that lifts $\sigma$.
The relation $\theta'\sigma_1'=\sigma'$ on $I'$ gives
$\tilde\theta\circ\sigma_1=\tilde\sigma$ as maps $I\to\tilde S$.

Let $G$ be a $p$-divisible group over $R$ and let $\Phi^{o}(G)$
be as above, i.e.\ 
\[
P=\DD(G)_{S/R}=\DD(G_{R_0})_{S/R_0}
\] 
with $R_0=R/pR$, the submodule $Q\subseteq P$ is the kernel of
$P\to\Lie G$, and $F^\sharp$ and $V^\sharp$ are induced 
by the Verschiebung and Frobenius of $G_{R_0}$.
The proof of \cite[A2]{Kisin-crys} shows that $F(Q)\subseteq pP$.
Let us recall the argument:
For $S_0=S/pS$, the kernel of $S_0\to R_0$ is a nil-ideal because it carries divided powers.
By \cite[th\'eor\`eme~4.4]{Illusie} there is a lift $G_{S_0}$ of $G_{R_0}$ to $S_0$, 
and we have $P=\DD(G_{S_0})_{S/S_0}$. Let $Q_1=\Ker(P\to\Lie G_{S_0})$.
Then $Q\subseteq Q_1+IP$, and the image of $F$ applied to both summands lies in $pP$.

Since $pJ=0$ and $S$ is $p$-adically complete, so is $\tilde S$. By \eqref{Ax-pd} we can define
\[
\tilde P=\DD(G)_{\tilde S/R}=\DD(G_{R_0})_{\tilde S/R_0}.
\] 
Here we use the (dual of the) Dieudonn\'e crystal of \cite[Ch.II \S 9]{Mazur-Messing},
which is defined for divided power extensions that are not necessarily compatible
with the canonical divided powers of $p$; see also \cite[\S 1.4]{BBM}.
Let $\tilde Q\subseteq\tilde P$ be the kernel of $\tilde P\to\Lie G$;
this is the inverse image of $Q$ under the projection
$\tilde P\to\tilde P/J\tilde P=P$.
Again, the Verschiebung and Frobenius of $G_{R_0}$ induce 
$\tilde S$-linear maps $\tilde F^\sharp:\tilde P^{(\tilde\sigma)}\to\tilde P$ 
and $\tilde V^\sharp:\tilde P\to\tilde P^{(\tilde\sigma)}$.
Since the divided powers stabilise the ideal $p\tilde S$, the argument
of \cite[A2]{Kisin-crys} again shows that $\tilde F(\tilde Q)\subseteq p\tilde P=\tilde\theta\tilde P$, 
where $\tilde F:\tilde P\to\tilde P$ is the $\tilde\sigma$-linear map corresponding to $\tilde F^\sharp$.
Since $\tilde\sigma$ annihilates $J$, the map $\tilde F$ 
induces a map $\tilde F:P\to\tilde P$ which lifts $F$.
Let $F_1:Q\to P$ be the composition
\[
F_1:Q\xrightarrow{\;\tilde F\;}\tilde\theta\tilde P\xleftarrow[\sim]{\;\;\tilde\theta\;\;}P
\]
i.e.\ $F_1=\tilde\theta^{-1}\circ\tilde F$. We define $\Phi(G)=(P,Q,F,F_1)$.
In order that this is an $\FFF$-window we have to verify that 
\begin{enumerate}
\renewcommand{\theenumi}{\roman{enumi}}
\item
\label{xcv-1}
for $x\in P$ and $a\in I$ we have $F_1(ax)=\sigma_1(a)F(x)$;
\item
\label{xcv-2}
the image $F_1(Q)$ generates $P$.
\end{enumerate}
Moreover, $uV^\sharp$ is the operator associated to $\Phi(G)$ 
as we have claimed if and only if
\begin{enumerate}
\renewcommand{\theenumi}{\roman{enumi}}
\addtocounter{enumi}{2}
\item
\label{xcv-3}
for $x\in Q$ we have $uV^\sharp(F_1(x))=1\otimes x$ in $P^{(\sigma)}$.
\end{enumerate}
The equation in \eqref{xcv-1} is equivalent
to $\tilde F(ax)=\tilde\theta(\sigma_1(a)F(x))$.
Since $\tilde F(ax)=\tilde\sigma(a)\tilde F(x)$ and 
$\tilde\sigma=\tilde\theta\circ\sigma_1$ this is clear.
To prove \eqref{xcv-2} it suffices to show that for each maximal
ideal $\Fm$ of $R$ and a perfect field extension $R/\Fm\subseteq k$
the vector space $\tilde P\otimes_{\tilde S}k$ is generated by $F_1(Q)$. 
Using \eqref{Ax-Fa} we get a sequence of $\sigma$-equivariant 
maps $S\to\tilde S/\Fa\to W(\tilde S/\Fa)\to W(k)$; 
the second arrow exists uniquely since $\tilde S/\Fa$ has no
$p$-torsion and carries a Frobenius lift induced by $\tilde\sigma$;
see \cite[IX, \S1.2, Prop.\ 2]{Bourbaki-Comm-Alg} and
the explanation following \cite[Th.~4]{Zink-Windows}.
By functoriality we are reduced to the case where $\FFF'=\FFF=\WWW_k$,
which is classical. Assertion \eqref{xcv-3} is equivalent to
$\tilde u\tilde V^\sharp(\tilde F(x))=\tilde\theta(1\otimes x)$ for $x\in Q$,
which holds since $\tilde V^\sharp(\tilde F(x))=p(1\otimes x)$
in $\tilde P^{(\tilde\sigma)}$ for all $x\in\tilde P$.
\end{proof}

Now we construct an example for Proposition \ref{Pr-win-frame-lift}
with $p=2$.
Let $A_{\red}$ be a perfect ring of characteristic $2$ and let $A=W(A_{\red})[[t]]$ 
where $t$ is a finitely generated projective $W(A_{\red})$-module.
Let $\Fm=(2,t)$ be the kernel of $A\to A_{\red}$.
We write $A_n=A/2^nA$ and $A_{n+}=A/2^n\Fm$.
Only the rings 
\[
A_{2+}\to A_{1+}\to A_1
\]
will play a role. 
We consider the frames 
$\FFF'=\DDD_{A_{2+}}\!\to\FFF=\DDD_{A_{1+}}$, i.e.:
\begin{align*}
S=\WW(A_{1+})&\qquad\, I=\II_{A_{1+}} \qquad\, R=A_{1+} \\
S'=\WW(A_{2+})&\qquad I'=\II_{A_{2+}} \qquad R'=A_{2+} 
\end{align*}
Then $\theta'=2-[4]$ in $S'$ and thus $\theta=2$ in $S$.
Let $\Kideal\subseteq S'$ 
be the ideal of all Witt vectors $v([x])$ with $x\in 4A_{2+}$
and let $\tilde S=S'/\Kideal$. As above we write
\[
J=\Ker(\tilde S\to S)=\hat W (2\Fm/4\Fm)/\Kideal
\]
and
\[
\tilde I=\Ker(\tilde S\to R)=(I'+\hat W(2\Fm/4\Fm))/\Kideal.
\]

\begin{Prop}
\label{Pr-frame-lift}
These data satisfy the axioms \eqref{Ax-J}-\eqref{Ax-Fa}.
\end{Prop}

\begin{proof}
The divided powers required in \eqref{Ax-pd} are given by
Proposition \ref{Pr-strange-pd}.
Since $2\Fm/4\Fm\subseteq A_{2+}$ is an ideal of square zero we have
\[
J':=\hat W(2\Fm/4\Fm)=(2\Fm/4\Fm)^{(\NN)}
\]
as $W(A_{2+})$-modules, where $W(A_{2+})$ acts on the $i$-th
component of the right hand side via the $i$-th Witt polynomial.
We have $\sigma'(J')=2J'=0$ and $J=J'/\Kideal$.
Thus $\tilde\sigma:\tilde S\to\tilde S$ is defined and vanishes on $J$, and $2J=0$.

\begin{Lemma}
\label{Le-mult-2}
Multiplication by $2$ induces an isomorphism of groups
\[
\hat W(2A_{1+})\xrightarrow\sim\hat W(4A_{2+})/\Kideal.
\]
\end{Lemma}

\begin{proof}
The divided Witt polynomials for the canonical divided powers of $2$
give an isomorphism $\Log:W(2A)\cong2A^\NN$. The composition
\[
W(2A)\xrightarrow{\Log}2A^\NN\to (2A/4A)^\NN
\]
is given by $(2a_0,2a_1,\ldots)\mapsto 2[a_0, a_0^2+a_1,a_1^2+a_2,\ldots]$,
while the composition 
\[
W(4A)\xrightarrow{\Log}4A^\NN\to (4A/8A)^\NN
\]
is simply $(4a_0,4a_1,\ldots)\mapsto 4[a_0,a_1,\ldots]$.
It follows that the homomorphism of the lemma is isomorphic to
the homomorphism
\[
A_{\red}^{(\NN)}\to A_{\red}^{(\NN)},\qquad
(a_0,a_1,\ldots)\mapsto(a_0,a_1^2+a_2,a_2^2+a_3,\ldots).
\]
Since $A_{\red}$ is perfect this map is bijective.
\end{proof}

Let us continue in the proof of Proposition \ref{Pr-frame-lift}.
To verify \eqref{Ax-J} let $x\in\tilde S$ with $2x=0$. 
Since $\WW(A_1)$ has no $2$-torsion we have $x\in\hat W(2A_{2+})/U$.
Lemma \ref{Le-mult-2} implies that $x\in J$, and \eqref{Ax-J} is proved.
Let $u=1-[2]$ in $\WW(A_{2+})$, which is a unit.
By the proof of Lemma \ref{Le-mult-2} we have $2u=2-(4,4,0,\ldots)
\equiv 2-[4]=\theta'$ modulo $\Kideal$, which proves \eqref{Ax-unit}.
In \eqref{Ax-Fa} we can take $\Fa=\hat W(2A_{1+})$.
\end{proof}

\subsection{The Dieudonn\'e display associated to a $2$-divisible group}

Let $p=2$ and let $u=1-[2]$ in $\WW(\ZZ_2)$.
We begin to construct the functor $\Phi_R$ in an initial case.
Recall that $\Phi^o_R$ was defined in the end of
\ref{Subse-strange-pd} when $4R=0$. 

\begin{Prop}
\label{Pr-Phi-initial}
For each admissible ring $R$ with $p=2$ and $2\NNN_R=0$ there is a functor 
\[
\Phi_R:(2\text{-divisible groups over }R)\to(\DDD_R\text{-windows})
\]
compatible with base change in $R$ such that for $\Phi^o_R(G)=(P,Q,F^\sharp,V^\sharp)$
the filtered $F$-$V$-module associated to $\Phi_R(G)$ 
is equal to $(P,Q,F^\sharp,uV^\sharp)$.
\end{Prop}

\begin{proof}
This is similar to the proof of Theorem \ref{Th-bt2disp}.
Propositions \ref{Pr-win-frame-lift} and \ref{Pr-frame-lift} give the desired 
system of functors $\Phi_R$ for topological admissible rings $R$ 
of the type $R=A_{1+}$ as above. 
For a $p$-divisible group $G$ over an admissible
ring $R$ as in the proposition let $\Lambda=W(R_{\red})$
and $\bar G=G\otimes_R R_{\red}$. Let $A$ be the
$\Lambda$-algebra that pro-represents the functor
$\Def_{\bar G}$ on $\Nil_{\Lambda/K}$ 
(this $A$ was denoted by $B$ in section \ref{Se-bt2disp}),
let $\GGG$ over $A$ be the universal
deformation, and let $\GGG_{1+}=\GGG\otimes_AA_{1+}$.
The unique homomorphism of $\Lambda$-algebras $A\to R$ 
with $G=\GGG\otimes_AR$ as deformations of $\bar G$
factors over a homomorphism $A_{1+}\!\to R$, and we define $\Phi_R(G)$
as the base change of $\Phi_{A_{1+}}(\GGG_{1+})$
under this map. 
We have to show that the operator $F_1$ attached to $\Phi^{o}_R(G)$
in this way is functorial in $G$ and in $R$. This is analogous to
the proof of Theorem \ref{Th-bt2disp}, using that $F_1$ is functorial
with respect to homomorphisms of rings of the type $A_{1+}$.
\end{proof}

For an admissible ring $R$ let $i:\DDD_R\to\DDD_R^+$ be the
natural homomorphism.

\begin{Prop}
\label{Pr-Phi-Phi+-initial}
Let $R$ be an admissible ring with $p=2$ and $2\NNN_R=0$.
For each $2$-divisible group $G$ over $R$ there is a natural
isomorphism of $\DDD_R^+$-windows 
\[
\Phi^+_R(G)\cong i_*\Phi_R(G).
\]
\end{Prop}

The functor $\Phi_R^+$ was defined in Proposition \ref{Pr-bt2disp+}.

\begin{proof}
Let us write $\Psi_R(G)=i_*\Phi_R(G)$ so that we have two functors 
\[
\Phi_R^+,\Psi_R:(2\text{-divisible groups over }R)\to(\DDD^+_R\text{-windows}).
\]
When $2R=0$, thus $\DDD_R=\DDD^+_R$, these functors coincide by the
uniqueness assertion of Corollary \ref{Co-bt2disp+}. The rest is quite formal.
Let $R_1=R/2R$.
For a $p$-divisible group $G$ over $R$ let $G_1=G\otimes_RR_1$ and
$\bar G=G\otimes_RR_{\red}$.
The canonical divided powers of $2$ make $R\to R_1$ into a divided power extension.
By Corollaries \ref{Co-deform-DDD+} and \ref{Co-bt2disp-DD-2}, the 
$\DDD_R^+$-windows $\Phi^+_R(G)$ and $\Psi_R(G)$ correspond
to two lifts of the Hodge filtration of $G$ to $\DD(G_1)_{R/R_1}$.
Their difference is measured by a homomorphism of $R_1$-modules
\[
h'_G:V(G_1)\to\Lie(G_1)\otimes_{R}2R
\]
where $V(G)$ is the kernel of\/ $\DD(G)_R\to\Lie(G)$.
We have to show that $h'_G$ is zero for all $G$.
Since $h'_G$ is functorial in $R$ we may assume that 
$R=A_{1+}/\Fm^n$ for some $n\ge 2$
where $A$ is the universal deformation ring of $\bar G$.
Then $2R$ is a free $R_{\red}$-module of rank one, 
so $h'_G$ corresponds to an element 
\[
h_G\in\Hom(V(\bar G),\Lie(\bar G)).
\]
Now an injective homomorphism $R_{\red}\to R'_{\red}$
gives an injective homomorphism of the associated rings $R\to R'$, 
while a product decomposition $R_{\red}=\prod R_{i,\red}$ gives $R=\prod R_i$.
Since $R_{\red}$ embeds into the product of its localisations
at minimal prime ideals we may assume that $k:=R_{\red}$ is a field.
There is a deformation $\bar G'$ of $\bar G$ over
$R'_{\red}:=k[[x]]^{per}$ with ordinary generic fibre. Let $A'$ be
its universal deformation ring and let $G'$ over $R'=A'_{1+}/\Fm^n$
be given by the universal deformation. 
By functoriality it suffices to show that $h_{G'}=0$. 
Again we can pass to the field of fractions $k((x))^{per}$.
Thus we are left to show that $h_G=0$ if $G$ is ordinary over $R=A_{1+}/\Fm^n$
where $k=R_{\red}$ is a perfect field.
There is a deformation $G''$ of $G_k$ over $R'':=W_2(k)$
which decomposes into the direct sum of its \'etale and multiplicative part.
Let $R\to R''$ be the unique homomorphism such that $G''=G\otimes_RR''$ as deformations of $G_k$. 
Since this does not change $h_G$ we may replace $G$ by $G''$. 
Since $\Psi_R$ and $\Phi^+_R$ both preserve
direct sums we may assume that $G$ is \'etale or of multiplicative type.
Then $h_G$ vanishes since $\Hom(V(\bar G),\Lie(\bar G))$ is zero.
\end{proof}

\begin{Lemma}
\label{Le-frames-cart}
Let $R$ be an admissible ring with $p=2$ and let $R_{1+}=R/2\NNN_R$.
The commutative diagram of frames
\[
\xymatrix@M+0.2em{
\DDD_R \ar[r]^i \ar[d] & \DDD^+_R \ar[d] \\
\DDD_{R_{1+}} \ar[r]^i & \DDD^+_{R_{1+}}
}
\]
is Cartesian on each component of the frames, and the associated 
diagram of window categories is $2$-Cartesian.
\end{Lemma}

\begin{proof}
The vertical arrows are surjective, and the horizontal arrows
are injective with equal cokernel by Lemma \ref{Le-WW+} and its proof.
Thus the diagram of frames is Cartesian on each component. 
For a ring $A$ let $\V(A)$ be the category of projective $A$-modules of finite type. 
The functor 
\[
\V(\WW(R))\to \V(\WW^+(R))\times_{\V(\WW^+(R_{1+}))}\V(\WW(R_{1+}))
\]
is fully faithful since the diagram is Cartesian, and it is essentially
surjective since $\V(\WW(R))\to \V(\WW(R_{1+}))$ and
$\V(\WW^+(R))\to \V(\WW^+(R_{1+}))$ are bijective on
isomorphism classes and surjective on automorphism groups. 
It follows easily that the diagram of window categories is $2$-Cartesian.
\end{proof}

\begin{Thm}
\label{Th-Phi-2}
For each admissible ring $R$ with $p=2$ there is a functor
\[
\Phi_R:(2\text{-divisible groups over }R)\to(\DDD_R\text{-windows})
\]
compatible with base change in $R$ such that $\Phi_R$ is given by
Proposition \ref{Pr-Phi-initial} when $2\NNN_R=0$, and such that
there is a natural isomorphism of $\DDD_R^+$-windows
\[
\Phi^+_R(G)\cong i_*\Phi_R(G).
\]
\end{Thm}

\begin{proof}
This is clear from Propositions \ref{Pr-bt2disp+} and
\ref{Pr-Phi-Phi+-initial} and Lemma \ref{Le-frames-cart}.
\end{proof}

\begin{Cor}
\label{Co-DDG-DDP-2}
Let $p=2$. For each $2$-divisible group $G$ over an admissible ring $R$ 
with associated Dieduonn\'e display $\PPP=\Phi_R(G)$ there is a natural
isomorphism of crystals on $\Cris_{\adm}(R)$
\[
\DD(G)\cong\DD(\PPP)
\]
which is compatible with the natural isomorphism $\Lie(G)\cong\Lie(\PPP)$.
\end{Cor}

\begin{proof}
We have a natural isomorphism of crystals on $\Cris_{\adm}(R/\ZZ_2)$
\[
\DD(G)\cong\DD(\Phi_R^+(G))\cong\DD(\Phi_R(G))
\]
by Corollary \ref{Co-bt2disp-DD-2}, Theorem \ref{Th-Phi-2}, and Lemma \ref{Le-crys-DDD+}. 
The isomorphism of crystals on $\Cris_{\adm}(R)$ 
follows as in the proof of Corollary \ref{Co-bt2disp-DD-gen}.
\end{proof}

\subsection{Uniqueness of the functor $\Phi_R$}

\begin{Prop}
\label{Pr-Phi-unique-2}
Assume that for each admissible ring $R$ with $p=2$ we have a functor
\[
\Phi'_R:(2\text{-divisible groups over }R)\to(\DDD_R\text{-windows})
\]
compatible with base change in $R$ such that $\Phi_R'=\Phi_R$ when $2R=0$.
Then there is a natural isomorphism $\Phi'_R\cong\Phi_R$ which is functorial
in $R$ and equal to the identity when $2R=0$.
\end{Prop}

\begin{proof}
We first show that $\Phi_R'\cong\Phi_R$ when $4R=0$.
Let $R_1=R/2R$. For a $p$-divisible group $G$ over $R$ let
$G_1=G\otimes_RR_1$ and let 
\[
\PPP_1=\Phi_{R_1}(G_1)=(P,Q,F,F_1)
\]
be its Dieudonn\'e display. If we take the trivial divided powers
on the ideal $2R$, Corollary \ref{Co-deform-DDD} implies that
the difference between $\Phi_R(G)$ and $\Phi'_R(G)$ as lifts
of $\PPP_1$ is measured by a homomorphism 
\[
h'_G:Q/\II_{R_1}P\to P/Q\otimes_{R_1}2R.
\]
Let $V(G)=\Lie(G^\vee)^\vee$. By Corollary \ref{Co-bt2disp+} and by the 
construction of $\Phi^o_R(G)$
we can view $h'_G$ as a homomorphism
\[
h_G:V(G_1)\to\Lie(G_1)\otimes_{R_1}2R.
\] 
We want to show that $h_G=0$. 
We may assume that $R=A/(\Fm^n+4A)$ where $A$ 
is the universal deformation ring of $G\otimes_RR_{\red}$ and $\Fm$ 
is the kernel of $A\to R_{\red}$. 
As in the proof of Proposition \ref{Pr-Phi-Phi+-initial} one reduces to the case where
$k=R_{\red}$ is a field and $G$ is ordinary. 
Assume that $G$ is an extension $0\to\mu_{p^\infty}\to G\to\QQ_p/\ZZ_p\to 0$.
Then $V(G)=V(\QQ_p/\ZZ_p)=R$ and $\Lie(G)=\Lie(\mu_{p^{\infty}})=R$
so that $h_G\in 2R$. Thus $G\mapsto h_G$ defines a map 
$g:\uExt^1(\QQ_p/\ZZ_p,\mu_{p^\infty})\to\GG_a$
of functors on the category of local Artin rings with residue field $k$ and annihilated by $4$.
It is easy to see that $g$
is additive. Here $\uExt^1(\;)=\mu_{p^\infty}$, and it follows that $g=0$.
This implies easily that $h'_G=0$ when $G$ is ordinary. 
Thus $\Phi_R\cong\Phi'_R$ when $4R=0$.
If for some $n\ge 1$ we know that $\Phi_R\cong\Phi_R'$ when $2^nR=0$, the
same reasoning shows that $\Phi_R\cong\Phi_R'$ when $2^{n+1}R=0$,
and the proposition follows.
\end{proof}

%---------------------------------------------------------------

\section{Equivalence of categories}

Let $R$ be an admissible ring.
Dieudonn\'e displays over $R_{\red}$ are displays,
and they are equivalent to Dieudonn\'e modules
over $R_{\red}$ by Lemma \ref{Le-disp-perf}.
Under this equivalence, the functor $\Phi_{R_{\red}}$ 
corresponds to $\Phi_{R_{\red}}^o$.

\begin{Prop}
\label{Pr-Phi-Cartesian}
For an admissible ring $R$ the following diagram of categories is $2$-Cartesian.
$$
\xymatrix@M+0.2em@C+1em{
(\text{$p$-divisible groups over $R$}) 
\ar[r]^-{\Phi_R} \ar[d] &
(\text{Dieudonn\'e displays over $R$}) 
\ar[d] \\
(\text{$p$-divisible groups over $R_{\red}$}) 
\ar[r]^-{\Phi_{R_{\red}}} &
(\text{Dieudonn\'e modules over $R_{\red}$})
}
$$
\end{Prop}

\begin{proof}
The category of $p$-divisible groups (resp.\ Dieudonn\'e
displays) over $R$ is the direct limit of the corresponding
categories over all finitely generated $W(R_{\red})$-algebras
contained in $R$; see \ref{Subse-finiteness}.
Thus we may assume that the ideal $\NNN_R$ is nilpotent.
If $\Fa\subseteq R$ is an ideal equipped with nilpotent divided 
powers and if the proposition holds for $R/\Fa$ then it holds
for $R$. This follows from the comparison of crystals in
Corollaries \ref{Co-bt2disp-DD} and \ref{Co-DDG-DDP-2},
since lifts from $R/\Fa$ to $R$ of $p$-divisible groups and of
Dieudonn\'e displays are both classified by lifts of the Hodge filtration
by \cite{Messing-Crys} and by Corollary \ref{Co-deform-DDD}.
When $\Fa^2=0$ we can take the trivial divided powers on $\Fa$.
The proposition follows by induction on the order of nilpotence of $\NNN_R$.
\end{proof}

\begin{Remark}
Since $p$-divisible groups and Dieudonn\'e displays over a
perfect ring $K$ have universal deformation rings which are twisted
power series rings over $\Lambda=W(K)$, in order to prove
Proposition \ref{Pr-Phi-Cartesian} the case $R=K[\varepsilon]$ 
is sufficient.
In particular, for $p=2$ this means that as soon as the functors
$\Phi_R$ defined in Corollary \ref{Co-bt2disp+} when $2R=0$
are known to exist for all $R$, Proposition \ref{Pr-Phi-Cartesian} is automatic.
This reasoning does not apply to the functors $\Phi_R^+$
(which also extend the functors $\Phi_R$ for $2R=0$ to all $R$ but
which are not an equivalence in general)
because the deformation functors of $v$-stabilised Dieudonn\'e displays 
are not pro-representable.
\end{Remark}

We have the following result of Gabber, which is 
classical when $R_{\red}$ is a field.
It is also proved in \cite[Corollary~6.5]{Lau-Smoothness}.

\begin{Thm}
\label{Th-Phi-equiv-red}
The functor\/ $\Phi_{R_{\red}}$ is an equivalence of categories.
\qed
\end{Thm}

\begin{Cor}
\label{Co-Phi-equiv}
For every admissible ring $R$ the functor $\Phi_R$ is an equivalence of exact
categories. 
\end{Cor}

\begin{proof}
By Theorem \ref{Th-Phi-equiv-red} and Proposition \ref{Pr-Phi-Cartesian}
the functor $\Phi_R$ is an equivalence of categories. A short sequence
$0\to A\to B\to C\to 0$ of $p$-divisible groups or of Dieuonn\'e displays
over $R$ is exact if and only if all its scalar extensions to perfect fields
are exact. Thus $\Phi_R$ and its inverse preserve exact sequences since
this holds over perfect fields. 
\end{proof}

This proves Theorem \ref{Th-A}. 
Using Lemmas \ref{Le:BT-limit} and \ref{Le:Disp-limit} we also get:

\begin{Cor}
\label{Co-Phi-top-equiv}
For every admissible topological ring $R$ with a countable basis of topology,
$p$-divisible groups over $R$ are equivalent to Dieudonn\'e displays over $R$. \qed
\end{Cor}

Finally we note the following consequence of the crystalline 
duality theorem. The duality of windows is recalled 
in the end of \ref{Subse-frames}.

\begin{Cor}
Let $G$ be a $p$-divisible group over an admissible
ring $R$ and let $G^\vee$ be its Cartier dual. 
There is a natural isomorphism
$$
\Phi_R(G^\vee)\cong \Phi_R(G)^t.
$$
\end{Cor}

\begin{proof}
Assume first that $p$ is odd. The crystalline duality theorem \cite[5.3]{BBM} 
gives an isomorphism of filtered $F$-$V$-modules
$\Phi_R^o(G^\vee)^t\cong\Phi_R^o(G)$.
Since the functor from windows to filtered $F$-$V$-modules
preserves duality, the uniqueness part of
Theorem \ref{Th-bt2disp} implies that 
this isomorphism preserves $F_1$, i.e.\ it is 
an isomorphism of Dieudonn\'e displays $\Phi_R(G^\vee)^t\cong\Phi_R(G)$.
For $p=2$, using the uniqueness part of Corollary \ref{Co-bt2disp+}
we similarly get an isomorphism of Dieudonn\'e displays
$\Phi_R(G^\vee)^t\cong\Phi_R(G)$
when $2R=0$. Then Proposition \ref{Pr-Phi-unique-2} gives such an
isomorphism for all $R$.
\end{proof}

%---------------------------------------------------------------

\section{Breuil-Kisin modules}
\label{Se-Breuil}

We recall the main construction of \cite{Lau-Frames} without restriction on $p$.
Let $R$ be a complete regular local ring with maximal ideal $\Fm_R$ and
with perfect residue field $k$ of characteristic $p$.
Choose a representation $R=\FS/E\FS$ with
$$
\FS=W(k)[[x_1,\ldots,x_r]]$$ 
such that $E$ is a power series with constant term $p$. 
Let $J\subset\FS$ be the ideal generated by $x_1,\ldots,x_r$. 
Choose a ring endomorphism $\sigma:\FS\to\FS$
which lifts the Frobenius of $\FS/p\FS$
such that $\sigma(J)\subseteq J$. Let 
$\sigma_1:E\FS\to\FS$ be defined by $\sigma_1(Ex)=\sigma(x)$
for $x\in\FS$. These data define a frame
$$
\BBB=(\FS,E\FS,R,\sigma,\sigma_1).
$$
For each positive integer $a$ let $R_a=R/\Fm_R^a$ 
and $\FS_a=\FS/J^a$. We have frames
$$
\BBB_a=(\FS_a,E\FS_a,R_a,\sigma,\sigma_1)
$$
where $\sigma$ and $\sigma_1$ are induced by the
corresponding operators of $\BBB$. 

The frames $\BBB$ and $\BBB_a$ are related with the Witt 
and Dieudonn\'e frames of $R$ and of $R_a$ as follows.
Let $\delta:\FS\to W(\FS)$ be
the unique lift of the identity 
of $\FS$ such that $f\delta=\delta\sigma$,
or equivalently $w_n\delta=\sigma^n$ for $n\ge 0$;
see \cite[IX, \S1.2, Prop.\ 2]{Bourbaki-Comm-Alg} and
the explanation following \cite[Th.~4]{Zink-Windows}.
The composition of $\delta$ with the
projection $W(\FS)\to W(R)$ is a ring homomorphism 
$$
\varkappa:\FS\to W(R)
$$
which lifts the projection $\FS\to R$ such that
$f\varkappa=\varkappa\sigma$.
The same construction gives compatible homomorphisms
$$
\varkappa_a:\FS_a\to W(R_a)
$$
for $a\ge 1$, which induce $\varkappa$ in the projective limit.
Since the element $\varkappa(E)$ maps to zero in $R$ it lies in 
the image of $v:W(R)\to W(R)$. Let
$$
u=v^{-1}(\varkappa(E))=f_1(\varkappa(E)).
$$
We will denote the image of $u$ in $W(R_a)$ also by $u$.

\begin{Lemma}
\label{Le-varkappa-W}
The element $u\in W(R)$ is a unit. 
The homomorphisms $\varkappa$ and $\varkappa_a$
are $u$-homomorphisms of frames $\varkappa:\BBB\to\WWW_R$ 
and $\varkappa_a:\BBB_a\to\WWW_{R_a}$.
\end{Lemma}

\begin{proof}
Cf.\ \cite[Proposition 6.1]{Lau-Frames}.
Since $W(R)\to W(k)$ is a local homomorphism, 
in order to show that $u$ is a unit we can work with $\varkappa_1$,
i.e.\ consider the case where $R=k$ and $\FS=W(k)$. 
Then $E=p$ and $u=1$. In order that $\varkappa$ and
$\varkappa_a$ are $u$-homomorphisms of frames
we need that $f_1\varkappa=u\cdot\varkappa\sigma_1$. 
For $x\in\FS$ we calculate
$f_1(\varkappa(Ex))=f_1(\varkappa(E)\varkappa(x))=
f_1(\varkappa(E))\cdot f(\varkappa(x))=u\cdot\varkappa(\sigma(x))=
u\cdot\varkappa(\sigma_1(Ex))$ as required.
\end{proof}

Let $\bar\sigma$ be the semi-linear endomorphism of the
free $W(k)$-module $J/J^2$ induced by $\sigma$. 
Since $\sigma$ induces the Frobenius modulo $p$, 
$\bar\sigma$ is divisible by $p$.

\begin{Prop}
\label{Pr-varkappa-image}
The following conditions are equivalent.
\begin{enumerate}
\renewcommand{\theenumi}{\roman{enumi}}
\item
\label{Cond-varkappa}
The image of $\varkappa:\FS\to W(R)$ lies in $\WW(R)$.
\item
\label{Cond-delta}
The image of $\delta:\FS\to W(\FS)$ lies in $\WW(\FS)$.
\item
\label{Cond-sigma-1}
The endomorphism $p^{-1}\bar\sigma$ of $J/J^2$ 
is nilpotent modulo $p$.
\end{enumerate}
\end{Prop}

\begin{Remark}
In the special case $\sigma(x_i)=x_i^p$ the conditions of Propsition \ref{Pr-varkappa-image} hold.
This is easy to see directly:
We have $\delta(x_i)=[x_i]$, which gives \eqref{Cond-varkappa} and
\eqref{Cond-delta}, moreover \eqref{Cond-sigma-1} holds 
since $\bar\sigma$ is zero.
\end{Remark}

\begin{proof}[Proof of Proposition \ref{Pr-varkappa-image}]
For odd $p$ the equivalence between \eqref{Cond-varkappa}
and \eqref{Cond-sigma-1} is \cite[Proposition 9.1]{Lau-Frames}; 
its proof shows that 
\eqref{Cond-varkappa}
$\Rightarrow$
\eqref{Cond-sigma-1}
$\Rightarrow$
\eqref{Cond-delta}
$\Rightarrow$
\eqref{Cond-varkappa}.
The proof also applies for $p=2$ if \cite[Lemma 9.2]{Lau-Frames}
is replaced by the following Lemma \ref{Le-gr-tau}.
\end{proof}

\begin{Lemma}
\label{Le-gr-tau}
For $x\in\FS$ let $\tau(x)=(\sigma(x)-x^p)/p$.
Let $\Fm$ be the maximal ideal of $\FS$.
For $n\ge 0$
the map $\tau$ preserves $\Fm^nJ$ and induces a 
$\sigma$-linear endomorphism $gr_n(\tau)$ of the $k$-module
$gr_n(J)=\Fm^nJ/\Fm^{n+1}J$. The endomorphism
$gr_0(\tau)$ is equal to $p^{-1}\bar\sigma$ modulo $p$. 
For $n\ge 1$ there is a surjective $k$-linear map
$$
\pi_n:gr_n(J)\to gr_0(J)
$$
such that $gr_0(\tau)\pi_n=\pi_ngr_n(\tau)$ and 
such that $gr_n(\tau)$ vanishes on $\Ker(\pi_n)$.
In particular, $p^{-1}\bar\sigma$ is nilpotent
modulo $p$ if and only if $gr_0(\tau)$ is nilpotent,
which implies that $gr_n(\tau)$ is nilpotent for each $n$.
\end{Lemma}

\begin{proof}
We have $\sigma(J)\subseteq J^p+pJ\subseteq\Fm J$ 
and thus $\sigma(\Fm^nJ)\subseteq\Fm^{n+1}J$. 
It follows that 
$p\tau(\Fm^nJ)\subseteq p\FS\cap\Fm^{n+1}J=p\Fm^nJ$
and $\tau(\Fm^nJ)\subseteq\Fm^nJ$.
For $x,y\in\Fm^nJ$ the element $\tau(x+y)-\tau(x)-\tau(y)$
is a multiple of $xy$ and thus lies in $\Fm^{2n+1}J$.
Hence $\tau$ induces an additive
endomorphism $gr_n(\tau)$ of $gr_n(J)$. It
is $\sigma$-linear because for $a\in\FS$ and $x\in\Fm^nJ$ 
the element $\tau(ax)-\sigma(a)\tau(x)=\tau(a)x^p$ 
lies in $\Fm^{pn}J^p\subseteq\Fm^{n+1}J$.
Let us write $\sigma(x_i)=x_i^p+py_i$ with $y_i\in J$.
We have $\tau(x_i)=y_i$ and $p^{-1}\bar\sigma(x_i)\equiv y_i$
modulo $J^2$. Thus $gr_0(\tau)$ coincides with 
$p^{-1}\bar\sigma$ modulo $p$.

For each $n\ge 0$, a basis of $gr_n(J)$ is given by
all elements $p^b\underline x^{\underline c}$ with
$\underline c\in\NN^r$ and
$1\le|\underline c|\le n+1$ and $b+|\underline c|=n+1$.
Let $n\ge 1$ and define $\pi_n$ to be the $k$-linear
map with $\pi_n(p^nx_i)=x_i$ and 
$\pi_n(p^b\underline x^{\underline c})=0$ if $|\underline c|>1$. 
Then $gr_n(\tau)$ vanishes on $\Ker(\pi_n)$
because $\sigma(J)\subseteq\Fm J$, thus 
$\sigma(J^2)\subseteq\Fm^2 J^2$, and because for
$x\in\Fm^n J$ we have $x^p\in\Fm^{n+2}J$.
The relation $gr_0(\tau)\pi_n=\pi_ngr_n(\tau)$ holds
since $\tau(p^nx_i)\equiv p^{n-1}x_i^p+p^n y_i$ 
modulo $\Fm^{n+1}(J)$.
The last assertion of the lemma is immediate.
\end{proof}

\begin{Lemma}
\label{Le-varkappa-WW}
If the equivalent conditions of Proposition \ref{Pr-varkappa-image} hold, 
then $\varkappa$ and $\varkappa_a$ are $\uu$-homomorphisms of frames 
$$
\varkappa:\BBB\to\DDD_R,\qquad
\varkappa_a:\BBB_a\to\DDD_{R_a},
$$
where the unit $\uu\in\WW(R)$ is given by
$$
\uu=\vv^{-1}(\varkappa(E))=\ff_1(\varkappa(E)).
$$
In $W(R)$ we have\/ $\uu=u$ if $p$ is odd and 
$\uu=(v^{-1}(2-[2]))^{-1}u$ if $p=2$.
\end{Lemma}

\begin{proof}
The proof of Lemma \ref{Le-varkappa-W} with $f_1$
replaced by $\ff_1$ shows that $\uu$ is a unit
of $\WW(R)$ and that $\varkappa$ and $\varkappa_a$
are $\uu$-homomorphisms of frames as indicated. The
relation between $\uu$ and $u$ follows from the fact
that $\DDD_R\to\WWW_R$ is a $u_0$-homomorphism where
$u_0=1$ if $p$ is odd and $v(u_0)=2-[2]$ if $p=2$.
\end{proof}

\begin{Thm}
\label{Th-varkappa-crys}
If the equivalent conditions of Proposition \ref{Pr-varkappa-image}
hold, the frame homomorphisms $\varkappa:\BBB\to\DDD_R$ and
$\varkappa_a:\BBB_a\to\DDD_{R_a}$ are crystalline.
\end{Thm}

\begin{proof}
The proof for odd $p$ in \cite[Theorem 9.3]{Lau-Frames}
works almost literally for $p=2$ as well. Let us recall the
essential parts of the
argument. Fix an integer $a\ge 1$. One can define
a factorisation of the projection $\BBB_{a+1}\to\BBB_a$
into strict frame homomorphisms
\begin{equation}
\label{Eq-BBB-factor}
\BBB_{a+1}\xrightarrow{\iota}\tilde\BBB_{a+1}
\xrightarrow{\pi}\BBB_a
\end{equation}
such that 
$\BBB_{a+1}=(\FS_{a+1},\tilde I,R_a,\sigma,\tilde\sigma_1)$.
This determines $\tilde I$ and $\tilde\sigma_1$ uniquely as follows.
Let $\bar J^a=J^a/J^{a+1}$. 
We have $\tilde I=E\FS_{a+1}+\bar J^a$ and
$E\FS_{a+1}\cap\bar J^a=p\bar J^a$. 
The endomorphism $\bar\sigma$ of $\bar J^a$
induced by $\sigma$ is divisible by $p^a$,
and the operator $\tilde\sigma_1:\tilde I\to\FS_{a+1}$
is the unique extension of $\sigma_1$ such that
$\tilde\sigma_1(x)=p^{-1}\bar\sigma(x)$
for $x\in\bar J^a$. On the other hand, 
we consider the factorisation
\begin{equation}
\label{Eq-DDD-factor}
\DDD_{R_{a+1}}\xrightarrow{\iota'}
\DDD_{R_{a+1}/R_a}\xrightarrow{\pi'}\DDD_{R_a}
\end{equation}
with respect to the trivial divided powers on the
kernel $\Fm_R^{a}/\Fm_R^{a+1}$. 
Then $\varkappa_{a+1}$ is a $\uu$-homomorphism 
of frames $\tilde\BBB_{a+1}\to\DDD_{R_{a+1}/R_a}$.
Indeed, the only condition to be verified is that
for $x\in\bar J^a$ we have 
\begin{equation}
\label{Eq-non-triv}
\tilde\ff_1(\varkappa_{a+1}(x))=
\uu\cdot\varkappa_{a+1}(\tilde\sigma_1(x))
\end{equation}
in the $k$-vector space $\hat W(\Fm_R^a/\Fm_R^{a+1})$. 
On this space $\uu$ acts as the identity.
Let $y=(y_0,y_1,\ldots)$ in $W(\bar J^a)$ be defined
by $y_n=\tilde\sigma_1^n(x)$. Then $\delta(x)=y$
because the Witt polynomials give 
$w_n(y)=p^n\tilde\sigma_1^n(x)=\sigma^n(x)=w_n(\delta(x))$ 
as required. Thus $\varkappa_{a+1}(x)$ is the reduction of $y$.
Since $\tilde\ff_1$ acts on $\hat W(\Fm_R^a/\Fm_R^{a+1})$
by a shift to the left, the relation \eqref{Eq-non-triv}
follows. We obtain compatible $\uu$-homomorphisms of frames 
$\varkappa_*:\eqref{Eq-BBB-factor}\to\eqref{Eq-DDD-factor}$.
The homomorphisms $\pi$ and $\pi'$ are crystalline;
see the proof of \cite[Theorem 9.3]{Lau-Frames}.
Lifts of windows under $\iota$ and under $\iota'$
are both classified by lifts of the Hodge filtration
from $R_a$ to $R_{a+1}$ in a compatible way.
Thus if $\varkappa_a$ is crystalline then so is
$\varkappa_{a+1}$, and Theorem \ref{Th-varkappa-crys}
follows by induction, using that $\varkappa_1$ is
an isomorphism.
\end{proof}

Following the \cite{Vasiu-Zink} terminology, 
a Breuil window relative to $\FS\to R$ is 
a pair $(Q,\phi)$ where $Q$ is a free $\FS$-module of
finite rank and where $\phi:Q\to Q^{(\sigma)}$ is an
$\FS$-linear map with cokernel annihilated by $E$. 
For such $(Q,\phi)$ there is a unique linear map $\psi:Q^{(\sigma)}\to Q$ 
with $\psi\phi=E$; the pairs $(Q,\psi)$ are usually called
Breuil-Kisin modules or Kisin modules.
The category of $\BBB$-windows
is equivalent to the category of Breuil windows relative
to $\FS\to R$ by the assignment $(P,Q,F,F_1)\mapsto
(Q,\phi)$, where $\phi$ is the composition of the inclusion 
$Q\to P$ with the inverse of $F_1^\sharp:Q^{(\sigma)}\cong P$;
see \cite[Lemma 8.2]{Lau-Frames}.

\begin{Cor}
\label{Co-B-win}
If the equivalent conditions of Proposition \ref{Pr-varkappa-image}
hold, there is an equivalence of exact categories between
$p$-divisible groups over $R$ and Breuil windows relative
to $\FS\to R$. 
\end{Cor}

\begin{proof}
This is analogous to \cite[Corollary 8.3]{Lau-Frames},
using Corollary \ref{Co-Phi-equiv}.
\end{proof}

Following \cite{Vasiu-Zink} again, 
a Breuil module relative to $\FS\to R$
is a triple $(M,\phi,\psi)$ where $M$ is a finitely
generated $\FS$-module annihilated by a power of $p$
and of projective dimension at most one, and where
$\phi:M\to M^{(\sigma)}$ and $\psi:M^{(\sigma)}\to M$
are $\FS$-linear maps with $\phi\psi=E$ and
$\psi\phi=E$. If $R$ has characteristic zero, such
triples are equivalent to pairs $(M,\phi)$ or $(M,\psi)$;
see \cite[Lemma 8.6]{Lau-Frames}.
Again, the pairs $(M,\psi)$ are usually called Breuil-Kisin modules
or Kisin modules.

\begin{Cor}
If the equivalent conditions of Proposition \ref{Pr-varkappa-image}
hold, there is an equivalence of exact categories between
commutative
finite locally free group schemes of $p$-power order over $R$
and Breuil modules relative to $\FS\to R$.
\end{Cor}

\begin{proof}
This is analogous to \cite[Theorem 8.5]{Lau-Frames}.
\end{proof}

\begin{Example}
\label{Ex-motiv}
Let $R=W(k)$ and $\FS=W(k)[[t]]$ with $\sigma(t)=t^p$. 
Define $\FS\to R$ by $t\mapsto p$, thus $E=p-t$. We
have $\varkappa(E)=p-[p]$ and thus $u=v^{-1}(p-[p])$.
Assume that $p=2$. Then $u=u_0$, and $\BBB\to\DDD_R$ is
a strict frame homomorphism. This example has motivated
the definition of Dieudonn\'e displays for $p=2$.
\end{Example}

%---------------------------------------------------------------

\section{Breuil-Kisin modules and crystals}
\label{Se-Breuil-crys}

We keep the notation of section \ref{Se-Breuil}
and assume that the equivalent conditions of Proposition
\ref{Pr-varkappa-image} hold. Assume that
$R$ has characteristic zero. Let $S$ be the
$p$-adic completion of the divided power envelope of
the ideal $E\FS\subset\FS$ and let $I$ be the kernel 
of $S\to R$.
Since $\sigma:\FS\to\FS$ preserves the ideal $(E,p)$
it extends to $\sigma:S\to S$. It is easy to see that
$\sigma(I)\subseteq pS$, thus $\sigma:S\to S$ is a 
Frobenius lift again.

\begin{Prop}
\label{Pr-Breuil-crys}
Let $(Q,\phi)$ be a Breuil window relative to $\FS\to R$
and let $G$ be the associated $p$-divisible group over $R$;
see Corollary \ref{Co-B-win}. There is a natural isomorphism 
$$
\DD(G)_{S/R}\cong S\otimes_{\FS}Q^{(\sigma)}
$$
such that the Hodge filtration of\/ $\DD(G)_{S/R}$ 
corresponds to the submodule generated by 
$\phi(Q)+IQ^{(\sigma)}$, and the Frobenius
of\/ $\DD(G)_{S/R}$ corresponds to the
$\sigma$-linear endomorphism of $Q^{(\sigma)}$ 
defined by $x\mapsto 1\otimes\phi^{-1}(Ex)$. 
\end{Prop}

In Kisin's theory (when $R$ is one-dimensional) 
the analogous result is immediate from the construction. 
To prove Proposition \ref{Pr-Breuil-crys}
we consider the frame
$$
\SSS=(S,I,R,\sigma,\sigma_1)
$$
with $\sigma_1(x)=\sigma(x)/p$ for $x\in I$.
The inclusion $\FS\to S$ is a $u$-homomorphism of
frames $\iota:\BBB\to\SSS$ with $u=\sigma(E)/p\in S$. 
This element is a unit as required since the arrow $\FS\to R$ is mapped 
surjectively onto $W(k)\to k$, which gives a local homomorphism 
$S\to W(k)$ that maps $u$ to $1$. Recall that we have
frames $\DDD_R\to\DDD_R^+$ when $p=2$ and let us write
$\DDD^+_R=\DDD_R$ when $p\ge 3$. Then we have
the following commutative diagram of frames.
$$
\xymatrix@M+0.2em{
\BBB \ar[r]^\iota \ar[d]_\varkappa & 
\SSS \ar[d]^{\varkappa_S} \\
\DDD_R \ar[r] &
\DDD_R^+
}
$$
Indeed, since $\WW^+(R)\to R$ is a divided power extension
of $p$-adically complete rings,  the ring homomorphism $\varkappa:\FS\to\WW^+(R)$ 
extends to $\varkappa_S:S\to\WW^+(R)$, 
which is a strict frame homomorphism $\SSS\to\DDD_R^+$.
Here $\varkappa$ is crystalline by Theorem \ref{Th-varkappa-crys}.
The proof of Proposition \ref{Pr-Breuil-crys} will use the following fact.

\begin{Thm}
\label{Th-varkappa-S}
The frame homomorphism $\varkappa_S$ is crystalline.
\end{Thm}

This is a variant of the main result of \cite{Zink-Windows}.
It is easy to see that $S$ is an admissible topological ring 
in the sense
of Definition \ref{Def-top-admissible} if and only if $r=1$,
i.e.\ if $R$ is a discrete valuation ring.
In that case, the methods of \cite{Zink-Windows} 
apply directly, but additional effort is needed to prove 
Theorem \ref{Th-varkappa-S} in general.
The proof is postponed to the next section.

\begin{proof}[Proof of Proposition \ref{Pr-Breuil-crys}]
Let $\PPP_0=(P,Q,F,F_1)$ be the $\BBB$-window associated to
$(Q,\phi)$, thus $P=Q^{(\sigma)}$, the inclusion map $Q\to P$ 
is $\phi$, and $F:P\to P$ is
the $\sigma$-linear endomorphism of $Q^{(\sigma)}$
defined by $x\mapsto 1\otimes\phi^{-1}(Ex)$.
By definition we have $\Phi_R(G)=\varkappa_*(\PPP_0)$,
which implies that $\Phi_R^+(G)=\varkappa_{S*}\iota_*(\PPP_0)$;
here we use Theorem \ref{Th-Phi-2} when $p=2$.
On the other hand,
the frames $\SSS$ and $\DDD_R^+$ both satisfy the
hypotheses of the beginning of \ref{Subse-cryst-frame}.
Thus Construction \ref{Const-fil-mod} and Proposition \ref{Pr-fil-mod-win}
applied to $G$ give an $\SSS$-window $\PPP_1$ with an isomorphism
$\varkappa_{S*}(\PPP_1)\cong\Phi_R^+(G)$, using the characterisation of
$\Phi_R^+$ in Theorem \ref{Th-bt2disp} and Proposition \ref{Pr-bt2disp+}.
Since the base change functor $\varkappa_{S*}$ is
fully faithful by Theorem \ref{Th-varkappa-S},
the isomorphism $\varkappa_{S*}(\PPP_1)\cong\Phi_R^+(G)\cong
\varkappa_{S*}\iota_*(\PPP_0)$ descends 
to an isomorphism $\PPP_1\cong\iota_*(\PPP_0)$,
which proves the proposition.
\end{proof}

%---------------------------------------------------------------

\subsection{Proof of Theorem \ref{Th-varkappa-S}}

Let us begin with a closer look on the $p$-adically complete ring $S$.
For $m\ge 0$ let $S_{<m>}\subseteq S$
be the closure of the $\FS$-algebra 
generated by $E^i/i!$ for $i\le p^m$. 

\begin{Prop}
\label{Pr-S-m}
For $m\ge 1$ there is a surjective homomorphism of $\FS$-algebras
$\FS[[t_1,\ldots,t_m]]\to S_{<m>}$ defined by 
$t_i\mapsto E^{p^i}/p^i!$.
\end{Prop}

In particular, $S_{<m>}$ is a noetherian complete local ring.

\begin{Lemma}
\label{Le-lim-surjective}
Let $A$ be a noetherian complete local ring with a descending
sequence of ideals 
$A\supseteq\Fa_0\supseteq\Fa_1\supseteq\cdots$. 
Then $A\to\varprojlim_i A/\Fa_i$ is surjective.
\end{Lemma}

\begin{proof}
Let $\Fm$ be the maximal ideal of $A$. For each $r$, 
the images of $\Fa_i\to A/\Fm^r$ stabilise for $i\to\infty$
to an ideal $\bar\Fa_r\subseteq A/\Fm^r$.
We have 
$$
\varprojlim_i A/\Fa_i=\varprojlim_{i,r}A/(\Fa_i+\Fm^r)
=\varprojlim_r(A/\Fm^r)/\bar\Fa_r.
$$
Since the ideals $\bar\Fa_r$ form a surjective system,
taking the limit over $r$ of the exact sequences 
$0\to\bar\Fa_r\to A/\Fm^r\to(A/\Fm^r)/\bar\Fa_r\to 0$ 
proves the lemma. 
\end{proof}

\begin{proof}[Proof of Proposition \ref{Pr-S-m}]
Since the image of $E^{p^i}/p^i!$ in $S/p^nS$ is nilpotent,
there is a well-defined homomorphism
$\pi_{m,n}:\FS[[t_1,\ldots,t_m]]\to S/p^nS$
with $t_i\mapsto E^{p^i}/p^i!$.
By definition, $S_{<m>}$ is the projective limit over $n$ 
of the image of $\pi_{n,m}$. 
The proposition follows by Lemma \ref{Le-lim-surjective}.
\end{proof}

Let $K=W(k)\otimes\QQ$ and $\FS_\QQ=K[[x_1,\ldots,x_r]]$.
Since $\sigma:\FS\to\FS$ preserves the ideal $J=(x_1,\ldots,x_r)$ 
it extends to a homomorphism $\sigma:\FS_\QQ\to\FS_\QQ$.
For $r=1$ it is easy to describe $S$ and $S_{<m>}$ as
explicit subrings of $\FS_\QQ$ since instead of the divided 
powers of $E$ one can take the divided powers of $x_1^e$
where $e$ is defined by $pR=\Fm_R^e$.
For $r\ge 2$ the situation is more complicated.

\begin{Prop}
\label{Pr-S}
The natural embedding $\FS\to\FS_\QQ$
extends to an injective homomorphism $S\to\FS_\QQ$
that commutes with $\sigma$.
\end{Prop}

Thus $S_{<m>}$ is the image of $\FS[[t_1,\ldots,t_m]]\to\FS_\QQ$
as in Proposition \ref{Pr-S-m}.

\begin{proof}[Proof of Proposition \ref{Pr-S}]
Recall that $J=(x_1,\ldots,x_r)$ as an ideal of $\FS$.
Choose $E'\in J^e$ with $E-E'\in p\FS$ such that $e$ is maximal,
thus $p\in\Fm_R^e\setminus\Fm_R^{e+1}$.
Let us write $gr_{E'}^i(\FS)=E^{\prime i}\FS/E^{\prime i+1}\FS$ etc.

\begin{Lemma}
\label{Le-graded-inj}
The map of graded rings 
$gr_{E'}(\FS)\to gr_{E'}(\FS_{\QQ})$ is injective.
\end{Lemma}

\begin{proof}
It suffices to show that $\FS/E'\FS\to\FS_\QQ/E'\FS_\QQ$
is injective.
The choice of $E'$ implies that the image of $E'$
in the regular local rings $\FS/p\FS$ and $\FS_{\QQ}$ lies
in the same power of the maximal ideals. Therefore
the $k$-dimension of $\FS/(p\FS+E'\FS+J^n)$ 
is equal to the $K$-dimension of $\FS_{\QQ}/(E'\FS_{\QQ}+J^n\FS_\QQ)$.
Since the last module is isomorphic to $\FS/(E'\FS+J^n)\otimes\QQ$
it follows that $\FS/(E'\FS+J^n)$ is a free $W(k)$-module and
injects into $\FS_{\QQ}/(E'\FS_{\QQ}+J^n\FS_\QQ)$.
Since $\FS/E'\FS$ and $\FS_\QQ/E'\FS_\QQ$ are $J$-adically complete
the lemma follows.
\end{proof}

Let $S_0\subseteq\FS_\QQ$
be the $\FS$-algebra generated by $E^{\prime i}/i!$
for $i\ge 1$, or equivalently by $E^i/i!$ for $i\ge 1$,
so $S$ is the $p$-adic completion of $S_0$.
Let $S_{0,n}$ be the image of 
$S_{0}\to\FS_\QQ/E^{\prime n}\FS_\QQ$ 
and let $\tilde S=\varprojlim_nS_{0,n}$.
Each $S_{0,n}$ is a noetherian complete local ring with
residue field $k$ and thus a $p$-adically complete ring.
Since $S_{0,n}$ has no $p$-torsion it follows that
$\tilde S$ is $p$-adically complete.
We obtain a homomorphism $S\to\tilde S\subseteq\FS_\QQ$
which extends $S_0\subseteq\tilde S\subseteq\FS_\QQ$.

\begin{Lemma}
\label{Le-S0-tildeS}
We have $S_0\cap p\tilde S=pS_0$ inside $\tilde S$.
\end{Lemma}

\begin{proof}
Let $x\in S_0\cap p\tilde S$ be given.
We have to show that $x$ lies in $pS_0$.
Assume that $x\ne 0$ and choose an expression
$(\star)$ $x=\sum_{i=0}^sa_iE^{\prime n_i}/n_i!$ 
with $a_i\in\FS$ such that $n_0<\ldots<n_s$. 
We use induction on $n_s-n_0$. 

Suppose $E'$ divides $a_0$ in $\FS$.
Then $a_0E^{\prime n_0}/n_0!=a_0'E^{\prime n_0'}/(n_0')!$ 
with $n_0'=n_0+1$ and $a_0'=n_0'a_0/E'$.
If $s>0$ this allows to find a new expression of $x$ 
of the type $(\star)$ with smaller value of $n_s-n_0$,
and we are done by induction. 
If $s=0$ we replace the expression $(\star)$
by $x=a_0'E^{\prime n_0'}/n_0'!$; 
call this a modification of the first type.

Suppose $E'$ does not divide $a_0$ in $\FS$.
Lemma \ref{Le-graded-inj} implies that the image of $x$
in $gr^{n_0}_{E'}(\FS_\QQ)$ is non-zero.
In $S_{0,n_0+1}$ we have $x=py$.
Choose an expression $y=\sum_{i=\ell}^{n_0}c_iE^{\prime i}/i!$ 
with $c_i\in\FS$ such that $\ell$ is maximal.
Then $E'$ does not divide $c_\ell$ in $\FS$,
and Lemma \ref{Le-graded-inj} implies that $y$ has
non-zero image in $gr^\ell_{E'}(\FS_\QQ)$. 
Thus $\ell=n_0$. Using Lemma \ref{Le-graded-inj} again, 
it follows that the image of $a_0$ in $\FS/E'\FS$ is divisible
by $p$. Let $a_0=pb_0+b_1E'$ with $b_i\in\FS$ and let
$x'=x-pb_0E^{\prime n_0}/n_0!$. Then $x-x'\in pS_0$;
thus $x'\in S_0\cap p\tilde S$,
and we have to show that $x'\in pS_0$.
If $s>0$ we get an expression of $x'$ of the type $(\star)$
with smaller value of $n_s-n_0$, and we are done
by induction. If $s=0$ we replace $x$ by $x'$ and
take for $(\star)$
the expression $x'=a'_0E^{\prime n'_0}/n_0'!$
with $n_0'=n_0+1$ and $a_0'=n_0'b_1$; 
call this a modification of the second type.

If $s>0$ the inductive step is already finished.
So we may assume that $s=0$. 
We successively apply modifications of the first or 
second type depending on whether $E'$ divides $a_0$.
After at most $p$ steps, the new value of $a_0$ becomes
divisible by $p$, and thus $x$ lies in $pS_0$.
\end{proof}

Lemma \ref{Le-S0-tildeS} implies that $S_0/p^nS_0\to\tilde S/p^n\tilde S$ is injective,
so the projective limit $S\to\tilde S$ is injective, and thus $S\to\FS_\QQ$ is injective.
In order that this map commutes with $\sigma$
it suffices to show that $S\to\FS_\QQ/J^n\FS_\QQ$ commutes
with $\sigma$ for each $n$; 
this is true since $S_0\to\FS_\QQ/J^n\FS_\QQ$
commutes with $\sigma$, and the image of this map is $p$-adically complete.
Thus Proposition \ref{Pr-S} is proved.
\end{proof}

We turn to the frames associated to the rings $S$ and $S_{<m>}$.

\begin{Lemma}
\label{Le-SSS-m}
For $m\ge 1$ we have a sub-frame of $\SSS=(S,I,R,\sigma,\sigma_1)$,
$$
\SSS_{\!<m>}=(S_{<m>},I_{<m>},R,\sigma,\sigma_1).
$$
\end{Lemma}

\begin{proof}
Necessarily $I_{<m>}=I\cap S_{<m>}$.
We have to show that $\sigma:S\to S$ stabilises $S_{<m>}$
and that $\sigma_1=p^{-1}\sigma:I\to S$ maps
$I_{<m>}$ into $S_{<m>}$. We will show that
$\sigma(S)$ and $\sigma_1(I)$ are even contained in $S_{<1>}$.
Namely, we have $\sigma(E)=px$ with $x\in\FS[E^p/p]$.
Thus $\sigma_1(E^i/i!)=(p\cdot i!)^{-1}(px)^i$ lies in 
$\FS[E^p/p]$, using that $1+v_p(i!)\le i$ for $i\ge 1$.
Since $I/p^n I$ is the kernel of $S/p^nS\to R/p^n R$,
this ideal is generated as an $\FS$-module by 
the elements $E^i/i!$ for $i\ge 1$.
Thus the image of the map $I/p^{n+1}I\to S/p^nS$ 
induced by $\sigma_1$ lies in the image of $S_{<1>}$,
and it follows that $\sigma_1(I)\subseteq S_{<1>}$.
Since $S=\FS+I$ we get $\sigma(S)\subseteq S_{<1>}$.
\end{proof}

\begin{Prop}
\label{Pr-varkappa-S-1}
For $m\ge 1$ the inclusion $\SSS_{\!<m>}\to\SSS$ is crystalline.
\end{Prop}

\begin{proof}
This is a formal consequence of the relations
$\sigma(S)\subseteq S_{<m>}$ and
$\sigma_1(I)\subseteq S_{<m>}$
verified in the proof of Lemma \ref{Le-SSS-m}.

Indeed, let $\PPP=(P,Q,F,F_1)$ be an $\SSS$-window.
Choose a normal decomposition $P=L\oplus T$,
and let $\Psi:L\oplus T\to P$ be the $\sigma$-linear
isomorphism defined by $F_1$ on $L$ and by $F$ on $T$.
Then $P_{<m>}:=S_{<m>}\Psi(L\oplus T)$ is a free
$S_{<m>}$-module with $S\otimes_{S_{<m>}}P_{<m>}=P$.
Moreover $F_1(Q)\subseteq P_{<m>}$ and $F(P)\subseteq P_{<m>}$.
Let $Q_{<m>}=Q\cap P_{<m>}$. 
Then $P_{<m>}/Q_{<m>}=P/Q$ is a projective $R$-module.
Let $P_{<m>}=L_{<m>}\oplus T_{<m>}$ be a normal decomposition
and let $\Psi_{<m>}:L_{<m>}\oplus T_{<m>}\to P_{<m>}$
be the $\sigma$-linear map defined by $F_1$ on $L_{<m>}$
and by $F$ on $T_{<m>}$. 
In order that the quadruple $\PPP_{\!<m>}=(P_{<m>},Q_{<m>},F,F_1)$ 
is an $\SSS_{\!<m>}$-window with base change $\PPP$ we need that
the determinant of $\Psi_{<m>}$ is invertible. But the
determinant of $\Psi_{<m>}$ becomes invertible in $S$ because 
$\PPP$ is a window, and $S_{<m>}\to S$ is a local homomorphism. 
Thus the base change functor from $\SSS_{\!<m>}$-windows
to $\SSS$-windows is essentially surjective.

In order that the functor is fully faithful it suffices to
show that it induces a bijection $\End(\PPP_{\!<m>})\to\End(\PPP)$.
Clearly the map is injective. We have to show that every
$h\in\End(\PPP)$ stabilises $P_{<m>}$. But 
$h(F_1(Q))=F_1(h(Q))\subseteq F_1(Q)\subseteq P_{<m>}$,
and $F_1(Q)$ generates $P_{<m>}$ as an $S_{<m>}$-module.
This proves the proposition.
\end{proof}

\begin{Prop}
\label{Pr-varkappa-S-2}
For $m\ge 1$ the composition 
$\SSS_{\!<m>}\to\SSS\xrightarrow{\varkappa_S}\DDD_R^+$ 
is crystalline.
\end{Prop}

This is the main step in the proof of Theorem \ref{Th-varkappa-S}.
The proof of Proposition \ref{Pr-varkappa-S-2} is a variant
of the proof of Theorem \ref{Th-varkappa-crys}.

\begin{proof}
We choose $e$ such that $p\in\Fm_R^e\setminus\Fm_R^{e+1}$
and consider the index set $N=\{1,2,\ldots\}\cup\{e+\}$,
ordered by the natural order of $\ZZ$ and $e<e+<e+1$.
For $n\in N$ let $n+\in N$ be its successor.
Let $\Fm_R^{e+}=\Fm_R^{e+1}+pR$.
For $n\in N$ let $R_n=R/\Fm_R^{n}$. We equip the 
ideal $\Fm_R^n/\Fm_R^{n+}$ of $R_{n+}$ 
with the trivial divided powers if
$n\ne e+$ and with the canonical divided powers of $p$
if $n=e+$; these are again trivial if $p$ is odd.
In all cases the divided powers are compatible with
the canonical divided powers of $p$, and we obtain frames
$$
\DDD^+_{R_{n+}/R_n}=
(\WW^+(R_{n+}),\II^+_{R_{n+}/R_n},R_n,f,\tilde f_1).
$$

Let $T_n$ be the image of 
$S_{<m>}\xrightarrow{\varkappa_S}\WW^+(R)\to\WW^+(R_n)$. 
Since $\varkappa_S\sigma=f\varkappa_S$, 
the ring $T_n$ is stable under $f$.
Let $K_n$ be the kernel of $T_n\to R_n$ and let
$\tilde K_{n}$ be the kernel of $T_{n+}\to R_n$.

We claim that $\tilde f_1(\tilde K_n)\subseteq T_{n+}$.

To prove this, let $M_n$ be the kernel of $S_{<m>}\to R_n$, 
so $\tilde K_n$ is the image of $M_n\to\WW^+(R_{n+})$. 
Since $\varkappa_S\sigma=f\varkappa_S$
and since $f_1$ is $f$-linear it suffices
to show that a set of generators $x_i$ of the ideal $M_n$
with images $\varkappa_S(x_i)=\bar x_i\in\tilde K_n$ satisfies 
$f_1(\bar x_i)\in T_{n+}$. 
Since $\Fm_R=JR$, for $n\ne e+$ the ideal $M_n$ is generated 
by $I_{<m>}$ and $J^n$, while $M_{e+}$
is generated by $I_{<m>}$ and $J^{e+1}$ and $p$.
We check these generators case by case.

First, for $x\in I_{<m>}$ we have $f_1(\bar x)\in T_{n+}$
because $\SSS_{\!<m>}\to\DDD^+_{R_{n+}}$ is a frame homomorphism.

Assume that $n\ne e+$.
The homomorphism $\delta:J^n/J^{n+1}\to W(J^n/J^{n+1})$ 
is given by $\delta(x)=(x,\sigma_1(x),(\sigma_1)^2(x),\ldots)$.
Indeed, applying the Witt polynomial $w_n$ to this equation
gives $\sigma^n(x)=p^n(\sigma_1)^n(x)$, which is true.
Since the divided powers on $\Fm_R^{n}/\Fm_R^{n+}$ are trivial,
the endomorphism $\tilde f_1$ of $W(\Fm_R^{n}/\Fm_R^{n+})$ is
given by a shift to the left. Thus the map 
$\varkappa_S:J^n/J^{n+1}\to W(\Fm_R^{n}/\Fm_R^{n+})$
satisfies $\varkappa_S\sigma_1=\tilde f_1\varkappa_S$,
and we see that $f_1(\bar x)\in T_{n+}$ for
$x\in J^n$.

Assume now that $n=e+$. Since $J^{e+1}$ maps to
zero in $W(R_{e+1})$ it remains to show that 
$\tilde f_1(p)\in T_{n+}$. Now $\Log(p-v(1))=[p,0,0,\ldots]$;
cf.\ the proof of Lemma \ref{Le-WW-large-v}.
Thus $\tilde f_1(p)=f_1(v(1))=1$, and the claim is proved.

We obtain frames
$$
\TTT_n=(T_n,K_n,R_n,f,f_1), 
$$
$$
\TTT_{n+/n}=(T_{n+},\tilde K_n,R_n,f,\tilde f_1),
$$
and a commutative diagram of frames with strict homomorphisms:
$$
\xymatrix@M+0.2em{
\TTT_{n+} \ar[r]^-{\psi'} \ar[d]_{\iota_{n+}} &
\TTT_{n+/n} \ar[r]^-{\pi'} \ar[d] &
\TTT_n \ar[d]^{\iota_n} \\
\DDD^+_{R_{n+}} \ar[r]^-\psi &
\DDD^+_{R_{n+}/R_n} \ar[r]^-\pi &
\DDD^+_{R_n}
}
$$

Here $\pi$ is crystalline because the hypotheses of
Theorem \ref{Th-frame-crys} are satisfied; see the
proof of Corollary \ref{Co-DDD+-crys}. 
Since the vertical arrows are injective, it follows
that $\pi'$ satisfies the hypotheses of Theorem 
\ref{Th-frame-crys} as well, thus $\pi'$ is crystalline. 
Moreover lifts of windows under $\psi$ and under $\psi'$
correspond to lifts of the Hodge filtration from $R_n$
to $R_{n+}$ in the same way. 
Since $\iota_1$ is bijective, it follows that
$\iota_n$ is crystalline for each $n$.
Consider the limit frame
$$
\TTT=\varprojlim_n\TTT_n=(T,K,R,f,f_1).
$$
The inclusion $\iota:\TTT\to\DDD^+_R$
is the projective limit over $n$ of $\iota_n$ 
and thus crystalline.
Since $S_{<m>}$ is noetherian by Proposition \ref{Pr-S-m}, 
Lemma \ref{Le-lim-surjective} implies that 
$T=\varprojlim_nT_n$ is the image of 
$\varkappa_S:S_{<m>}\to\WW^+(R)$.
If $\varkappa_S$ is injective, we get
$\SSS_{\!<m>}=\TTT$, so $\SSS_{<m>}\to\DDD^+_R$
is crystalline as required.

Since we have not proved that $\varkappa_S$ is injective
we need an extra argument. Let $\Fa$ be the kernel of
$\varkappa_S:S_{<m>}\to\WW^+(R)$
and let $\Fa_n=\Fa\cap J^n\FS_\QQ$ for $n\ge 1$; 
here we use
that $S$ is a subring of $\FS_\QQ$ by Proposition \ref{Pr-S}.
We have $\Fa=\Fa_1$.
The ideals $\Fa_n$ of $S_{<m>}$ are stable under $\sigma$,
and they are also stable under $\sigma_1$ since
$S_{<m>}/\Fa$ and $\Fa_n/\Fa_{n+1}$ have no $p$-torsion.
Thus we can define frames 
$\SSS_{\!<m>,n}=(S_{<m>}/\Fa_n,I_{<m>}/\Fa_n,R,\sigma,\sigma_1)$.
We have $\SSS_{\!<m>,1}=\TTT$, 
and the projective limit over $n$ of $\SSS_{\!<m>,n}$
is isomorphic to $\SSS_{\!<m>}$ by Lemma \ref{Le-lim-surjective}
and Proposition \ref{Pr-S}.
The ideal $\Fa_n/\Fa_{n+1}$ is a finitely generated
$W(k)$-submodule of $(J^n/J^{n+1})\otimes\QQ$. 
Since the conditions of Proposition \ref{Pr-varkappa-image}
are satisfied, the endomorphism
$\sigma_1$ of $\Fa_n/\Fa_{n+1}$ is $p$-adically nilpotent.
Thus $\SSS_{\!<m>,n+1}\to\SSS_{\!<m>,n}$ is crystalline;
see the proof of \cite[Theorem 9.3]{Lau-Frames}.
It follows that $\SSS_{\!<m>}\to\TTT$ is crystalline, 
so $\SSS_{<m>}\to\DDD^+_R$ is crystalline too.
\end{proof}

Theorem \ref{Th-varkappa-S} follows from
Propositions \ref{Pr-varkappa-S-1} and \ref{Pr-varkappa-S-2}.
\qed

\begin{Remark}
Assume that $r=1$, i.e.\ $R$ is a discrete valuation ring. 
If $pR=\Fm_R^e$, the ring $S$ is the $p$-adic completion
of $W(k)[[t]][\{t^{em}/m!\}_{m\ge 1}]$. It is easy
to see that each quotient $S/p^nS$ is admissible,
so the $p$-adically complete ring $S$ is an admissible topological ring.
In particular, $\WW^+(S)$ is defined.
Since we assumed that the image of $\delta:\FS\to W(\FS)$
lies in $\WW(\FS)$, it follows that
the image of $\delta:S\to W(S)$ lies in $\WW^+(S)$,
using that $\WW^+(S)\to R$ is the projective limit of the divided
power extensions $\WW^+(S/p^nS)\to R/p^nR$ and that each
$\WW^+(S/p^nS)$ is $p$-adically complete.
If $p$ is odd this means that $\SSS$ is a Dieudonn\'e frame 
in the sense of \cite[Definition 3.1]{Zink-Windows}, and
Theorem \ref{Th-varkappa-S} becomes a special case of
\cite[Theorem 3.2]{Zink-Windows}. 
For $p=2$ the proof of loc.cit.\ works as well.
The starting point is the construction of an inverse
functor of $\varkappa_{S*}$; it maps a $\DDD_R^+$-window 
$\PPP$ to the value of its crystal $\DD^+(\PPP)_{S/R}$, 
equipped with an appropriate $\SSS$-window structure.

If $r\ge 2$, the ring $S$ is not admissible and thus
the crystal of a $\DDD_R^+$-window can not be evaluated at $S/R$. 
However, one can define by hand a sub-frame
$\DDD^+_{S/R}$ of $\WWW_{S/R}$ such that 
$\DDD^+_{S/R}\to\DDD^+_R$ is crystalline.
This allows to evaluate the crystal at $S/R$ and
to define an inverse functor of $\varkappa_{S*}$ as before.
The underlying ring of $\DDD^+_{S/R}$ is defined as follows.
Let $S_{m,n}$ be the image of $S_{<m>}\to S/p^nS$
and let $I_{m,n}$ be the kernel of $S_{m,n}\to R/p^nR$.
The divided Witt polynomials define an isomorphism
$\Log:W(I/p^nI)\cong (I/p^nI)^\NN$, and our ring is
$\varprojlim_n\varinjlim_m$ of the rings
$\WW^+(S_{0,n})+\Log^{-1}((I_{m,n})^{<\NN>})$.
The $\varprojlim_n$ of these rings for fixed $m\ge 1$
gives a frame $\DDD^+_{S_{<m>}/R}$ with a crystalline
homomorphism to $\DDD^+_R$.
This allows to construct the inverse functor from 
$\DDD^+_R$-windows to $\SSS_{\!<m>}$-windows.
We leave out the details.
\end{Remark}

%---------------------------------------------------------------

\section{Rigidity of $p$-divisible groups}
\label{Se-Rigidity}

In this section we record a rigidity property of the category of
$p$-divisible groups that will be used in Section \ref{Se-BT}.
As a preparation, for a local ring $R$ we consider the additive 
category $\FFF_R$ of commutative finite locally free $p$-group schemes over $R$. 
It is known that $\FFF_R$ is equivalent to the full subcategory 
of the bounded derived category of the exact 
category of $p$-divisible groups over $R$ formed by the complexes of 
length one which are isogenies; cf.\ \cite[(2.3.5)]{Kisin-crys}.
In elementary terms this equivalence can be expressed as follows.

\begin{Prop}
\label{Pr-finite-derived}
For a local ring $R$ 
let $\III_R$ be the category with isogenies of $p$-divisible
groups over $R$ as objects and
homomorphisms of complexes modulo homotopy as homomorphisms.
The set $S$ of quasi-isomor\-phisms in $\III_R$ allows a 
calculus of right fractions. In particular, 
the localised category $S^{-1}\III_R$ is additive.
It is equivalent to the additive category $\FFF_R$.
\end{Prop}

For completeness let us prove this directly.

\begin{proof}
Let $\tilde\III_R$ be the category with isogenies of $p$-divisible groups
over $R$ as objects and homomorphisms of complexes as homomorphisms.
We denote isogenies by $X=[X^0\to X^1]$.
Let $h^0(X)$ be the kernel of $X^0\to X^1$.
A homomorphism $f:X\to Y$ in $\tilde\III_R$ is homotopic to zero
if and only if $h^0(f)$ is the zero map; 
the homotopy is unique if it exists. We claim:

($\star$) For each homomorphism $f:X\to Y$ in $\tilde\III_R$ one can find
a quasi-isomorphism $t:Z\to X$ in $\tilde\III_R$ and a homomorphism
$g:Z\to Y$ in $\tilde\III_R$ which is an epimorphism in
both components such that $ft$ is homotopic to $g$.
Namely, embed $h^0(X)$ into $Z^0=X^0\oplus Y^0$ by $(1,f)$ 
and put $Z^1=Z^0/h^0(X)$. Define $t$ and $g$ by the projections 
$Z^0\to X^0$ and $Z^0\to Y^0$. There is a homotopy
between $ft$ and $g$ because $ft=g$ on $h^0(Z)$,
and ($\star$) is proved.

Next, for given homomorphisms $X\xrightarrow fY\xleftarrow sY'$ in $\tilde\III_R$, 
where $s$ is a quasi-iso\-mor\-phism, one can find an isogeny
$X'$ with a homomorphism $g:X'\to Y'$ and a
quasi-isomorphism $t:X'\to X$ such that $ft$
is homotopic to $sg$. 
Indeed, by ($\star$) we can assume that the components of $f$ 
are epimorphisms. Then take $X'=X\times_YY'$ componentwise.
It follows easily that $S$ allows a calculus of right fractions.
We have an additive  functor $h^0:S^{-1}\III\to\FFF$. 
It is surjective on isomorphism classes by a theorem of Raynaud
\cite[Th.~3.1.1]{BBM}. Let $X$ and $Y$ be isogenies.
The functor $h^0$ is full because for a given homomorphism
$f_0:h^0(X)\to h^0(Y)$, the construction in ($\star$) 
allows to represent $f_0$ as $h^0(gt^{-1})$. 
The functor is faithful because if a right fraction 
$gt^{-1}:X\to Y$ induces zero on $h^0$ then $g$ induces
zero on $h^0$ and thus $g$ is homotopic to zero.
\end{proof}

Let $\ArtC$ be the category of local Artin schemes
with perfect residue field of characteristic $p$,
and let $\pdivC\to\ArtC$ be the fibered category of 
$p$-divisible groups over schemes in $\ArtC$.

\begin{Lemma}
\label{Le-Aut-pdiv}
Assume that $u$ is an exact automorphism of the fibered 
category $\pdivC$ over $\ArtC$ such that for the group 
$E=\QQ_p/\ZZ_p$ over $\Spec\FF_p$ there is an isomorphism 
$u(E)\cong E$. Then $u$ is isomorphic to the identity functor. 
\end{Lemma}

\begin{proof}
For each $U$ in $\ArtC$ we are given a functor 
$G\mapsto G^u$ from the category of $p$-divisible groups
over $U$ to itself, which preserves short exact sequences, 
compatible with base change in $U$,
such that $\Hom(G,H)\cong\Hom(G^u,H^u)$.
We have to show that there is an natural isomorphism 
$G^u\cong G$ for all $G$, compatible with base change in $U$.
Let $\pfinC\to\ArtC$ be the fibered category of commutative 
finite locally free $p$-group schemes over schemes in $\ArtC$.
By Proposition \ref{Pr-finite-derived}, $u$ induces 
an automorphism of $\pfinC$ over $\ArtC$.
Let $H\in\pfinC$ over $U\in\ArtC$ be given. 
Assume that $p^n$ annihilates $H$ and $H^u$. For
each $T\to U$ in $\ArtC$ there is a natural isomorphism
$$
H(T)\cong\Hom_T(\ZZ/p^n\ZZ,H_T)\cong
\Hom_T(\ZZ/p^n\ZZ,H^u_T)\cong H^u(T),
$$
using that $(\ZZ/p^n\ZZ)^u\cong\ZZ/p^n\ZZ$. 
Since commutative finite locally free group schemes over $U$ form a full 
subcategory of the category of abelian presheaves on 
$\ArtC/U$, we get a natural isomorphism $H\cong H^u$,
which induces a natural isomorphism $G\cong G^u$ for
all $p$-divisible groups $G$ over $U$.
\end{proof}

%---------------------------------------------------------------

\section{The reverse functor}
\label{Se-BT}

We fix an admissible ring $R$ which is local of dimension zero,
thus $k=R_{\red}$ is a perfect field of characteristic $p$.
In this case one can write down an inverse of the functor $\Phi_R$ as follows. 
The construction appears in \cite{Lau-Dual} when $p\ge 3$ or $pR=0$ 
and extends to the general case with appropriate changes.

\begin{Defn}
Let $\JJJ_R$ be the category of $R$-algebras $A$ such that 
$\NNN_A$ is bounded nilpotent and $A_{\red}$ is the union 
of finite dimensional $k$-algebras.
\end{Defn}

We call a ring homomorphism $A\to B$ ind-\'etale if
$B$ is the filtered direct limit of \'etale $A$-algebras.

\begin{Lemma}
\label{Le-JJJR}
Every $A\in\JJJ_R$ is admissible. 
The category $\JJJ_R$ is stable under tensor products. 
If $A\to B$ is an ind-\'etale or a quasi-finite ring 
homomorphism with $A\in\JJJ_R$ then $B\in\JJJ_R$. 
\end{Lemma}

\begin{proof}
Since a reduced finite $k$-algebra is \'etale 
and thus perfect, every $A$ in $\JJJ_R$ is admissible. 
Let $A\to B$ a ring homomorphism with $A\in\JJJ_R$. 
Then $\NNN_AB$ is bounded nilpotent, so $B$ is lies in 
$\JJJ_R$ if and only if $B/\NNN_AB$ lies in $\JJJ_R$.
For given homomorphisms $B\leftarrow A\to C$ in $\JJJ_R$
we have to show that $B\otimes_AC$ lies in $\JJJ_R$.
By the preceding remark we may assume that $A,B,C$
are reduced. Then $B\otimes_AC$ is the direct limit
of \'etale $k$-algebras and thus lies in $\JJJ_R$.
Let $g:A\to B$ be an ind-\'etale or quasi-finite
ring homomorphism with $A\in\JJJ_R$. In order to show
that $B\in\JJJ_R$ we may assume that $A$ is reduced.
Then every finitely generated $k$-subalgebra of $A$ is 
\'etale. Thus each \'etale $A$-algebra is defined 
over an \'etale $k$-subalgebra of $A$. If $g$ is ind-\'etale
it follows that $B$ lies in $\JJJ_R$. 
Assume that $g$ is quasi-finite. Then $B$ 
is defined over an \'etale $k$-subalgebra of $A$. 
Since all finite $k$-algebras lie in $\JJJ_R$ and since
$\JJJ_R$ is stable under tensor products it follows that
$B\in\JJJ_R$. 
\end{proof}

Let $S=\Spec R$ and let $\JJJ_S$ be the category of
affine $S$-schemes $\Spec A$ with $A\in\JJJ_R$.
If $\tau$ is either ind-\'etale or fpqc, we consider 
the $\tau$-topology on $\JJJ_S$ in which coverings 
are finite families of morphisms $(\Spec B_i\to\Spec A)$ 
such that the associated homomorphism $A\to\prod_i B_i$ is 
faithfully flat, and ind-\'etale if $\tau$ is ind-\'etale.
Let $\tilde S_{\JJJ,\tau}$ be the category of $\tau$-sheaves on $\JJJ_S$.

\begin{Lemma}
\label{Le-WW-sheaf}
The presheaf of rings $\WW$ on $\JJJ_S$ is an fpqc sheaf.
\end{Lemma}

\begin{proof}
See \cite[Lemma 1.5]{Lau-Dual}.
Since the presheaf $W$ is an fpqc sheaf, it suffices to
show that for an injective ring homomorphism $A\to B$ in
$\JJJ_R$ we have $\WW(A)=\WW(B)\cap W(A)$ in $W(B)$.
This is easily verified using that $A_{\red}\to B_{\red}$ 
is injective and $\hat W(\NNN_A)=\hat W(\NNN_B)\cap W(A)$ 
in $W(\NNN_B)$.
\end{proof}

%\subsection{The functor \boldmath $\BT$}

Let $\PPP$ be a Dieudonn\'e display over $R$.
For $\Spec A$ in $\JJJ_S$ let $\PPP_A=(P_A,Q_A,F,F_1)$ 
be the base change of $\PPP$ to $A$. 
We define two complexes $Z(\PPP)$ and $Z'(\PPP)$ 
of presheaves of abelian groups on $\JJJ_S$ by
\begin{gather}
\label{Eq-ZPPP}
Z(\PPP)(\Spec A)=[Q_A\xrightarrow{F_1-1}P_A],\\
\label{Eq-ZprimePPP}
Z'(\PPP)(\Spec A)=[Q_A\xrightarrow{F_1-1}P_A]\otimes[\ZZ\to\ZZ[1/p]],
\end{gather}
such that $Z(\PPP)$ lies in degrees $0,1$ and
$Z'(\PPP)$ lies in degrees $-1,0,1$, so the second
tensor factor lies in degrees $-1,0$.

\begin{Prop}
\label{Pr-BT(PPP)}
The components of $Z'(\PPP)$ are fpqc sheaves on $\JJJ_S$.
The ind-\'etale (and thus the fpqc) cohomology 
sheaves of $Z'(\PPP)$ vanish outside degree zero,
and the cohomology sheaf in degree zero is represented
by a well-defined $p$-divisible group $\BT_R(\PPP)$ over $R$. 
This defines an additive and exact functor
$$
\BT_R:(\text{Dieudonn\'e displays over $R$})
\to(\text{$p$-divisible groups over $R$}).
$$
\end{Prop}

\noindent
One can also express the definition of the functor 
$\BT_R$ by the formula
$$
\BT_R(\PPP)=[Q\xrightarrow{F_1-1} P]\otimes^L\QQ_p/\ZZ_p
$$
in the derived category of either ind-\'etale or fpqc 
abelian sheaves on $\JJJ_S$.

\begin{proof} %[Proof of Proposition \ref{Pr-BT(PPP)}]
This is essentially proved in \cite{Lau-Dual},
but we recall the arguments for completeness and
because there is a small modification when $p=2$. 
To begin with, $p$-divisible groups over $R$ form
a full subcategory of the abelian presheaves on $\JJJ_S$ 
because finite group schemes over $R$
lie in $\JJJ_S$; see Lemma \ref{Le-JJJR}.
Hence $\BT_R$ is a well-defined additive functor if the 
assertions on the cohomology of $Z'(\PPP)$ hold.
Since an exact sequence of Dieudonn\'e displays over $R$
induces an exact sequence of the associated complexes
of presheaves $Z'$, the functor $\BT_R$ is exact
if it is defined.

The components of $Z(\PPP)$ and $Z'(\PPP)$ are
fpqc sheaves on $\JJJ_S$ by Lemma \ref{Le-WW-sheaf}.
These complexes carry two filtrations. 
First, a Dieudonn\'e display is called \'etale if
$Q=P$, and nilpotent if $V^\sharp$
is topologically nilpotent. Every Dieudonn\'e
display over $R$ is naturally an extension of an
\'etale by a nilpotent Dieudonn\'e display,
which induces exact sequences of the associated
complexes $Z(\ldots)$ and $Z'(\ldots)$. Thus  
we may assume that $\PPP$ is \'etale or nilpotent. 
Second, for every $\PPP$
we have an exact sequence of complexes of presheaves
$$
0\to\hat Z(\PPP)\to Z(\PPP)\to Z_{\red}(\PPP)\to 0,
$$
defined by $Z_{\red}(\PPP)(\Spec A)=Z(\PPP)(\Spec A_{\red})$.
The same holds for $Z'$ instead of $Z$.
We write $\hat Z(\PPP)=[\hat Q\to\hat P]$.

Assume that $\PPP$ is \'etale. Then $F_1:P\to P$ is an $f$-linear
isomorphism. Thus $F_1:\hat P\to\hat P$ is elementwise nilpotent,
and the complex $\hat Z(\PPP)$ is acyclic.
It follows that $Z(\PPP)$ is quasi-isomorphic to the complex 
$Z_{\red}(\PPP)=Z_{\red}$,
which is the projective limit of the complexes 
$Z_{\red,n}=Z_{\red}/p^nZ_{\red}$.
In the \'etale topology, each $Z_{\red,n}$ is a 
surjective homomorphism of
sheaves whose kernel is a locally constant sheaf 
$G_n$ of free $\ZZ/p^n\ZZ$-modules of rank equal to the
rank of $P$. The system $(G_n)_n$ defines an \'etale
$p$-divisible group $G$ over $R$, and $Z_{\red}$
is quasi-isomorphic to $T_pG=\varprojlim G_n$ as ind-\'etale
sheaves. It follows that
$Z'(\PPP)\simeq Z'_{\red}(\PPP)$ is quasi-isomorphic to 
the complex $[T_pG\to T_pG\otimes\ZZ[1/p]]$ in degrees $-1,0$,
which is quasi-isomorphic to $G$ in degree zero (as ind-\'etale sheaves).

Assume that $\PPP$ is nilpotent. Then the complex 
$Z_{\red}(\PPP)$ is acyclic because its value over 
$\Spec A$ is isomorphic to
$[1-V:P_{A_{\red}}\to P_{A_{\red}}]$ where $V$ is a
topologically nilpotent $f^{-1}$-linear homomorphism. 
Thus $Z(\PPP)$ is quasi-isomorphic to $\hat Z(P)$. 
To $\PPP$ we associate a nilpotent display by the
$u_0$-homomorphism of frames $\DDD_R\to\WWW_R$.
By \cite[Theorem 81 and Corollary 89]{Zink-Disp} 
there is a formal $p$-divisible group $G$ over $R$
associated to this display such that for each 
$A\in\JJJ_R$ there is an exact sequence
$$
0\to\hat Q(A)\xrightarrow{u_0F_1-1}\hat P(A)\to G(A)\to 0;
$$
this is the direct limit of the corresponding sequences for
the finitely generated (nilpotent) subalgebras of $\NNN_A$.
Since $u_0\in W(\ZZ_p)$ maps to $1$ in $W(\FF_p)$, 
there is a unique $c\in W(\ZZ_p)$ which maps to $1$ 
in $W(\FF_p)$ such that $u_0=cf(c^{-1})$, 
namely $c=u_0f(u_0)f^2(u_0)\cdots$. 
Multiplication by $c$ in both components defines an isomorphism 
of complexes
$$
[\hat Q(A)\xrightarrow{F_1-1}\hat P(A)]
\cong[\hat Q(A)\xrightarrow{u_0F_1-1}\hat P(A)]
$$
It follows that $Z'(\PPP)\simeq\hat Z'(\PPP)$ 
is quasi-isomorphic to $G$ in degree zero. 
\end{proof} 

\begin{Remark}
Recall that $\DDD_R=(\WW(R),\II_R,f,\ff_1)$ is viewed
as a Dieudonn\'e display over $R$.
We have $\BT_R(\DDD_R)=\mu_{p^{\infty}}$ by \cite[(211)]{Zink-Disp}.
\end{Remark}

\begin{Lemma}
\label{Le-BT-change}
Let $R\to R'$ be a homomorphism of admissible rings which are local of dimension zero.
For each Dieudonn\'e display $\PPP$ over $R$ there is a natural isomorphism 
$$
\BT_R(\PPP)_{R'}\cong\BT_{R'}(\PPP_{R'}).
$$
\end{Lemma}

\begin{proof}
If the residue field of $R'$ is an algebraic extension of $k$,
every ring in $\JJJ_{R'}$ lies in $\JJJ_R$, and the assertion
follows directly from the construction of $\BT_R$. In general,
let $\EEE_R$ be the category of all $R$-algebras which are
admissible rings,
and let $\EEE_S$ be the category of affine $S$-schemes $\Spec A$
with $A\in\EEE_R$, endowed with the ind-\'etale topology. 
The complexes of presheaves $Z(\PPP)$ and $Z'(\PPP)$
on $\JJJ_S$ defined in \eqref{Eq-ZPPP} and \eqref{Eq-ZprimePPP}
extend to complexes of presheaves on $\EEE_S$ defined by the
same formulas. The proof of Lemma \ref{Le-JJJR} shows that
for an ind-\'etale ring homomorphism $A\to B$ with $A\in\EEE_R$
we have $B\in\EEE_R$ as well.
Using this, the proof of Proposition \ref{Pr-BT(PPP)}
shows that the ind-\'etale cohomology sheaves of $Z'(\PPP)$ on
$\EEE_S$ vanish outside degree zero, and $H^0(Z'(\PPP))$
is naturally isomorphic to $\BT_R(\PPP)$ as a sheaf on
$\EEE_S$. Since every ring in $\EEE_{R'}$ lies in $\EEE_R$,
the lemma follows as in the first case.
\end{proof}

%\subsection{Equivalence}
%\label{Subse-Equivalence}

\begin{Prop}
\label{Pr-BT-equiv}
The functor $\BT_R$ is an equivalence of exact categories
which is a quasi-inverse of the functor $\Phi_R$.
\end{Prop}

%Recall that $\Phi_R$ is an equivalence by Theorem \ref{Th-Phi-equiv}. If $p=2$ and $R$ is not annihilated by $p$, the relation between $\BT_R$ and the functor $\Phi^+_R$ defined in Proposition \ref{Pr-bt2disp+} is explained in Corollary \ref{Co-BT-equiv} below.

\begin{proof}
By \ref{Subse-finiteness} we may assume 
that $R$ is a local Artin ring.
Since $p$-divisible groups and Dieudonn\'e displays over
$k$ have universal deformation rings which are power series
rings over $W(k)$, once the functor $\BT_R$ is defined,
in order to show that it is an equivalence of categories
it suffices to consider the cases $R=k$ and $R=k[\varepsilon]$.
In particular, if $p=2$, we may assume that $pR=0$, 
so that the
results of \cite{Zink-DDisp} and \cite{Lau-Dual} can
be applied. The category $\CCC_R$ used in \cite{Lau-Dual}
is the category of all $A\in\JJJ_R$ such that $\NNN_A$
is nilpotent. Since this subcategory is stable under
ind-\'etale extensions, it does not make a difference
whether the functor $\BT_R$ is defined in terms of
$\CCC_R$ or $\JJJ_R$. Thus $\BT_R$ is an equivalence
by \cite[Theorem 1.7]{Lau-Dual}, 
which relies on the equivalence proved in \cite{Zink-DDisp}.
It is easily verified that $\BT_R(\Phi_R(\QQ_p/\ZZ_p))$
is isomorphic to $\QQ_p/\ZZ_p$. Thus $\BT_R$ is a quasi-inverse
of $\Phi_R$ by Lemmas \ref{Le-Aut-pdiv} and \ref{Le-BT-change}. 
It is easily verified that the functors $\BT_R$ and $\Phi_R$
preserve short exact sequences.
\end{proof}

\appendix

\section{PD envelopes of regular immersions}

This section provides a reference for the flatness of the divided power envelope
of a regular immersion, which is used in the proof of Lemma \ref{Le-pd-alg}.
Let us recall regular immersions following \cite[VII]{SGA6}.
For a ring $A$, a projective $A$-Module $M$ of finite type, and a linear map $f:M\to A$ 
one defines the Koszul complex 
\[
K_*(A,f)=[\ldots\to\Lambda^2M\to\Lambda^1M\to A]
\]
with differential given by $x_1\wedge\ldots\wedge x_n\mapsto\sum(-1)^{i+1}f(x_i)x_1\wedge\ldots\widehat x_i\ldots\wedge x_n$. Let $I=f(M)\subseteq A$.
One calls $f$ regular if the augmentation $K_*(A,f)\to A/I$ is a quasi-isomorphism.
If $x_1,\ldots,x_r$ is a regular sequence in $A$ and $f:A^r\to A$ is given by $f(a)=\sum a_ix_i$,
then $f$ is regular in the previous sense. For a ring homomorphism $A\to A'$ let $f':M'\to A'$
be the scalar extension of $f$, and let $I'=f'(M')$. If both $f$ and $f'$ are regular, then
$\Tor_i^A(A/I,A')=0$ for $i\ge 1$ and thus $I'=I\otimes_AA'$. 
A closed immersion of schemes $Y\to X$ is called
regular if locally in $X$ it takes the form $\Spec A/I\to\Spec A$ where $I=f(M)$ for
a regular homomorphism $f:M\to A$.

\begin{Prop}
\label{Pr:pd-flat}
Let $S$ be a scheme and\/ $i:Y\to X$ a regular closed immersion of flat $S$-schemes.
Then the divided power envelope $\DDD_X(Y)$ is flat over $S$.
\end{Prop}

Under additional hypotheses this is proved in \cite[Le.~2.3.3]{BBM}.
We use the following description of the divided power polynomial algebra.

\begin{Lemma}
\label{Le:pd-algebra}
For a ring $R$ let $A_0=R[T_1,\ldots,T_n]$ and let $B_0=R\left<T_1,\ldots,T_n\right>$
be the divided power envelope of $I_0=(T_1,\ldots,T_n)\subseteq A_0$. 
Then one can write $B_0=\varinjlim_rM_{0,r}$ as an $A_0$-module, the direct limit
taken over $r\in\NN$ ordered multiplicatively, such that there are exact sequences
of $A_0$-modules
\[
0\to J_{0,r}\to M_{0,r}\to N_{0,r}\to 0
\]
with $J_{0,r}=(T_1^r,\ldots,T_n^r)$ and where $N_{0,r}$ has a finite filtration
with quotients isomorphic to $A_0/I_0=R$. 
\end{Lemma}

\begin{proof}
The assertion is stable under base change in $R$, so we may take $R=\ZZ$. 
Then $B_0$ is the $A_0$-subalgebra of $A_0\otimes\QQ$ generated by all $T_i^m/m!$. 
Let $M_{0,r}=B_0\cap r^{-1}A_0$ inside $A_0\otimes\QQ$. 
Then $r^{-1}J_{0,r}$ is contained in $M_{0,r}$, and the quotient $N_{0,r}$ 
coincides with the image of $M_{0,r}$ in $(A_0/J_{0,r})\otimes\QQ$. 
Any maximal filtration of the latter by monomial ideals gives the 
required filtration of $N_{0,r}$.
\end{proof}

\begin{proof}[Proof of Proposition \ref{Pr:pd-flat}]
We may assume that $S=\Spec R$ and $X=\Spec A$ and $Y=\Spec A/I$ where $I$ is the image of a regular map $f:A^r\to A$. We have $f(a)=\sum a_ix_i$ for a sequence $x_1,\ldots,x_r$ in $A$.
Let $A_0=\ZZ[T_1,\ldots,T_n]$ and $M_0=A_0^n$ with $f_0:M_0\to A_0$ given
by $a\mapsto\sum a_iT_i$. Let $I_0=f_0(M_0)$.
We consider the homomorphism $A_0\to A$ defined by $T_i\mapsto x_i$.
Let $B=\DDD_A(I)$ and $B_0=\DDD_{A_0}(I_0)$ be the divided power envelopes.
Since $f$ and $f_0$ are regular we have $I=I_0\otimes_AA_0$.
As in \cite[(3.4.8)]{Berthelot:CohCristalline} it follows that $B=B_0\otimes_{A_0}A$.
Using Lemma \ref{Le:pd-algebra} we get $B=\varinjlim M_r$ with $M_r=M_{0,r}\otimes_{A_0}A$.
Moreover, since $\Tor_1^{A_0}(A_0/I_0,A)=0$, we obtain exact sequences of $A$-modules
\[
0\to J_r\to M_r\to N_r\to 0
\]
with $J_r=J_{0,r}\otimes_{A_0}A$ and $N_r=N_{0,r}\otimes_{A_0}A$,
and we obtain filtrations of $N_r$ with quotients isomorphic to $A/I$.
Similarly there are exact sequences of $A$-modules
\[
0\to J_r\to A\to A/J_r\to 0
\]
and filtrations of $A/J_r$ with quotients isomorphic to $A/I$.
Since $A$ and $A/I$ are flat over $R$, it follows that $J_r$
and $M_r$ and $B$ are flat over $R$.
\end{proof}

We will use the following example of regular immersions.

\begin{Lemma}
\label{Le:R[[T]]-regular}
For a ring $R$ and a projective $R$-module $T$ of finite type we consider the complete symmetric algebra $A=R[[T]]=\prod_{n\ge 0}\Sym^n_R(T)$ and $M=T\otimes_RA$. Then the homomorphism $f:M\to A$ given by $t\otimes a\mapsto ta$ is regular.
\end{Lemma}

\begin{proof}
The complex $K_*(M,f)$ is the direct product over $m\ge 0$ of complexes $K_*^{(m)}$ 
with $K_n^{(m)}=\Lambda^nT\otimes_R\Sym^{m-n}(T)$, using the
convention $\Sym^r(T)=0$ for $r<0$. Since the complexes $K_*^{(m)}$ are
compatible with base change in $R$, the general case can be reduced to the case where $T$ is free.
Then an $R$-basis of $T$ is a regular sequence in $A$, and the assertion follows. 
\end{proof}

%---------------------------------------------------------------

\end{document}